\documentclass{amsart}
\usepackage{amsmath}
\usepackage{amssymb}
\usepackage{amscd}
\usepackage{amsfonts}
\usepackage{latexsym}
\usepackage{amsfonts}
\usepackage{graphicx}
\usepackage[all]{xy}
\usepackage{mathrsfs}
\usepackage{amscd}
\usepackage{verbatim}
\usepackage{url}

\usepackage[usenames]{color}

\numberwithin{equation}{section}

\newcommand{\fg}{{\mathfrak{g}}}
\newcommand{\fn}{{\mathfrak{n}}}

\newcommand{\sG}{{\mathscr G}}
\newcommand{\E}{\mathcal{E}}

\newcommand{\cA}{\mathcal{A}}
\newcommand{\cB}{\mathcal{B}}

\newcommand{\cT}{\mathcal{T}}

\newcommand{\cW}{\mathcal{W}}
\newcommand{\cZ}{\mathcal{Z}}

\newcommand{\dbF}{{\mathbb{F}}}

\newcommand{\dbQ}{{\mathbb{Q}}}

\newcommand{\dbZ}{{\mathbb{Z}}}

\newcommand{\BC}{{\mathbb{C}}}

\newcommand{\BG}{{\mathbb{G}}}
\newcommand{\BF}{{\mathbb{F}}}

\newcommand{\BK}{{\mathbb{K}}}
\newcommand{\BN}{{\mathbb{N}}}

\newcommand{\BQ}{{\mathbb{Q}}}
\newcommand{\BR}{{\mathbb{R}}}
\newcommand{\BZ}{{\mathbb{Z}}}

\newcommand{\fp}{{\mathfrak p}}

\newcommand{\tX}{{\tilde X}}

\newcommand{\BQbar}{\overline{\dbQ}}
\newcommand{\BFbar}{\overline{\dbF}}
\newcommand{\BZbar}{\overline{\dbZ}}

\newcommand{\iso}{\buildrel{\sim}\over{\longrightarrow}}
\newcommand{\epi}{\twoheadrightarrow}
\newcommand{\mono}{\hookrightarrow}

\DeclareMathOperator{\cont}{{cont}}

\DeclareMathOperator{\Ss}{{ss}}
\DeclareMathOperator{\Aff}{{Aff}}
\DeclareMathOperator{\Br}{{Br}}

\DeclareMathOperator{\coarse}{{coarse}}

\DeclareMathOperator{\Dyn}{{Dyn}}
\DeclareMathOperator{\Epi}{{Epi}}
\DeclareMathOperator{\Ex}{{Ex}}

\DeclareMathOperator{\Irr}{{Irr}}
\DeclareMathOperator{\Tr}{{Tr}}

\DeclareMathOperator{\Spec}{{Spec}}
\DeclareMathOperator{\Gal}{{Gal}}

\DeclareMathOperator{\red}{{red}}
\DeclareMathOperator{\et}{{et}}

\DeclareMathOperator{\tors}{{tors}}

\DeclareMathOperator{\mot}{{mot}}

\DeclareMathOperator{\Hom}{{Hom}}

\DeclareMathOperator{\FIsoc}{\mathit{F}-Isoc}
\DeclareMathOperator{\FIsocd}{\mathit{F}-Isoc^{\dagger}}

\DeclareMathOperator{\Inj}{{Inj}}
\DeclareMathOperator{\Lie}{{Lie}}
\DeclareMathOperator{\Norm}{{Norm}}
\DeclareMathOperator{\NORM}{{\bf Norm}}
\DeclareMathOperator{\Open}{{Open}}
\DeclareMathOperator{\Par}{{Par}}
\DeclareMathOperator{\Rep}{{Rep}}
\DeclareMathOperator{\Stab}{{Stab}}
\DeclareMathOperator{\Syl}{{Syl}}

\DeclareMathOperator{\Aut}{{Aut}}
\DeclareMathOperator{\id}{{id}}
\DeclareMathOperator{\Id}{{Id}}
\DeclareMathOperator{\Coker}{{Coker}}
\DeclareMathOperator{\Ker}{{Ker}}
\DeclareMathOperator{\Fr}{{Fr}}
\DeclareMathOperator{\im}{{Im}}

\DeclareMathOperator{\pol}{{pol}}
\DeclareMathOperator{\Res}{{Res}}
\DeclareMathOperator{\Tann}{{Tann}}
\DeclareMathOperator{\Val}{{Val}}
\DeclareMathOperator{\Vect}{{Vec}}

\DeclareMathOperator{\Pross}{{Pro-ss}}
\DeclareMathOperator{\PROSS}{{\bf Pro-ss}}
\DeclareMathOperator{\Prored}{{Pro-red}}

\newcommand{\sfC}{{\mathsf{C}}}

\newtheorem{cor}[subsubsection]{Corollary}
\newtheorem{lem}[subsubsection]{Lemma}
\newtheorem{prop}[subsubsection]{Proposition}

\newtheorem{conj}[subsubsection]{Conjecture}
\newtheorem{theorem}[subsubsection]{Theorem}

\newtheorem{quest}[subsubsection]{Question}

\theoremstyle{remark}
\newtheorem{rem}[subsubsection]{Remark}

\theoremstyle{definition}
\newtheorem{ex}[subsubsection]{Example}

\newcommand{\SupSet}{\raise1.75pt
     \hbox{$\,\,\scriptstyle\supset\,\,$}}
\newcommand{\SubSet}{\raise1.75pt
     \hbox{$\,\,\scriptstyle\subset\,\,$}}

\catcode`\@=11
\renewcommand{\over}{\@@over}
\renewcommand{\atop}{\@@atop}
\renewcommand{\above}{\@@above}
\renewcommand{\overwithdelims}{\@@overwithdelims}
\renewcommand{\atopwithdelims}{\@@atopwithdelims}
\renewcommand{\abovewithdelims}{\@@abovewithdelims}
\catcode`\@=13

\begin{document} 

\title[The pro-semisimple completion of $\pi_1 (X)$, where $X$ is a smooth $\BF_q$-variety]{On the pro-semisimple completion of the fundamental group of a smooth variety over a finite field}
\author{Vladimir Drinfeld}
\address{University of Chicago, Department of Mathematics, Chicago, IL 60637}
\email{drinfeld@math.uchicago.edu}

\dedicatory{To Dima Kazhdan with gratitude and admiration}

\begin{abstract}
Let $\Pi$ be the fundamental group of a smooth variety $X$ over $\BF_p$. 
 Let $\BQbar$ be an algebraic closure of $\BQ$. Given a non-Archimedean place
$\lambda$ of $\BQbar$ prime to $p$, consider the $\lambda$-adic pro-semisimple completion of $\Pi$ as an object of the groupoid whose objects are pro-semisimple groups and whose morphisms are isomorphisms up to conjugation by elements of the neutral connected component.
We prove that this  object does not depend on $\lambda$. If $\dim X=1$ we also prove a crystalline generalization of this fact.

We deduce this from the Langlands conjecture for function fields (proved by L. Lafforgue) and its crystalline analog (proved by T. Abe) using a reconstruction theorem in the spirit of Kazhdan-Larsen-Varshavsky.

We also formulate two related Conjectures \ref{conj:reciprocity} and \ref{conj:2reciprocity}. Each of them is a kind of ``reciprocity law" involving a sum over all $\ell$-adic cohomology theories (including the crystalline theory for $\ell=p$).
\end{abstract}

\thanks{Partially supported by NSF grant DMS-1303100}

\maketitle

\tableofcontents

\section{Introduction}
\subsection{Some notation}   \label{ss:some}
\subsubsection{The setting}   \label{sss:some}
Once and for all, we fix a prime $p$.

Let $X$ be an irreducible normal variety over $\BF_p \,$.
We fix a universal cover $\tX\to X$ and set $\Pi:=\Aut (\tilde X/X)$. If one also chooses a geometric point $\xi$ of $X$ and a lift of $\xi$ to $\tilde X$ then $\Pi$ identifies with $\pi_1 (X,\xi )$.

Let $|X|$ (resp.~$|\tX|$) denote the set of closed points of $X$ (resp.~of $\tX$). We have a canonical $\Pi$-equivariant map
\begin{equation} \label{e:Frob1}
|\tX|\to\Pi,\quad \tilde x\mapsto F_{\tilde x} \, ,
\end{equation}
where $F_{\tilde x}$ is the geometric Frobenius, i.e., the unique automorphism of $\tX$ over $X$ whose restriction to $\{\tilde x\}$ equals the composition $\{\tilde x\}\overset{\varphi}\longrightarrow\{\tilde x\}\mono\tX$, where $\varphi :\{\tilde x\}\to\{\tilde x\}$ is the Frobenius morphism with respect to $x$ (this means that for any regular function $f$ on $\{\tilde x\}$ one has $\varphi^*(f)=f^{q_x}$, where $q_x$ is the order of the residue field of $x$).

\subsubsection{The subset $\Pi_{\Fr}\subset\Pi$}  \label{sss:Pi_Fr}
Let $\Pi_{\Fr}\subset\Pi$ be the subset formed by elements $F_{\tilde x}^n$, where $\tilde x$ runs through $|\tilde X|$ and $n$ runs through $\BN$. The subset $\Pi_{\Fr}\subset\Pi$ is dense (by \v {C}ebotarev's theorem). The group $\Pi$ acts on $\Pi_{\Fr}$ by conjugation.

The subset $\Pi_{\Fr}\subset\Pi$ behaves functorially in the following sense. Suppose we have another pair $(X',\tilde X')$ as above and a morphism $(X',\tilde X')\to (X,\tilde X)$. Set $\Pi':=\Aut (\tilde X'/X')$, then there is a unique homomorphism $f:\Pi'\to\Pi$ that makes the map $\tilde X'\to\tilde X$ equivariant with respect to $\Pi'$. It is easy to see that $f(\Pi'_{\Fr})\subset\Pi_{\Fr}\,$. Moreover, if the morphism $X'\to X$ is finite and etale then $f(\Pi'_{\Fr})=f(\Pi )\cap\Pi_{\Fr}\,$.

\subsection{The $\lambda$-adic pro-semisimple completion of $\Pi$ and its ``coarse" version $\hat\Pi_{(\lambda)}\,$}
\subsubsection{The $\ell$-adic pro-semisimple completion of $\Pi$}    \label{sss:thecompletion}
Fix a prime $\ell\ne p$. Consider the category opposite to the category of pairs $(G,\rho )$, where $G$ is a (not necessarily connected) semisimple algebraic group over $\BQ_\ell$ and $\rho :\Pi\to G(\BQ_\ell )$ is a continuous homomorphism whose image is Zariski-dense in $G$. This category is equivalent to a poset $I$. We have a projective system of semisimple algebraic groups over $\BQ_{\ell}$ indexed by $I$. Its projective limit is called the \emph{$\ell$-adic pro-semisimple completion} of $\Pi$. We will denote it by $\hat\Pi_\ell\,$. It is clear that $\hat\Pi_\ell$ is a pro-semisimple group scheme, and one has a canonical exact sequence of group schemes
\begin{equation}   \label{e:exact}
0\to \hat\Pi_\ell^\circ\to \hat\Pi_\ell\to\Pi\to 0,
\end{equation}
where $\hat\Pi_\ell^\circ$ is the neutral connected component of $\hat\Pi_\ell \,$.  By the definition of 
 $\hat\Pi_\ell\,$, one has a canonical map $\Pi\to\hat\Pi_\ell (\BQ_{\ell} )$ such that the composition 
 $\Pi\to\hat\Pi_\ell (\BQ_{\ell})\to\Pi$ is equal to the identity.

\subsubsection{The scheme $[\hat\Pi_\ell ]$}   \label{sss:charFr}
For any pro-reductive group $G$, let $[G]$ denote the GIT quotient of $G$ by the conjugation action of the neutral connected component $G^\circ$ (i.e., $[G]$ is the spectrum of the algebra of those regular functions on $G$ that are invariant under $G^\circ$-conjugation). We have the projection $G\to [G]$ and the canonical map $[G]\epi\pi_0 (G):=G/G^\circ$.

In particular, we have the scheme $[\hat\Pi_\ell ]$ and the canonical map $[\hat\Pi_\ell ]\epi \pi_0 (\hat\Pi_\ell )=\Pi$.

\subsubsection{The groupoids $\Prored (E)$ and $\Pross (E)$}   \label{sss:coarsecat}
For any algebraically closed field $E$, let $\Prored (E)$ denote the groupoid whose objects are pro-reductive groups over $E$ and whose morphisms are as follows: a morphism $G_1\to G_2$ is an isomorphism of group schemes $G_1\to G_2$ defined up to composing with automorphisms of $G_2$ of the form $x\mapsto gxg^{-1}$, $g\in G^\circ_2\,$. 
Let $\Pross (E)\subset \Prored (E)$ denote the full subcategory formed by pro-semisimple groups.

Note that the functors $G\mapsto [G]$ and $G\mapsto \pi_0 (G)$ are well-defined on the category $\Prored (E)$. For $G\in\Prored (E)$ the affine scheme $[G]$ is equipped with a map $[G]\to\pi_0 (G)$ and an action of $\pi_0 (G)$.

It is easy to see and well known that for any  homomorphism $E\to \tilde E$ between algebraically closed fields the corresponding functors $\Prored (E)\to \Prored (\tilde E)$ and $\Pross (E)\to \Pross (\tilde E)$ are equivalences (in the pro-semisimple case, see Proposition~\ref{p:E-independence}).

\subsubsection{The objects $\hat\Pi_{(\lambda)}\in \Pross (\BQbar)$}     \label{sss:lambda-completion}
Fix an algebraic closure $\BQbar$ of $\BQ$. For each non-Archimedean place $\lambda$ of $\BQbar$ not dividing $p$, one can define an object 
$\hat\Pi_{(\lambda)}\in \Pross (\BQbar)$ as follows. 
For each subfield $E\subset\BQbar$ finite over $\BQ$ let $E_\lambda$ denote the completion of $E$ with respect to the place of $E$ corresponding to $\lambda$. Let $\BQbar_\lambda$ denote the direct limit of all such fields $E_\lambda\,$; this is an algebraic closure of $\BQ_\ell\,$. The embedding $\BQbar\mono\BQbar_\lambda$ induces an equivalence
\begin{equation}   \label{e:Qbar and completion}
\Pross (\BQbar )\iso \Pross (\BQbar_\lambda ).
\end{equation}
Let $\ell$ be the prime such that $\lambda$ divides $\ell$, then we have an embedding $\BQ_\ell\mono \BQbar_\lambda\,$. Set 
$\hat\Pi_\lambda :=\hat\Pi_\ell\otimes_{\BQ_{\ell}}\BQbar_\lambda\,$.
Consider the pro-semisimple group $\hat\Pi_\lambda$ as an object of $\Pross (\BQbar_\lambda)\,$, and then apply the functor inverse to \eqref{e:Qbar and completion}. Thus we get an object of $\Pross (\BQbar)$, which will be denoted by 
$\hat\Pi_{(\lambda)}\,$. Let us emphasize that $\hat\Pi_{(\lambda)}$ is defined as an object of the rather ``coarse" category $\Pross (\BQbar)$ but not as a group scheme ``on the nose".

\subsection{A result of L.~Lafforgue}    \label{ss:add-structure}
By \S\ref{sss:charFr}-\ref{sss:coarsecat}, we have an affine scheme $[\hat\Pi_{(\lambda)}]$ over $\BQbar$, a morphism $[\hat\Pi_{(\lambda)}]\to\pi_0 (\hat\Pi_{(\lambda)})=\Pi$, and an action of $\Pi$ on $[\hat\Pi_{(\lambda)}]$.  
We also have a canonical map 
\begin{equation}   \label{e:Frob4}
\Pi_{\Fr}\to [\hat\Pi_{(\lambda)}] (\BQbar_\lambda)\,  , 
\end{equation}
namely, the composition 
$\Pi_{\Fr}\mono\Pi\to\hat\Pi_\ell (\BQ_\ell)\to [\hat\Pi_\ell ](\BQ_\ell)\mono [\hat\Pi_{(\lambda)}] (\BQbar_\lambda)$. The map \eqref{e:Frob4} is clearly $\Pi$-equivariant. It has Zariski-dense image (because $\Pi_{\Fr}$ is dense in $\Pi$).

\begin{prop}  \label{p:algebraicity}
The image of the map \eqref{e:Frob4} is contained in  $[\hat\Pi_{(\lambda)}] (\BQbar)$.
\end{prop}

\begin{proof}
This immediately follows from Proposition~VII.7(i) of \cite{La} (which is a corollary of the Langlands conjecture proved by L.~Lafforgue in \cite{La}).
One can apply \cite[Proposition~VII.7]{La} because $X$ is assumed normal.
\end{proof}

Thus by Proposition~\ref{p:algebraicity}, we get a diagram of sets
\begin{equation}   \label{e:diagram of sets}
\Pi_{\Fr}\to [\hat\Pi_{(\lambda)}] (\BQbar)\epi\Pi\, .
\end{equation}

\subsection{Main theorem}
The following theorem is our main result.

\begin{theorem}  \label{t:main}
Assume that $X$ is smooth. Let $\lambda$ and $\lambda'$ be non-Archimedean places of $\BQbar$ not dividing $p$. Then there exists a unique isomorphism $\hat\Pi_{(\lambda)}\iso\hat\Pi_{(\lambda')}$ in the category $\Pross (\BQbar )$ 
which sends diagram \eqref{e:diagram of sets} to a similar diagram $\Pi_{\Fr}\to [\hat\Pi_{(\lambda')}] (\BQbar)\epi\Pi\,$. 
\end{theorem}

In this theorem the uniqueness statement holds without the smoothness assumption; this statement immediately follows from the easy Proposition~\ref{p:2rigidity} combined with Zariski-density of the image of the map $\Pi_{\Fr}\to [\hat\Pi_{(\lambda)}] $. 
The proof of the existence statement will be given in \S\ref{s:main}; it relies on the main theorem of \cite{Dr}, which is proved under the smoothness assumption on $X$.

\subsubsection{The object $\hat\Pi\in \Pross (\BQbar)$}  \label{sss:hatPi}
Theorem \ref{t:main} says that  the object $\hat\Pi_{(\lambda)}\in\Pross (\BQbar )$ does not depend on $\lambda$. We will denote it by 
$\hat\Pi\in\Pross (\BQbar )$ and call it \emph{the pro-semisimple completion of $\Pi$.} 

By definition, $\hat\Pi/\Pi^\circ$ canonically identifies with $\Pi$; moreover, $\hat\Pi$ is equipped with a canonical $\Pi$-equivariant map 
$\Pi_{\Fr}\to [\hat\Pi ] (\BQbar )$ with Zariski-dense image,  whose composition with the projection
$[\hat\Pi ]\to\Pi/\Pi^\circ=\Pi$ is equal to the inclusion $\Pi_{\Fr}\mono \Pi$. 

The object $\hat\Pi\in\Pross (\BQbar )$ is clearly $\Gal (\BQbar /\BQ)$-equivariant\footnote{The image of any element of $\Pi_{\Fr}$ in $[\hat\Pi_{(\lambda)}] (\BQbar)$  is fixed by the action of $\Gal (\BQbar /\BQ)$ on $[\hat\Pi_{(\lambda)}] (\BQbar)$.}; in particular, the Dynkin diagram of $\hat\Pi$ is equipped with a canonical action of $\Gal (\BQbar /\BQ)$, which commutes with the action of $\Pi=\Pi/\Pi^\circ$. 
A more detailed discussion of $\hat\Pi$ is contained in Appendix~\ref{s:what we know}.

\begin{rem}   \label{r:mot-hope1}
According to the philosophy of motives, the $\Gal (\BQbar /\BQ)$-equivariant object $\hat\Pi\in\Pross (\BQbar )$ should come from a much finer object, namely, a pro-semisimple gerbe\footnote{This is the gerbe of fiber functors on the conjectural Tannakian category $\cT (X)$ form \S\ref{sss:mot Tann}.} over $\BQ$ equipped with a morphism to the classifying stack of $\Pi$. Here ``pro-semisimple gerbe over $\BQ$" is a shorthand for ``a fpqc-gerbe on the category of $\BQ$-schemes which is locally isomorphic to the classifying stack of a pro-semisimple group scheme".

\end{rem}

\begin{rem}  \label{r:mot-hope2}
In \S\ref{sss:coarsecat} we defined $\Pross (E)$-iso\-mor\-phisms between $G_1$ and $G_2$ as elements of a certain quotient \emph{set}. 
Replacing the quotient set by the corresponding quotient \emph{groupoid}, one gets a 2-groupoid $\Pross_{\rm true}(E)$, whose 1-categorical truncation is $\Pross (E)$. We have the functor
\[
\pi_0 :\Pross_{\rm true}(E)\to\mbox{\{Pro-finite groups\}}.
\]
Now the motivic hope from Remark~\ref{r:mot-hope1} can be reformulated as follows: \emph{$\hat\Pi$ should canonically lift to a  $\Gal (\BQbar /\BQ)$-equivariant object $\hat\Pi_{\rm true}$ of the 2-groupoid $\Pross_{\rm true}(\BQbar )$ equipped with an isomorphism 
$\pi_0(\hat\Pi_{\rm true})\iso\Pi$.} 
In fact, the philosophy of motives suggests even more\footnote{Unlike Remarks~\ref{r:mot-hope1}-\ref{r:mot-hope2}, the conjectural picture of Appendix~\ref{s:Tannakian} takes in account the canonical \emph{polarization} on the category of pure motives.}, see Appendix~\ref{s:Tannakian}.
\end{rem}

\begin{rem}  
The fact that $\hat\Pi_{(\lambda)}^\circ\in\Pross (\BQbar )$ does not depend on $\lambda$ was proved by CheeWhye Chin in~\cite{Ch2}.
\end{rem}

\subsection{A crystalline analog of Theorem~\ref{t:main}}   \label{ss:crystalline analog}
Using overconvergent $F$-isocrystals on $X$, one can define $\hat\Pi_{(\lambda)}$ even for $\lambda$ dividing $p$, see  \S\ref{lambda |p}. If $\dim X=1$ then Theorem~\ref{t:main} holds for \emph{arbitrary} non-Archimedean places of $\BQbar$, see Theorem~\ref{t:p-main} (to prove this, one uses crystalline analogs of the results of \cite{La}, which were proved by T.~Abe \cite{A}). If $\dim X>1$ this is not clear (the missing piece is the existence of ``crystalline companions" of $\ell$-adic local systems on $X$).

\subsection{Relation with the works \cite{VLa2} and \cite{Ar}}
\subsubsection{Relation with \S 12.2.4 of V.~Lafforgue's article \cite{VLa2}}  
In \S\ref{ss:motivic Langlands parameter} we give an unconditional definition of ``motivic Langlands parameter" in the sense of  \cite[\S 12.2.4]{VLa2} using a certain object $\hat\Pi^{\mot}\in\Prored (\BQbar)$, which is a variant of the object $\hat\Pi\in \Pross (\BQbar)$ defined in \S\ref{sss:hatPi}. A brief discussion of $\hat\Pi^{\mot}$ is contained in \S\ref{sss:variant of hatPi} below.

\subsubsection{Relation with J.~Arthur's work \cite{Ar}}  
The relation is philosophical.

Assuming certain conjectural properties of automorphic representations, J.~Arthur constructs in \cite{Ar} the ``automorphic Langlands group" of a global field $F$ (the idea goes back to \cite[\S 2]{Lan}). 

Now suppose that $F$ is the field of rational functions on a smooth connected curve over $\BF_p\,$. Then
the ``automorphic Langlands group" from \cite{Lan,Ar} should be more or less\footnote{The words ``more or less" are mostly due to the fact that the automorphic Langlands group from \cite{Lan} is an algebraic group over $\BC$ (rather than $\BQbar$) and the one from \cite{Ar} is a locally compact topological group. In addition, the ``automorphic Langlands group" from \cite{Ar} is an extension of the Weil group rather than of the Galois group.} isomorphic to the projective limit of the pro-semisimple completions of $\pi_1 (U)$ for all non-empty open $U\subset X$.

Unlike \cite{Lan,Ar}, we consider only function fields (which are not required to be 1-dimensional), and we work with Galois representations (automorphic representations appear only behind the scenes, namely in the proof of the results of \cite[Ch.~VII, \S 2]{La}, which are crucial for us); this allows us to give an unconditional definition of 
$\hat\Pi$. Unlike \cite{Lan,Ar}, we do not study ramification.

\subsection{A variant of $\hat\Pi$ and an open problem}   \label{ss:weakly motivic}
\subsubsection{The group schemes $\hat\Pi_\lambda^{\mot}$ and $\hat\Pi^{\mot}$}  \label{sss:variant of hatPi}
It is easy to see that a finite-dimensional representation of $\hat\Pi_\lambda$ is the same as a semisimple $\lambda$-adic local system on $X$ with the following property: the determinant of each of its irreducible components has finite order. Unfortunately, this property is not  stable under pullbacks. At least for this reason, it is natural to replace it with the property which we call\footnote{``Weakly motivic" is just a name. In particular, if $\dim X>1$ it is unknown whether all weakly motivic local systems on $X$ come from motives over the field of rational functions on $X$.} ``weakly motivic"; by definition, this means that for every closed point $x\in X$ all eigenvalues of the geometric Frobenius of $x$ are $q_x$-Weil numbers, where $q_x$ is the order of the residue field of $x$. The property of ``weakly motivic" is clearly stable under pullbacks.

A semisimple weakly motivic $\lambda$-adic local system on $X$ is the same as a finite-dimensional representation of a certain pro-reductive group scheme $\hat\Pi_\lambda^{\mot}$ over $\BQbar_\lambda\,$, which can be easily expressed in terms of the pro-semisimple group 
$\hat\Pi_\lambda$. Similarly to \S\ref{sss:lambda-completion}, one defines $\hat\Pi_{(\lambda )}^{\mot}\in\Prored (\BQbar)$. Theorem~\ref{t:main} easily implies that $\hat\Pi_{(\lambda )}^{\mot}$ does not depend on $\lambda$, so one can write simply $\hat\Pi^{\mot}\in\Prored (\BQbar)$.

The details are explained in \S\ref{ss:def of mot}-\ref{ss:mot&red}.

\subsubsection{An open problem}
Now suppose we have  a morphism $X'\to X$ between smooth varieties over $\BF_p\,$. Let $\Pi'$ and $\Pi$ be the fundamental groups of $X'$ and $X$ corresponding to some geometric point of $X'$. The morphism $X'\to X$ induces a homomorphism $f:\Pi'\to\Pi$ and then a 
$(\hat\Pi^{\mot})^\circ$-conjugacy  class of homomorphisms $\hat f_{(\lambda)}:\widehat{\Pi'}^{\mot}\to\hat\Pi^{\mot}$, 
which \emph{a priori} depends on the additional choice of a non-Archimedean place $\lambda$ of $\BQbar$ not dividing $p$.
Conjecture \ref{c:functoriality in X} says that $\hat f_{(\lambda)}$ does not depend on $\lambda$.

I cannot prove this conjecture in general (e.g., if $X$ is a surface and $X'$ is a curve on $X$). The difficulty is due to the difference between conjugacy and ``element-conjugacy" in the sense of M.~Larsen~\cite{L1,L2,W1,W2}.

 \subsection{Method of the proof of Theorem~\ref{t:main}}
 The method is quite elementary modulo the results of~\cite{La,Dr}. 
  
 Let $K^+(\hat\Pi_{(\lambda )} )$ denote the Grothendieck semiring of the category of finite-dimensional representations of $\hat\Pi_{(\lambda )}$. It is equipped with the lambda-operations (here ``lambda" stands for ``exterior power"). The lambda-semiring  $K^+(\hat\Pi_{(\lambda )} )$ is a ``poor man's substitute" for the Tannakian category of representations of $\hat\Pi_{(\lambda )}\,$.
 
 Associating to a $\lambda$-adic representation of $\Pi$ the restriction of its character to 
 $\Pi_{\Fr}\subset\Pi$, we realize $K^+(\hat\Pi_{(\lambda )} )$ as a lambda-subsemiring of the algebra of functions\footnote{Functions $\Pi_{\Fr}\to \BQbar$ form a lambda-ring. The corresponding Adams operation $\psi^n$ is defined by $(\psi^n f)(g):=f(g^n)$.} $\Pi_{\Fr}\to \BQbar$. According to \cite{Dr} (and according to \cite{La} if $\dim X=1$), this lambda-subsemiring does not depend on the place $\lambda$ of the field $\BQbar$. This remains true if one replaces $\Pi$ by any open subgroup $U\subset\Pi$. We prove a reconstruction theorem in the spirit of Kazhdan-Larsen-Varshavsky \cite{KLV}, which allows one to reconstruct $\hat\Pi_{(\lambda )}$ from the lambda-semirings $K^+(\hat\Pi_{(\lambda )}\times_\Pi U)$ and the natural homomorphisms between them assuming that $H^2 (U ,\BQ/\BZ)=0$ for any open subgroup $U\subset\Pi$. This assumption holds if $\dim X=1$; the general case is treated by reduction to curves using Hilbert irreducibility.

 \subsection{Dealing with 1-categorical truncations of 2-groupoids}
 Our $\hat\Pi_{(\lambda )}$ is an object of the groupoid $\Pross (\BQbar )$, which is a 1-categorical truncation of the conceptually better
2-groupoid $\Pross_{\rm true} (\BQbar )$ defined in Remark~\ref {r:mot-hope2}.  I cannot prove any theorem without this truncation. On the other hand, the truncation leads to some technical difficulties.
 
One of the main difficulties is as follows. As explained in \S\ref{ss:Pross and variants}, an object of $\Pross (E)$ involves a group extension~\eqref{e:ext of pi}. So proving independence of $\hat\Pi_{(\lambda )}$ on $\lambda$ involves showing that a certain extension of $\Pi$ by a certain pro-finite abelian group $Z$ does not depend on $\lambda$ as an object of the ``coarse" groupoid whose morphisms are isomorphisms of extensions \emph{modulo conjugations by elements of $Z$.} Despite a certain ugliness of this groupoid, this turns out to be possible due to the fact that the automorphism group of the Dynkin diagram of any simple Lie algebra is either cyclic or the group $S_3$ (which is not far from being cyclic\footnote{The relevant property of $S_3$ is that all its Sylow subgroups are cyclic. Finite groups with this property are discussed in Appendix~\ref{s:Zassenhaus}.}). This fact is used in the proof of Propositions~\ref{p:rigidity} and \ref{p:coarse extensions}.  
 
 \subsection{Organization of the article}
 In \S\ref{s:group schemes} we briefly discuss affine group schemes over algebraically closed fields of characteristic 0. We also explain there several ways to think about the groupoid $\Pross (E)$ defined in \S\ref{sss:coarsecat}. Finally we prove Proposition~\ref{p:rigidity}, which says
 that if $G\in\Pross (E)$ is such that $G^\circ$ is a product of almost-simple groups then $G$ has no nontrivial $\Pross (E)$-automorphisms inducing the identity on the scheme $[G]$.
 
 In \S\ref{s:completion} we recall some standard facts about the $\lambda$-adic pro-semisimple completion $\hat\Pi_\lambda\,$; in particular, we prove that it is simply connected\footnote{This simply-connectedness property is easy and well known to the experts. Probably it motivated Question~8.1 from Serre's article~\cite{Se}.} (see Proposition~\ref{p:2simply connected}). Combining this with Proposition~\ref{p:rigidity}, we prove in \S\ref{ss:rigidity} the uniqueness statement of our main Theorem~\ref{t:main}. In \S\ref{ss:pro-reductive} we discuss the $\lambda$-adic pro-reductive completion of $\Pi$.

In \S\ref{s:KLV} we formulate and prove a reconstruction theorem in the spirit of Kazhdan-Larsen-Varshavsky \cite{KLV} (see Theorem~\ref{t:variant of KLV}).

In \S\ref{s:main} we prove Theorem~\ref{t:main} by combining Theorem~\ref{t:variant of KLV} with the results of \cite{La,Dr}.

In \S\ref{s:mot} we explain the details related to the group schemes $\hat\Pi_\lambda^{\mot}$ and $\hat\Pi^{\mot}$ mentioned in \S\ref{ss:weakly motivic}. In \S\ref{ss:motivic Langlands parameter} we give an unconditional definition of ``motivic Langlands parameter" in the sense of  \cite[\S 12.2.4]{VLa2}.

\S\ref{s:crys} is devoted to the crystalline analog of Theorem~\ref{t:main} mentioned in  \S\ref{ss:crystalline analog}.

In Appendix~\ref{s:coarse hom} we prove the technical Proposition~\ref{p:coarse hom}.

In Appendix \ref{s:twisted conjugacy} we recall some results of S.~Mohrdieck and T.~A.~Springer on twisted conjugacy. We also recall there a surprising theorem of J.~C.~Jantzen and use it to construct an automorphism of the Grothendieck semiring of $SL(2n+1)$ which is not compatible with the lambda-operations (this shows that a certain difficulty in the proof of Theorem~\ref{t:variant of KLV} is not imaginary).

Appendix~\ref{s:Zassenhaus} is devoted to some facts from finite group theory related to the main body of the article.

The pair consisting of $\hat\Pi\in\Pross (\BQbar )$ and the canonical map $\Pi_{\Fr}\to [\hat\Pi](\BQbar )$
is an interesting number-theoretic object. In the first part of Appendix~\ref{s:what we know} we reformulate these data in more ``elementary" terms (such as Dynkin diagrams and extensions of $\Pi$ by finite abelian groups). In \S\ref{ss:Laurent Lafforgue}-\ref{ss:p-adic behavior} 
we translate some results of \cite{La,VLa, DK} into this language.

In Appendix~\ref{s:Tannakian} we formulate the conjectural picture suggested by the philosophy of motives and briefly mentioned in Remarks~\ref{r:mot-hope1}-\ref{r:mot-hope2}. This picture implies falsifiable Conjectures \ref{conj:reciprocity} and \ref{conj:2reciprocity}; each of them is a ``reciprocity law" involving a sum over all $\ell$-adic cohomology theories (including the crystalline theory for $\ell=p$). We also compare the conjectural picture with the unconditional results.

 \subsection{Acknowledgements} 

I thank G.~Glauberman, H.~Esnault, K.~S.~Kedlaya, V.~Lafforgue,  E.~B.~Vinberg, D.~Vogan, and Xinwen Zhu for stimulating discussions and useful references.

\section{Generalities on group schemes}   \label{s:group schemes}
Fix an algebraically closed field $E$ of characteristic 0. Unless stated otherwise, the words ``group scheme" will mean ``affine group scheme over $E$". Since $E$ has characteristic 0 all group schemes are automatically reduced. By an algebraic group we mean a group scheme of finite type.

Let us emphasize that semisimple or reductive groups are \emph{not assumed to be connected}.

\subsection{Conventions and general remarks} \label{ss:group schemes gen}
\subsubsection{Group schemes as pro-objects} \label{sss:group schemes pro}
A group scheme is the same as a pro-object of the category of algebraic groups. Moreover, any such pro-object can be represented by a filtering projective system of algebraic groups in which all transition maps are \emph{epimorphisms.} 

This point of view allows one to easily generalize various notions and results from algebraic groups to group schemes.

A \emph{pro-semisimple group} is a group scheme $G$ such that each quotient of $G$ having finite type is semisimple. Similarly, one has the notions of \emph{pro-reductive group} and \emph{pro-torus}. 
Note that a group scheme $G$ all of whose finite type quotients are finite is the same as a pro-finite group in the usual sense (i.e., a totally disconnected compact topological group).

\subsubsection{Almost-simple groups} 
A connected algebraic group is said to be \emph{almost-simple} if any normal subgroup of $G$ different from $G$ is finite.

Any connected and simply-connected pro-semisimple group is a product of almost-simple algebraic groups. Any connected pro-semisimple group with trivial center is a product of simple algebraic groups.

\subsubsection{Pinnings}   \label{sss:pinnings}
Recall that according to Bourbaki, a \emph{pinning} of a reductive group $G$ is a choice of the following data: a maximal subtorus of $G^\circ$, a Borel subgroup of 
$G^\circ$ containing it, and for each simple root $\alpha$, a nonzero element in the root space $\fg_\alpha\, $, where $\fg$ is the Lie algebra of $G$.

If $G$ is a pro-reductive group then define a \emph{pinning} of $G$ to be a collection of pinnings of the simple quotients of $G^\circ_{ad}\,$. (Thus the notion of pinning depends only on $G^\circ$.) A \emph{pinned pro-reductive group} is a pro-reductive group equipped with a pinning.

\subsubsection{Dynkin diagrams}  \label{sss:Dynkin diagrams} 
By a \emph{finite Dynkin diagram} we mean a Dynkin diagram of some semisimple algebraic group over an algebraically closed field.
(E.g.,  $E_9$ is not allowed.) By a \emph{Dynkin diagram} we mean a (possibly infinite) disjoint union of finite Dynkin diagrams.

To any pro-reductive group $G$ one associates a Dynkin diagram, namely the disjoint union of the Dynkin diagrams of all simple quotients of $G^\circ_{ad}\,$. The group $\pi_0(G)$ acts on the Dynkin diagram of $G$; this action is continuous, i.e., the stabilizer of any vertex (or equivalently, of any connected component) is open in $\pi_0(G)$.

Given a pro-reductive group scheme $G$, one can talk about \emph{the} maximal pro-torus\footnote{The key point is that if $B,B'\subset G^\circ$ are Borel subgroups then $B/[B,B]$ identifies with $B'/[B',B']$ \emph{canonically} (because the normalizer of $B$ in $G^\circ$ equals $B$). } of $G$ and the Dynkin diagram of $G$ (in the above sense); they depend only on $G^\circ$.

\subsubsection{Some groups corresponding to a Dynkin diagram}   \label{sss:some groups}
Given a Dynkin diagram $\Delta$, let $G_\Delta$ denote the connected simply connected pinned pro-semisimple group with Dynkin diagram $\Delta$. Let $Z_\Delta$ denote the center of $G_\Delta\,$. Let $Z_\Delta (E)$ denote the group of $E$-points of $Z_\Delta$.

Note that $G_\Delta$ and $Z_\Delta$ are naturally defined over $\BQ$, so the group $\Aut (E)$ acts on $Z_\Delta (E)$. This action is usually nontrivial.\footnote{This is why in \S\ref{sss:Pross''}(iii) we write $Z_\Delta (E)$ rather than $Z_\Delta\,$.}

\subsection{The groupoid $\Pross (E)$}  \label{ss:Pross and variants}
The groupoid $\Pross (E)$ was introduced in  \S\ref{sss:coarsecat}. Recall that its objects are pro-semisimple groups over $E$, and a $\Pross (E)$-morphism $G_1\to G_2$ is an isomorphism of group schemes $G_1\to G_2$ defined up to composing with automorphisms of $G_2$ of the form $x\mapsto gxg^{-1}$, $g\in G^\circ_2\,$. 

Now we will introduce two other ``realizations" of the groupoid $\Pross (E)$. More precisely, we will define groupoids $\Pross' (E)$, 
$\Pross^{\prime\prime}(E)$, and equivalences
\[
\Pross (E) \buildrel{\sim}\over{\longleftarrow}\Pross' (E)\iso\Pross^{\prime\prime}(E).
\]

\subsubsection{The groupoid $\Pross' (E)$}  \label{sss:pinned}
Its objects are pinned pro-semisimple groups over $E$ (in the sense of \S\ref{sss:pinnings}), and its morphisms are as follows: if $G_1, G_2$ are pinned pro-semisimple groups then a $\Pross' (E)$-morphism $G_1\to G_2$ is an isomorphism of group schemes $G_1\to G_2$ compatible with the pinnings and defined up to composing with automorphisms of $G_2$ of the form $x\mapsto gxg^{-1}$, where $g$ belongs to the center of $G^\circ_2\,$. The canonical functor $\Pross' (E)\to\Pross  (E)$ is an equivalence.

\subsubsection{The groupoid $\Pross^{\prime\prime}(E)$}   \label{sss:Pross''}
Its objects are the following collections of data:

(i) a pro-finite group $\Gamma$;

(ii) a Dynkin diagram $\Delta$ equipped with a continuous action of $\Gamma$ (continuity means that the stabilizer of each vertex of $\Delta$ is open);

(iii) a pro-finite group $Z$ equipped with an action of $\Gamma$ and a $\Gamma$-equivariant epimorphism $Z_{\Delta}(E)\epi Z$;

(iv) a group extension 
\begin{equation}   \label{e:ext of pi}
0\to Z\to \tilde\Gamma\to\Gamma\to 0.
\end{equation}

There is an obvious notion of isomorphism between two such collections, but to define $\Pross^{\prime\prime}(E)$ we use a \emph{coarser} one,
in which isomorphisms between extensions 
\[
0\to Z_1\to \tilde\Gamma_1\to\Gamma_1\to 0 \quad \mbox{and} \quad 0\to Z_2\to \tilde\Gamma_2\to\Gamma_2\to 0
\]
 are defined only up to composing with inner automorphisms of $\tilde\Gamma_2$ corresponding to elements of $Z_2\,$. 
 
 \subsubsection{Remark} \label{sss:rem H^2}
 For fixed $\Gamma$ and $Z$, isomorphism classes of extensions \eqref{e:ext of pi} are parametrized by $H^2(\Gamma, Z)$. On the other hand, automorphisms of an extension \eqref{e:ext of pi} up to conjugations by elements of $Z$ are parametrized by $H^1(\Gamma, Z)$.

\subsubsection{The equivalence $\Pross' (E)\iso\Pross^{\prime\prime}(E)$}  
\label{sss:Pross' and Pross''}
To a pinned pro-semisimple group $G$ one associates the collection $(\Gamma ,\Delta , Z, \tilde\Gamma )$, where $\Gamma=G/G^\circ$, $\Delta$ is the Dynkin diagram of $G$, $Z$ is the center of $G^\circ$, and $\tilde\Gamma$ is the subgroup of elements $g\in G(E)$ such that conjugation by $g$ preserves the pinning. Thus one gets an equivalence $\Pross' (E)\iso\Pross^{\prime\prime}(E)$.

The inverse equivalence $\Pross^{\prime\prime}(E)\iso\Pross' (E)$ can be described as follows. 
Given $\Delta$ and $Z$, one sets $G^\circ:=G_\Delta/\Ker (Z_\Delta\epi Z)$, where $G_\Delta$ and
$Z_\Delta$ are as in \S\ref{sss:some groups}.
Then one sets
\begin{equation}  \label{e:inverse construction}
G:=(\tilde\Gamma\ltimes G^\circ)/\Ker (Z\times Z\overset{m}\longrightarrow Z).
\end{equation}
Here the semidirect product $\tilde\Gamma\ltimes G^\circ$ is defined using the canonical action of $\Gamma$ on $G^\circ$ (i.e., the action preserving the pinning), $Z\times Z$ is a subgroup of $\tilde\Gamma\ltimes G^\circ$, and $m:Z\times Z\to Z$ is the multiplication map.

\begin{prop}   \label{p:E-independence}
For any  homomorphism $E\to \tilde E$ between algebraically closed fields, the corresponding functor $\Pross (E)\to \Pross (\tilde E)$ is an equivalence.
\end{prop}

\begin{proof}   
The functor $\Pross^{\prime\prime} (E)\to \Pross^{\prime\prime} (\tilde E )$ is clearly an equivalence.
\end{proof}

\subsubsection{The element of $H^2$ corresponding to an object of $\Pross (E)$}   \label{sss:class in H^2}
Combining \S\ref{sss:rem H^2} with the equivalence $\Pross (E)\iso\Pross^{\prime\prime}(E)$, we see that any object $G\in \Pross (E)$ defines an element $\nu\in H^2 (\Gamma ,Z)$, 
where $\Gamma:=\pi_0 (G)$ and $Z$ is the center of $G^\circ$. Namely, fix a pinning of $G$ and let 
$\tilde\Gamma$ be the group formed by elements $g\in G(E)$ such that conjugation by $g$ preserves the pinning; then $\tilde\Gamma$ is an extension of $\Gamma$ by $Z(E)=Z$, and $\nu\in H^2 (\Gamma ,Z )$ is its class.

\subsubsection{Analogs for $\Prored (E)$}
Proposition~\ref{p:E-independence} remains valid if one replaces $\Pross$ by $\Prored$. For a stronger statement, see Proposition~\ref{p:reductive coarse} below.

The constructions of \S\ref{sss:pinned}-\ref{sss:Pross' and Pross''} can be easily generalized to pro-reductive groups (of course, in the pro-reductive case the group $Z$ from \S\ref{sss:Pross''} should be a commutative group scheme, and instead of surjectivity of the homomorphism $Z_\Delta\to Z$ one should require its cokernel to be a torus). Details are left to the reader.

\subsection{Coarse homomorphisms}
Given group schemes $G_1\, ,G_2$ over $E$, let $\Hom_{\coarse} (G_1\, ,G_2)$ denote the quotient of the set $\Hom(G_1\, ,G_2)$ by the conjugation action of $G_2^\circ (E)$. 

\begin{prop} \label{p:coarse hom}
The functor on the category of group schemes defined by $G_2\mapsto \Hom_{\coarse} (G_1\, ,G_2)$ commutes with filtering projective limits.
\end{prop}

The proposition will be proved in Appendix~\ref{s:coarse hom}. In most situations we will use the following easy particular case of Proposition~\ref{p:coarse hom}.

\begin{prop}  \label{p:coarse epi}
Let $G_1\, ,G_2$ be pro-semisimple groups over $E$. Then the canonical map
\begin{equation} \label{e:2bijection}
\Epi_{\coarse} (G_1\, ,G_2)\to\underset{H\in N}{\underset{\longleftarrow}{\lim}}\Epi_{\coarse} (G_1\, ,G_2/H)
\end{equation}
is bijective. Here $\Epi_{\coarse}\subset\Hom_{\coarse}$ is the subset corresponding to epimorphisms and $N$ is the subset of all normal subgroups $H\subset G_2$ such that $G_2/H$ has finite type.
\end{prop}

\begin{proof}
Suppose we are given a compatible family of elements $\xi_H\in \Epi_{\coarse} (G_1\, ,G_2/H)$, $H\in N$. Fix pinnings of $G_1\, ,G_2\,$. 
For each $H\in N$ let $S_H$ be the preimage of $\xi_H$ in the set of epimorphisms $G_1\to G_2/H$ compatible with the pinnings. Then $S_H$ is a non-empty set equipped with a transitive action of $Z_H\,$, where $Z_H$ is the center of $(G_2/H)^\circ$. The groups $Z_H$ are finite, so it is clear that $\underset{\longleftarrow}{\lim}\, S_H$ is a non-empty set equipped with a transitive action of $\underset{\longleftarrow}{\lim}\, Z_H\,$.
The proposition follows.
\end{proof}

The following statement is a generalization of Proposition~\ref{p:E-independence}.

\begin{prop} \label{p:reductive coarse}
Let $\tilde E$ be an algebraically closed field containing $E$. Let $G_1,G_2$ be group schemes over $E$. Assume that $G_1$ is pro-reductive.
Then the canonical map  $$\Hom_{\coarse} (G_1\, ,G_2)\to\Hom_{\coarse} (G_1\otimes_E\tilde E\, ,G_2\otimes_E\tilde E)$$
is bijective.
\end{prop}

\begin{proof}  
By Proposition~\ref{p:coarse hom}, we can assume that $G_2$ has finite type. Any normal subgroup of $G_1\otimes_E\tilde E$ is defined over 
$E$, so we can also assume that $G_1$ has finite type. In the finite type case the proposition is well known. For completeness, let us prove surjectivity in this case.

We have to show that any homomorphism $\tilde f:G_1\otimes_E\tilde E\to G_2\otimes_E\tilde E$ is $G_2^\circ (\tilde E)$-conjugate to a homomorphism defined over $E$. Clearly $\tilde f$ comes from a homomorphism $f:G_1\otimes_EA\to G_2\otimes_EA$, where $A\subset\tilde E$ is a finitely generated $E$-subalgebra. Choose a maximal ideal $m\subset A$, then $A/m=E$. Let $f_0:G_1\to G_2$ be the reduction of $f$ modulo $m$. We claim that $\tilde f$ is $G_2^\circ (\tilde E)$-conjugate to $f_0\,$. To prove this, consider the functor that to any $A$-algebra $\bar A$ associates the set of all $g\in G_2^\circ (\bar A)$ such that the $g$-conjugate of $\bar f_0$ equals~$\bar f$ (here 
$\bar f_0:G_1\otimes_E\bar A\to G_2\otimes_E\bar A$ and $\bar f:G_1\otimes_E\bar A\to G_2\otimes_E\bar A$ are obtained from $f_0$ and $f$ by base change). This functor is representable by a finitely generated $A$-algebra $B$. The problem is to show that $\Hom_A (B,\tilde E)\ne\emptyset$. It suffices to prove that $\Ker (A\to B)=0$. To do this, it suffices to check that $\Hom_A (B,\hat A)\ne\emptyset$, where $\hat A$ is the completion of $A$ at $m$. One proves this by a standard deformation-theoretic argument using the equality $H^1 (G_1\, , \fg_2)=0$, where $\fg_2$ is the Lie algebra of 
$G_2\,$.
\end{proof}

\subsection{The functor $G\mapsto [G]$}  \label{ss:[G]}
\subsubsection{Definition of $[G]$}
Recall that for any pro-reductive group $G$ we denote by $[G]$ the GIT quotient of $G$ by the conjugation action of $G^\circ$ (i.e., $[G]$ is the spectrum of the algebra of those regular functions on $G$ that are invariant under $G^\circ$-conjugation). Clearly $[G]$ is an affine scheme equipped with an action of $\pi_0(G):=G/G^\circ$ (by conjugation) and a morphism $\nu :[G]\to\pi_0(G)$. The map $\nu$ is $\pi_0(G)$-equivariant, and it is easy to see that the action of any $\gamma\in \pi_0(G)$ on $\nu^{-1} (\gamma )$ is trivial. It is also clear that the fibers of $\nu$ are irreducible.

Conjugation by elements of $G^\circ$ acts trivially on $[G]$ and $\pi_0 (G)$. So the functors $G\mapsto [G]$ and $G\mapsto\pi_0 (G)$ are well-defined on the groupoid $\Pross (E)$.

In Appendix~\ref{s:twisted conjugacy} we recall an explicit description of $[G]$; if $G$ is connected it amounts to the usual identification of $[G]$ with~$T/W$.

\subsubsection{Action of the center of $G^\circ$}
Let $G$ be a  pro-reductive group and $Z$ the center of $G^\circ$. Then $Z$ acts on the scheme $G$ by left multiplication and right multiplication. These actions may be different, but they induce the same action of $Z$ on $[G]$.

\subsubsection{The goal of this subsection}
We will prove the following statement.
\begin{prop}   \label{p:rigidity}
Let $G\in \Pross (E)$. Suppose that $\varphi\in\Aut G$ induces the identity on $[G]$. If $G^\circ$ is a product of almost-simple groups then 
$\varphi =\id_G\,$.
\end{prop}

The proof will be given in \S\ref{sss:proof-rigidity} below. We need the following lemmas.

\begin{lem}  \label{l:cyclic_Sylow}
Let $A$ be an abelian group equipped with an action of a group $H$. Assume that the group $H':=\im (H\to\Aut A)$ is finite and all its Sylow subgroups are cyclic. Suppose that $u\in H^1(H,A)$ has zero restriction to any cyclic subgroup of $H$. Then $u=0$.
\end{lem}

\begin{proof}
The restriction of $u$ to $\Ker (H\to\Aut A)$ is a homomorphism $H\to A$. This homomorphism is zero (because so is its restriction to any cyclic subgroup). So $u\in H^1 (H',A)$. For each prime $p$ the restriction of $u$ to the corresponding Sylow subgroup $\Syl_p\subset H'$ is zero, so
$(H':\Syl_p)\cdot u=0$. Therefore $u=0$.
\end{proof}

\begin{rem}
In Appendix~\ref{s:Zassenhaus} we recall a theorem of Zassenhaus describing the class of finite groups considered in Lemma~\ref{l:cyclic_Sylow}. We also formulate there  a generalization of Lemma~\ref{l:cyclic_Sylow} (see Proposition~\ref{p:2coarse extensions} and Remark~\ref{r:why generalization}).
\end{rem}

\begin{lem}  \label{l:minuscule}
Let $H$ be a connected semisimple group. Let $Z$ (resp.~T and $\Delta$) be the center of $H$ (resp.~the maximal torus and Dynkin diagram). Let $\chi\in\Hom (Z,\BG_m)$, and let $\Stab_\chi\subset\Aut\Delta$ be the stabilizer of $\chi$. Then there exists a dominant $\Stab_\chi$-invariant weight $\omega\in\Hom (T,\BG_m)$ such that $\omega|_Z=\chi$.
\end{lem}

\begin{proof}
The theory of minuscule weights tells us that there exists a unique dominant $\omega\in\Hom (T,\BG_m)$ such that $\omega|_Z=\chi$ and 
$(\omega ,\check\alpha )\in\{ 0,1\}$ for each positive coroot $\check\alpha$. Clearly $\omega$ is $\Stab_\chi$-invariant.
\end{proof}

\begin{lem}  \label{l:no kernel}
Let $G$ be a semisimple group and $Z$ the  center of $G^\circ$. Let $\sigma\in\pi_0(G)$,  and let $[G]_\sigma$ denote the preimage of $\sigma$ in $[G]$. If $z\in Z$ is such that multiplication by $z$ acts trivially on $[G]_\sigma$ then $z=\sigma (\zeta )/\zeta$ for some  $\zeta\in Z$ (here $\sigma (\zeta )$ denotes the result of the conjugation action of $\sigma\in\pi_0(G)$ on $\zeta$).
\end{lem}

Let us note that the converse statement is obvious because if $g\in G$ belongs to the preimage of $\sigma$ then $\zeta^{-1}\cdot  \sigma (\zeta)\cdot g=\zeta^{-1}g\zeta$.

\begin{proof}
We can assume that $\pi_0(G)$ is generated by $\sigma$ (otherwise replace $G$ by a subgroup).
It suffices to show that $z$ is killed by any $\sigma$-invariant character $\chi :Z\to\BG_m\,$. By Lemma~\ref{l:minuscule}, such a character can be extended to a $\sigma$-invariant weight $\omega$ of $G^\circ$.
 Let $\rho$ be a representation of $G$ whose restriction to $G^\circ$ is the irreducible representation with highest weight $\omega$ (such $\rho$ exists because $\omega$ is $\sigma$-invariant and $\pi_0(G)$ is generated by~$\sigma$).

Since multiplication by $z$ acts trivially on $[G]_\sigma$ we have $\Tr\rho (zg)=\Tr\rho (g)$ for all $g\in [G]_\sigma\,$. Equivalently,
$(\omega (z)-1)\cdot \Tr\rho (g)=0$ for $g\in [G]_\sigma\,$. So $\omega (z)=1$ unless $\Tr\rho (g)=0$ for all $g\in [G]_\sigma\,$.

It remains to show that the latter is impossible. Choose some $\gamma\in G$ mapping to $\sigma\in\pi_0 (G)$. We have to show that 
$\Tr (\rho (\gamma )\cdot\rho (g))$ cannot be zero for all $g\in G^\circ$. This is clear because $\rho (\gamma )\ne 0$ and the restriction of $\rho$ to $G^\circ$ is irreducible.
\end{proof}

\subsubsection{Proof of Proposition~\ref{p:rigidity}}   \label{sss:proof-rigidity}
We will denote by $\varphi$ an automorphism of $G$ defined on the nose (rather than up to inner automorphisms corresponding to elements of $G^\circ$). The problem is to show that $\varphi$ is given by conjugation with an element of $G^\circ$.

Since $\varphi$ acts trivially on $[G^\circ ]$ the restriction of $\varphi$ to $G^\circ$ is inner. We can assume that it equals the identity. The problem is then to show that $\varphi$ is given by conjugation with an element of $Z$, where $Z$ is the center of $G^\circ$.

Since $\varphi$ acts trivially on $[G]$ it acts trivially on $\pi_0 (G)$. So $\varphi$ has the form $\varphi (g)=f (g)\cdot g$, where 
$f:\pi_0 (G)\to Z$ is a 1-cocycle. The problem is to prove that the class $[f]\in H^1(\pi_0 (G),Z)$ equals $0$. It is enough to do this if $G$ has finite type. 
Using Shapiro's lemma and the assumption that $G^\circ$ is a product of almost-simple groups we reduce this to the situation where $G^\circ$ is
almost-simple. In this case it suffices to use Lemmas~\ref{l:cyclic_Sylow} and \ref{l:no kernel} (Lemma~\ref{l:cyclic_Sylow} is applicable because the automorphism group of a connected Dynkin diagram is either cyclic or isomorphic to $S_3$). \qed

\section{Generalities on the group scheme $\hat\Pi_\lambda$}   \label{s:completion}
The material of this section is well known, but I was unable to find references.

\subsection{The $\BQ_\ell$-pro-algebraic completion of $\Pi$}
In this subsection $\Pi$ can be any pro-finite group.

\subsubsection{}  \label{sss:0pro-algebraic}
The \emph{$\BQ_\ell$-pro-algebraic completion} of $\Pi$ is defined to be the affine group scheme $\tilde\Pi_\ell$ over $\BQ_\ell$ co-representing the functor
\[
H\mapsto\Hom_{\cont} (\Pi ,H(\BQ_\ell ))
\]
on the category of affine algebraic groups $H$ over $\BQ_\ell\,$. Equivalently, 
$\tilde\Pi_\ell$ is the projective limit of all algebraic groups $H$ over $\BQ_\ell$ equipped with a continuous homomorphism $\rho :\Pi\to H(\BQ_\ell )$ such that $\rho (\Pi )$ is Zariski-dense in $H$.

The above homomorphisms $\rho :\Pi\to H(\BQ_\ell )$ yield a canonical homomorphism $\Pi\to\tilde\Pi_\ell (\BQ_\ell )$.

On the other hand, for any finite group scheme $\Gamma$ over $\BQ_\ell$ one has
\[
\Hom (\tilde\Pi_\ell/\tilde\Pi_\ell^\circ \,,\Gamma )=\Hom (\tilde\Pi_\ell\, ,\Gamma )=\Hom_{\cont} (\Pi ,\Gamma (\BQ_\ell )),
\]
so $\tilde\Pi_\ell/\tilde\Pi_\ell^\circ$ canonically identifies with $\Pi$ (considered as a group scheme over $\BQ_\ell$). In other words, one has a canonical epimorphism $\tilde\Pi_\ell\epi\Pi$, whose kernel equals $\tilde\Pi_\ell^\circ\,$. 

The composition $\Pi\to\tilde\Pi_\ell (\BQ_\ell )\to\Pi$ equals the identity. 

\begin{lem}   \label{l:subgr of fin index}
Let $\Pi'\subset\Pi$ be an open subgroup. Then the canonical homomorphism $f:\tilde\Pi'_\ell\to\tilde\Pi_\ell$ is injective, and its image is equal to $\tilde\Pi_\ell\times_{\Pi}\Pi'$.
\end{lem}

\begin{proof}

Let $g:\Pi\to\tilde\Pi_\ell (\BQ_\ell )$ and $g':\Pi'\to\tilde\Pi'_\ell (\BQ_\ell )$ be the canonical homomorphisms. Since $g'$ has Zariski-dense image, $f (\tilde\Pi'_\ell)$ is the Zariski closure of $(f\circ g') (\BQ_\ell )$. Since $f\circ g'=g|_{\Pi'}$ and $g$ has Zariski-dense image, the Zariski closure of $(f\circ g') (\BQ_\ell )$ equals $\tilde\Pi_\ell\times_{\Pi}\Pi'$.

Let us prove that $f$ is injective.
Note that a finite-dimensional $\tilde\Pi_\ell$-module (resp.~$\tilde\Pi'_\ell$-module) is the same as an $\ell$-adic representation of the pro-finite group $\Pi$ (resp.~$\Pi'$). So for any finite-dimensional $\tilde\Pi'_\ell$-module $V'$ there exists a 
finite-dimensional $\tilde\Pi_\ell$-module $V$ equipped with an injective homomorphism of $\tilde\Pi'_\ell$-modules $V'\mono V$, where the 
$\tilde\Pi'_\ell$-module structure on $V$ is defined via $f:\tilde\Pi'_\ell\to\tilde\Pi_\ell\,$. Then $\Ker f$ acts trivially on $V$ and therefore on $V'$. Since $\Ker f$ acts trivially on any $V'$ we see that $\Ker f$ is trivial.
\end{proof}

\begin{lem}  \label{l:Homs from neutral comp}
For any algebraic group $H$ over $\BQ_\ell$ one has $\Hom (\tilde\Pi_\ell^\circ ,H)=F(H)$, where
\[
F(H):=\underset{\Pi'}{\underset{\longrightarrow}{\lim}}\, \Hom_{\cont} (\Pi' ,H(\BQ_\ell )).
\]
(the direct limit is taken over the set of all open subgroups $\Pi'\subset\Pi$).
\end{lem}

Note that elements of the above set $F(H)$ are \emph{germs} of continuous homomorphisms $\Pi' \to H(\BQ_\ell )$.

\begin{proof}
 Write $\Hom (\tilde\Pi_\ell^\circ ,H)$ as the direct limit of $\Hom (\tilde\Pi_\ell\times_{\Pi}\Pi' ,H)$, then apply Lemma~\ref{l:subgr of fin index}.
\end{proof}

\begin{prop} \label{p:simply connected}
The group scheme $\tilde\Pi_\ell^\circ$ is simply connected.
\end{prop}

\begin{proof}
Let $F$ be as in Lemma~\ref{l:Homs from neutral comp}. It suffices to check that for any etale homomorphism $\varphi :H'\to H$ of algebraic groups  over $\BQ_\ell\,$, the corresponding map $F(H')\to F(H)$ is bijective. This is clear because there exists an open subgroup 
$U\subset H' (\BQ_\ell )$ such that the map $U\to H(\BQ_\ell )$ induced by $\varphi$ is an open embedding.
\end{proof}

\subsection{Continuous homomorphisms from $\Pi$ to algebraic groups over $\BQbar_\lambda\,$}
The notation $\BQbar_\lambda$ was introduced in \S\ref{sss:lambda-completion}. Recall that $\BQbar_\lambda$ is an algebraic closure of $\BQ_\ell$ (the fact that $\BQbar_\lambda$ is equipped with an embedding of a fixed algebraic closure of $\BQ$ is not important for now).

\subsubsection{} 
We equip $\BQbar_\lambda$ with the topology defined by the absolute value. Then for any scheme $Y$ over $\BQbar_\lambda$ we get a topology on $Y (\BQbar_\lambda )$; if $Y$ is affine this is the coarsest topology such that all regular functions on $Y (\BQbar_\lambda )$ are continuous.

\begin{prop}  \label{p:cont_hom}
Let $E\subset\BQbar_\lambda$ be a subfield finite over $\BQ_\ell\,$, $H$ an affine algebraic group over $E$, and 
$f:\Pi\to H(\BQbar_\lambda)$ a continuous homomorphism. Then there is a field $E'\subset\BQbar_\lambda$ finite over $E$ such that 
$f(\Pi )\subset H(E')$.
\end{prop}

For completeness, we will give a proof. It is based on the following property of $\Pi$.

\begin{lem}   \label{l:top-fin-gen}
The maximal pro-$\ell$ quotient of any open subgroup of $\Pi$ is topologically finitely generated.
\end{lem}

\begin{proof}
Replacing $X$ be a finite etale cover, we can assume that the open subgroup equals $\Pi$. Then use finiteness of the group
$H^1(\Pi ,\BZ/\ell\BZ )=H^1_{\et} (X,\BZ/\ell\BZ )$.
\end{proof}

\begin{proof}[Proof of Proposition~\ref{p:cont_hom}]
There exists an open subgroup $U\subset H(\BQbar_\lambda )$ with the following property: for any $u\in U$ the elements $u^{\ell^n}$ converge to $1$. Then the homomorphism $f^{-1}(U)\to U$ factors through the maximal pro-$\ell$-quotient of $f^{-1}(U)$. The latter is topologically finitely generated by Lemma~\ref{l:top-fin-gen}. So $f(f^{-1}(U))\subset H(\tilde E)$ for some subfield $\tilde E\subset\BQ_\ell$ finite over $E$. Since 
$[\Pi:f^{-1}(U)]<\infty$, this implies that $f(\Pi )$ is  topologically finitely generated. The proposition follows.
\end{proof}

\subsection{The group $\hat\Pi_\lambda\,$}    \label{ss:hat-Pi-lambda}
Recall that (pro)-semisimple groups are \emph{not assumed to be connected.}

\subsubsection{The $\BQbar_\lambda$-pro-algebraic completion}   \label{sss:pro-algebraic}
In \S\ref{sss:0pro-algebraic} we defined the notion of $\BQ_\ell$-pro-algebraic completion.
Replacing $\BQ_\ell$ by $\BQbar_\lambda$ in this definition, one gets a similar notion of \emph{$\BQbar_\lambda$-pro-algebraic completion} of $\Pi$. It is easy to check that the $\BQbar_\lambda$-pro-algebraic completion of $\Pi$ is obtained from the $\BQ_\ell$-pro-algebraic completion by base change (use Proposition~\ref{p:cont_hom} and the Weil restriction functor $\Res_{E/\BQ_\ell}$ for finite extensions $E\supset\BQ_\ell$).

\subsubsection{The group $\hat\Pi_\ell\,$}   \label{sss:hat-Pi-ell}
We defined $\hat\Pi_\ell\,$ to be the $\ell$-adic pro-semisimple completion of $\Pi$. A definition of what this means was given in \S\ref{sss:thecompletion}. Equivalently, $\hat\Pi_\ell$ is the maximal pro-semisimple quotient of the $\BQ_\ell$-pro-algebraic completion of $\Pi$. 

\subsubsection{The group $\hat\Pi_\lambda\,$}  \label{sss:hat-Pi-lambda}
We will work with the group $\hat\Pi_\lambda\,$, which can be defined in three equivalent ways:

(i) $\hat\Pi_\lambda:=\hat\Pi_\ell\otimes_{\BQ_\ell}\BQbar_\lambda\,$;

(ii) $\hat\Pi_\lambda$ is the maximal semisimple quotient of the $\BQbar_\lambda$-pro-algebraic completion of $\Pi$;

(iii)  $\hat\Pi_\lambda$ is the projective limit of all semisimple algebraic groups $G$ over $\BQbar_\lambda$ equipped with a continuous homomorphism $\rho :\Pi\to G(\BQbar_\lambda)$ such that $\rho (\Pi )$ is Zariski-dense in $G$ (this is similar to the definition of $\hat\Pi_\ell$ from \S\ref{sss:thecompletion}).

By definition,we have a canonical continuous homomorphism $\Pi\to\hat\Pi_\lambda (\BQbar_\lambda )$, whose image is contained in 
$\hat\Pi_\ell (\BQ_\ell )$.

\begin{prop}  \label{p:2simply connected}
$\hat\Pi_\lambda^\circ$ is simply connected. 
\end{prop}

\begin{proof}
By Proposition~\ref{p:simply connected} and  \S\ref{sss:hat-Pi-ell}, $\hat\Pi_\ell^\circ$ is simply connected. By  \S\ref{sss:hat-Pi-lambda}(i), this implies that $\hat\Pi_\lambda$ is simply connected.
\end{proof}

\begin{prop}   \label{p:subgr of fin index}
Let $\Pi'\subset\Pi$ be an open subgroup. Then the canonical homomorphism $\hat\Pi'_\lambda\to\hat\Pi_\lambda$ is injective, and its image is equal to $\hat\Pi_\lambda\times_{\Pi}\Pi'$. 

\end{prop}

\begin{proof}
Follows from Lemma~\ref{l:subgr of fin index}, \S\ref{sss:hat-Pi-ell}, and \S\ref{sss:hat-Pi-lambda}(i).
\end{proof}

\subsection{A rigidity property of $\hat\Pi_{(\lambda )}\,$}  \label{ss:rigidity}
In \S\ref{sss:lambda-completion} we defined $\hat\Pi_{(\lambda )}$ to be the image of $\hat\Pi_\lambda$ in the ``coarse" category $\Pross (\BQbar_{\lambda})=\Pross (\BQbar )$. Recall that $[\hat\Pi_{(\lambda )}]$ denotes the GIT quotient of $\hat\Pi_{(\lambda )}$ by the conjugation action of 
the neutral connected component $\hat\Pi_{(\lambda )}^\circ\,$.

\begin{prop}   \label{p:2rigidity}
$\hat\Pi_{(\lambda )}\,$ has no nontrivial automorphisms inducing the identity on $[\hat\Pi_{(\lambda )}]$.
\end{prop}

\begin{proof}
By Proposition~\ref{p:2simply connected}, $\hat\Pi_{(\lambda )}^\circ$ is simply connected. So it is a product of almost-simple groups.
It remains to apply Proposition~\ref{p:rigidity}.
\end{proof}

Note that Proposition~\ref{p:2rigidity} immediately implies the uniqueness statement of our main 
Theorem~\ref{t:main}.

\subsection{The pro-reductive completion $\hat\Pi_\lambda^{\red}\,$}  \label{ss:pro-reductive}
Replacing in \S\ref{sss:hat-Pi-ell}-\ref{sss:hat-Pi-lambda} the word ``semisimple" by ``reductive", one defines the pro-reductive group scheme
$\hat\Pi_\lambda^{\red}$ over $\BQbar_\lambda\,$, which we call the $\lambda$-adic pro-reductive completion of $\Pi$. In this subsection (which will be used only in \S\ref{s:mot}) we express $\hat\Pi_\lambda^{\red}$ in terms of $\hat\Pi_\lambda$ by paraphrasing \cite[\S 1.3]{De}. The
result is formulated in Proposition~\ref{p:pro-reductive}.

\medskip

Let $G$ be a pro-reductive group scheme over an algebraically closed field. Let $A$ be an abelian group equipped with a homomorphism
\begin{equation}  \label{e:A to chars}
A\to\Hom (G,\BG_m).
\end{equation}
The homomorphism \eqref{e:A to chars} induces homomorphisms $G\to\Hom (A,\BG_m)$ and $\pi_0(G)\to\Hom (A_{\tors}\, ,\BG_m)$, where 
$A_{\tors}\subset A$ is the torsion subgroup. The homomorphism $G\to\pi_0 (G)$ factors through the maximal pro-semisimple quotient $G^{\Ss}$ of $G$. So  one gets homomorphisms $G^{\Ss}\to\pi_0 (G)\to\Hom (A_{\tors}\, ,\BG_m)$ and
\begin{equation}  \label{e:G to fibprod}
G\to G^{\Ss}\underset{\Hom (A_{\tors}\, ,\BG_m)}\times\Hom (A,\BG_m).
\end{equation}

\begin{lem}   \label{l:conditions that ensure}
The map \eqref{e:G to fibprod} is an isomorphism if the following conditions hold:

(a) $A/A_{\tors}$ is a $\BQ$-vector space;

(b) for any subgroup $G'\subset G$ of finite index, the map $A/A_{\tors}\to\Hom (G',\BG_m)/\Hom (G',\BG_m)_{\tors}$ induced by \eqref{e:A to chars} is an isomorphism.
\end{lem}

\begin{proof}
Let $Z$ denote the center of $G^\circ$. One has exact sequences
\[
0\to Z^\circ\to G\to G^{\Ss}\to 0 ,
\]
\[
0\to\Hom (A/A_{\tors}\, ,\BG_m)\to G^{\Ss}\underset{\Hom (A_{\tors}\, ,\BG_m)}\times\Hom (A,\BG_m)\to G^{\Ss}\to 0 .
\]
So the problem is to prove that the composition $Z^\circ\to G^\circ\to\Hom (A/A_{\tors}\, ,\BG_m)$ is an isomorphism. This is equivalent to showing that the composition
\[
A/A_{\tors}\overset{f}\longrightarrow\Hom (G^\circ ,\BG_m)\overset{g}\longrightarrow\Hom (Z^\circ ,\BG_m)
\]
is an isomorphism.

$\Hom (G^\circ ,\BG_m)$ is the direct limit of the groups $\Hom (G',\BG_m)$ corresponding to all subgroups $G'\subset G$ of finite index. So 
conditions (a)-(b) imply that $f$ is an isomorphism and $\Hom (G^\circ ,\BG_m)$ is a $\BQ$-vector space. On the other hand, it is clear that $g$ is injective and $\Coker g$ is torsion. Since $\Hom (G^\circ ,\BG_m)$ is a $\BQ$-vector space, this implies that $g$ is an isomorphism.
\end{proof}

Now let $G:=\hat\Pi_\lambda^{\red}$ and $A:=\BZbar_\lambda^\times$, where $\BZbar_\lambda$ is the ring of integers of $\BQbar_\lambda\,$.
Define the map \eqref{e:A to chars} to be the composition
\[
\BZbar_\lambda^\times\simeq\Hom_{\cont}(\hat\BZ ,\BQbar_\lambda^\times )\to\Hom_{\cont}(\Pi ,\BQbar_\lambda^\times )=
\Hom (\hat\Pi_\lambda^{\red},\BG_m ),
\]
where the map $\Hom_{\cont}(\hat\BZ ,\BQbar_\lambda^\times )\to\Hom_{\cont}(\Pi ,\BQbar_\lambda^\times )$ comes from the canonical homomorphism 
\begin{equation}  \label{e:to hat Z}
\Pi\to\Gal (\bar\BF_p/\BF_p)=\hat\BZ\, .
\end{equation}
Then $G^{\Ss}=\hat\Pi_\lambda\,$, $A_{\tors}$ is the subgroup $\mu_\infty (\BQbar )\subset\BQbar^\times$ formed by all roots of unity, and
\begin{equation}  \label{e:chars of Ators}
\Hom (A_{\tors}\, ,\BG_m)=\Hom (\mu_\infty (\BQbar )\, ,\BG_m)=\hat\BZ\, .
\end{equation}
So \eqref{e:G to fibprod} is a map
\begin{equation}  \label{e:our G to fibprod}
\hat\Pi_\lambda^{\red}\to \hat\Pi_\lambda\underset{\hat\BZ}\times\Hom (\BZbar_\lambda^\times ,\BG_m).
\end{equation}
Here the map $\hat\Pi_\lambda\to\hat\BZ$ is the composition $\hat\Pi_\lambda\epi\hat\Pi\to\hat\BZ$, where the second arrow is the map 
\eqref{e:to hat Z}; the map $\Hom (\BZbar_\lambda^\times ,\BG_m)\to\hat\BZ$ comes from \eqref{e:chars of Ators}. Note that 
$A=\BZbar_\lambda^\times$ is considered as an \emph{abstract} group (so the group $\Hom (\BZbar_\lambda^\times ,\BG_m)$ is \emph{huge}).

\begin{prop}  \label{p:pro-reductive}
The map \eqref{e:our G to fibprod} is an isomorphism.
\end{prop}

\begin{proof} 
It suffices to check that conditions (a)-(b) of Lemma~\ref{l:conditions that ensure} hold. This is clear for (a). By Proposition~\ref{p:subgr of fin index}, condition (b) essentially says that for any open subgroup $U\subset\Pi$ the map
\[
\Hom_{\cont}( \hat\BZ,\BQbar_\lambda^\times )/\Hom_{\cont}( \hat\BZ ,\BQbar_\lambda^\times )_{\tors}\to
\Hom_{\cont}( U,\BQbar_\lambda^\times )/\Hom_{\cont}(U ,\BQbar_\lambda^\times )_{\tors}
\]
induced by \eqref{e:to hat Z} is an isomorphism. This follows from finiteness of $\Coker (U\to\hat\BZ )$ combined with
Theorem~1.3.1 of \cite{De}, which says that the non-$p$ part of the quotient $\Ker (U\to\hat\BZ )/[U,U]$ is finite.
\end{proof}

\subsection{The category of representations of \,$\hat\Pi_\ell\,$ and \,$\hat\Pi_\lambda\,$}  \label{sss:rep hat-Pi-lambda}
By a $\lambda$-adic representation of $\Pi$ we mean a finite-dimensional vector space $V$ over $\BQbar_\lambda$ equipped with a continuous homomorphism $\Pi\to GL(V)$. A $\lambda$-adic representation of $\Pi$ is essentially the same as a lisse $\BQbar_\lambda$-sheaf on $X$.

The Tannakian category of finite-dimensional representations of $\hat\Pi_\lambda$ over $\BQbar_\lambda$ identifies with a Tannakian subcategory\footnote{By a \emph{Tannakian subcategory} of a Tannakian category $\cT$ we mean a strictly full subcategory $\cT'\subset\cT$ stable under tensor products, direct sums, dualization, and passing to subobjects} $\cT_\lambda (X)$ of the category of lisse $\BQbar_\lambda$-sheaves on $X$. It is easy to see that a lisse $\BQbar_\lambda$-sheaf $\E$ on $X$ is in
$\cT_\lambda (X)$ if and only if it is semisimple and for every connected finite etale covering $\pi :X'\to X$, the determinant of each of the irreducible components of $\pi^*\E$ has finite order. According to the next proposition, it is enough to check the latter property when $\pi$ is an isomorphism.

\begin{prop}  \label{p:Weil-2}
Let $\pi :X'\to X$ be a connected finite etale covering.  Let $\E$ be an irreducible lisse $\BQbar_\lambda$-sheaf on $X$ and $\E'\subset\pi^*\E$ an irreducible subsheaf. If $\det\E$ has finite order then so does $\det\E'$.
\end{prop}

\begin{proof}
We can assume that $\pi$ is a Galois covering; let $G$ be its Galois group. The sheaf $$\underset{\sigma\in G}\otimes\sigma^*\det\E'$$ has finite order because it is isomorphic to some tensor power of $\det\pi^*\E=\pi^*\det\E$. To deduce from this that $\det\E'$ has finite order, use that any rank 1 $\lambda$-adic local system on $X'$ can be written as $\cA\otimes\cB$, where $\cA$ is a rank 1 local system of finite order and $\cB$ is a pullback of a rank 1 local system on $\Spec\BF_p\,$; this is essentially \cite[Prop.~1.3.4]{De} (the only difference is that Proposition 1.3.4 of
\cite{De} is about Weil sheaves, but the proof remains the same).
\end{proof}

Let $\cT_\ell (X)$ denote the tensor category of  lisse $\BQ_\ell$-sheaves $\E$ on $X$ such that 
$\E\otimes_{\BQ_\ell}\BQbar_\lambda\in \cT_\lambda (X)$. It is easy to see that all objects of $\cT_\ell (X)$ are semisimple.
The Tannakian category of finite-dimensional representations of $\hat\Pi_\ell$ over $\BQ_\ell$ identifies with $\cT_\ell (X)$.

\section{A theorem in the spirit of Kazhdan-Larsen-Varshavsky \cite{KLV}}   \label{s:KLV}
In this section we fix an algebraically closed field $E$ of characteristic $0$; all schemes will be over $E$. We will write $\Pross$ instead of $\Pross (E)$.

\subsection{Recollections on Grothendieck semirings}
Given a group scheme $G$, let $\Rep (G)$ denote the category of its finite-dimensional representations. Let $K(G)$ denote the Grothendieck ring of $\Rep (G)$ and $K^+(G)\subset K(G)$ the Grothendieck semiring. If $G$ is pro-reductive then considering elements of $K(G)$ as characters, one identifies $K(G)\otimes E$ with the algebra of conjugation-invariant regular functions on $G$.

\begin{lem}   \label{l:normal subgr}
Let $G$ be pro-reductive. 

(i) The assignment $H\mapsto K^+(G/H)$ defines a bijection between the set of normal subgroups $H\subset G$ and the set of subsemirings $A\subset K^+(G)$ with the following property: if $a_1,a_2\in K^+(G)$ and $a_1+a_2\in A$ then $a_1,a_2\in A$.

(ii) $G$ has finite type if and only if the ring $K(G)$ is finitely generated.
\end{lem}

\begin{proof}
 It is well known that for any group scheme $G$, the assignment $H\mapsto \Rep (G/H)$ defines a bijection between the set of normal subgroups $H\subset G$ and the set of those Serre subcategories of $\Rep (G)$ that are closed under tensor products. Statement (i) follows because if $G$ is pro-reductive the category $\Rep (G)$ is semisimple.

Let us prove the ``if" part of statement (ii). Suppose that the ring $K(G)$ is generated by $u_1,\ldots ,u_n\,$. Then $G$ has a quotient $G'$ of finite type such that $u_i\in K(G')\subset K(G)$ for all $i$. So $K(G')=K(G)$. By statement (i), this implies that $G'=G$.
\end{proof}

\subsection{The result of \cite{KLV}}
\begin{prop}  \label{p:TannakaKLV}
Let $G$ and $G'$ be connected pro-reductive groups. Then any semiring isomorphism $\phi :K^+(G)\iso K^+(G')$ is induced by an isomorphism $f:G\iso G'\,$, which is unique up to conjugation by elements of $G'(E)$.
\end{prop}

\begin{proof}
If $G$ and $G'$ are reductive (rather than pro-reductive) this is \cite[Thm.~1.2]{KLV}.

In general, note that by Lemma~\ref{l:normal subgr},  $\phi$ induces an isomorphism between the directed set $I$ parametrizing all quotients of $G$ having finite type and a similar directed set for $G'$. For $i\in I$, let $G_i$ denote the corresponding quotient of $G$ and let $G'_i$ denote the corresponding quotient of $G'$. Choose a pinning in $G$ and $G'$; it induces a pinning in each of the groups $G_i$ and $G'_i\,$. By \cite[Thm.~1.2]{KLV}, the isomorphism $K^+(G_i)\iso K^+(G'_i)$ induced by $\phi$ comes from a unique isomorphism $f_i :G_i\iso G'_i$ compatible with the pinnings. The isomorphisms $f_i$ are compatible with each other and define an isomorphism $f:G\iso G'$ compatible with the pinnings and inducing $\phi$. Clearly this is the unique isomorphism $f$ with these properties. If one does not require compatibility with the pinnings one gets uniqueness up to conjugation.
\end{proof}

The connectedness assumption in Proposition~\ref{p:TannakaKLV} is important. First of all,  Proposition~\ref{p:TannakaKLV} does not hold for finite groups (e.g., see \cite{EG}). One also has to keep in mind the following example.

\begin{ex}   \label{ex:why Adams}
Set $H:=SL(2n+1)$, $n\ge 1$. Equip $H$ with the action of $\BZ/2\BZ$ such that $1\in\BZ/2\BZ$ acts on $H$ as $h\mapsto (h^t)^{-1}$.
Set $G:=(\BZ/2\BZ)\ltimes H$. Let $Z$ denote the center of $H$. In \S\ref{sss:explaining Example} we will construct a nontrivial semiring automorphism $\phi$ of $K^+(G)$; moreover, $\phi$ induces a nontrivial automorphism of $K^+(G/Z)\subset K^+(G)$. On the other hand, it is easy to see that any automorphism of $G$ or $G/Z$ is inner, so the corresponding automorphism of the Grothendieck semiring is trivial.  Let us note that $\phi$ does not commute with the exterior square operation $\lambda^2:K^+(G)\to K^+(G)$, see Lemma~\ref{l:why Adams}(ii).
\end{ex}

\subsection{Formulation of the theorem}  \label{ss:form_thm}
To remove the connectedness assumption from Proposition~\ref{p:TannakaKLV},  we consider the problem of reconstructing a pro-reductive group $G$ from the pro-finite group $\Gamma=\pi_0 (G)$, the collection of lambda-semirings $K^+(G\times_{\Gamma}U)$, and the natural homomorphisms between them; here $U$ runs through the set of open subgroups of $\Gamma$. (Example~\ref{ex:why Adams} explains why we have to consider lambda-semirings rather than plain semirings.)

\subsubsection{Lambda-semirings} \label{sss:Lambda}
By a cancellation semiring we mean a semiring $A^+$ such that the equality $a+b=a+c$ implies that $b=c$; in this case the canonical homomorphism from $A^+$ to the corresponding ring $A$ is injective. By a \emph{lambda-semiring} we will mean a pair consisting of a cancellation semiring $A^+$ and a lambda-structure on the corresponding ring $A$ such that $\lambda^i(A^+)\subset A^+$ for all $i\in\BN$.
For instance, $K^+(G)$ and $K(G)$ are lambda-semirings. (In fact, we care only about lambda-semirings of this type and about morphisms between them).  The category of lambda-semirings (resp.~ lambda-rings) will be denoted by $\Lambda^+$ (resp.~$\Lambda$). There is a functor $\Lambda^+\to\Lambda$ left adjoint to the embedding $\Lambda\mono\Lambda^+$.

\subsubsection{The groupoid $\Pross_\Gamma\,$}   \label{sss:ProssGamma}
If $\Gamma$ is a pro-finite group, let $\Pross_\Gamma$ denote the groupoid (or category?) of pairs consisting of an object $G\in\Pross$ and an isomorphism $\pi_0 (G)\iso\Gamma$.

\subsubsection{The category $\Open_\Gamma$ and the functors $\tilde\BK^+_G\,$}   \label{sss:Open}
From now on we fix a pro-finite group $\Gamma$.

Let $\Open_\Gamma\,$ denote the category whose objects are open subgroups $U\subset\Gamma$ and morphisms $U_1\to U_2$ are $\Gamma$-equivariant maps $\Gamma/U_2\to\Gamma/U_1$. Such a map is the same as an element 
$\gamma\in \Gamma/U_1$ such that 
\begin{equation}  \label{e:Op-morphisms}
\gamma^{-1}U_2\gamma\subset U_1\,.
\end{equation}
Note that the automorphism group of any $U\in\Open_\Gamma$ equals $N(U)/U$, where $N(U)$ is the normalizer.

Any $G\in\Pross_\Gamma\,$ defines a functor $\tilde\BK^+_G :\Open_\Gamma\to\Lambda^+$, where $\Lambda^+$ is as in \S\ref{sss:Lambda}. 
Namely, we set 
$\tilde\BK^+_G(U):=K^+(G\times_\Gamma U)$; given $U_1\,$, $U_2\,$, and $\gamma$ satisfying \eqref{e:Op-morphisms}, we define the morphism
\begin{equation}   \label{e:induced by conj}
K^+(G\times_\Gamma U_1)\to K^+(G\times_\Gamma U_2)
\end{equation}
 to be the one corresponding to the homomorphism
\[
G\times_\Gamma U_2\to G\times_\Gamma U_1\, , \quad x\mapsto \gamma^{-1}x\gamma\, .
\]

\begin{rem}
Here is an equivalent way to define the map \eqref{e:induced by conj}. For $U\in \Open_\Gamma$ let $A(U)$ denote the algebra of functions on $\Gamma/U$ viewed as an algebra object in $\Rep(G)$. Let $A(U)\mbox{-mod}$ denote the category of $A(U)$-modules in $\Rep(G)$. 
One has a canonical equivalence 
\[
A(U)\mbox{-mod}\iso\Rep(G\times_\Gamma U).
\]
A morphism $U_1\to U_2$ in the category $\Open_\Gamma$ is the same as an algebra morphism $A(U_1)\to A(U_2)$ in $\Rep(G)$. The latter defines the base change functor $A(U_1)\mbox{-mod}\to A(U_2)\mbox{-mod}$. 
The corresponding map between Grothendieck semirings is the map \eqref{e:induced by conj}.
\end{rem}

\subsubsection{The groupoid $\Lambda^+_\Gamma$ and the functor $\BK^+$}  \label{sss:BK}
Let $\tilde\Lambda^+_\Gamma$ denote the category of all functors $\Open_\Gamma\to\Lambda^+$, where $\Lambda^+$ is as in \S\ref{sss:Lambda}. In \S\ref{sss:Open} we defined $\tilde\BK^+_G\in\tilde\Lambda^+_\Gamma\,$ for any $G\in\Pross_\Gamma\,$; in particular, we have $\tilde\BK^+_\Gamma\in\tilde\Lambda^+_\Gamma\,$.

Objects $F\in\tilde\Lambda^+_\Gamma$ equipped with a morphism $\tilde\BK^+_\Gamma\to F$ form a category. 
Let  $\Lambda^+_\Gamma$ denote the groupoid obtained from this category by removing its non-invertible morphisms.

Similarly (replacing semirings by rings), one defines $\tilde\Lambda_\Gamma$ and $\Lambda_\Gamma\,$.

For any $G\in\Pross_\Gamma$ define $\BK^+_G\in\Lambda^+_\Gamma$ to be the object $\tilde\BK^+_G\in\tilde\Lambda^+_\Gamma$ equipped with the morphism $\tilde\BK^+_\Gamma\to \tilde\BK^+_G$ induced by the canonical epimorphism $G\epi\Gamma$. Thus we get a functor
\begin{equation}   \label{e:BK}
\Pross_\Gamma\to \Lambda^+_\Gamma \, , \quad G\mapsto \BK^+_G\, .
\end{equation}

\subsubsection{Formulation of the theorem}  \label{sss:form_thm}
Let $\Pross_\Gamma^{prod}\subset\Pross_\Gamma$ denote the full subcategory of those $G\in\Pross_\Gamma$ for which $G^\circ$ is a product of almost-simple groups. 

Define the full subcategory $\PROSS_\Gamma\subset\Pross_\Gamma^{prod}$ as follows. 
Recall that by \S\ref{sss:class in H^2}, any $G\in\Pross_\Gamma$ defines an element $\nu\in H^2 (\Gamma ,Z )$, where $Z$ is the center of $G^\circ$. We say that $G\in\PROSS_\Gamma$ if $G\in\Pross_\Gamma^{prod}$ and the corresponding $\nu\in H^2 (\Gamma ,Z )$ has the following property: for any open subgroup $U\subset\Gamma$ and any $f\in \Hom (Z,E^\times )^U$ the element $f_*(\nu )\in H^2 (U,E^\times )$ equals $0$.

\begin{theorem}   \label{t:variant of KLV}
(i) The restriction of the functor~\eqref{e:BK} to  $\Pross_\Gamma^{prod}$ is faithful.

(ii) The restriction of the functor~\eqref{e:BK} to  $\PROSS_\Gamma$ is fully faithful.

(iii) Let $G,G'\in\Pross_\Gamma$ and suppose that the images of $G$ and $G'$ under the functor~\eqref{e:BK} are isomorphic. If one of the objects $G,G'$ belongs to $\Pross_\Gamma^{prod}$ (resp.~$\PROSS_\Gamma$) then so does the other.
\end{theorem}

The rest of \S\ref{s:KLV} is devoted to the proof of the theorem. 
Statement (i) is proved at the end of \S\ref{ss:reconstructing [G]}. 
Statement (ii) is proved in \S\ref{ss:gen case}.
Statement (iii) is proved in \S\ref{ss:Proof of iii}.

\begin{rem}
The definition of $\PROSS_\Gamma$ is somewhat technical. However, note that $\PROSS_\Gamma$ contains all 
$G\in
\Pross_\Gamma$ such that  $G^\circ$ is adjoint. Also note that if the strict cohomological dimension of $\Gamma$ is $\le 2$ then 
$\PROSS_\Gamma =\Pross_\Gamma^{prod}$. Indeed, for any open subgroup $U\subset\Gamma$ the group $H^2 (U,E^\times )$ is   isomorphic 
(non-canonically) to $H^2 (U,\BQ/\BZ )=H^3 (U,\BZ )$.
\end{rem}

\subsection{Proof of Theorem~\ref{t:variant of KLV}(iii)}   \label{ss:Proof of iii}
Let $G\in\Pross_\Gamma\,$. We have to show that knowing the corresponding functor~\eqref{e:BK} is enough to determine whether 
$G$ belongs to $\Pross_\Gamma^{prod}$ (resp.~$\PROSS_\Gamma$). For $\Pross_\Gamma^{prod}$ this immediately follows from 
Proposition~\ref{p:TannakaKLV}. For $\PROSS_\Gamma$ one also uses the equivalence (i)$\Leftrightarrow$(iii) in the following lemma:

\begin{lem}
Let $G\in\Pross_\Gamma\,$. Let $Z$ and $T$ be the center and the maximal torus of $G^\circ$. Let $\nu\in H^2 (\Gamma ,Z )$ be as in 
\S\ref{sss:class in H^2}. Let $U\subset\Gamma$ be an open subgroup. Then the following are equivalent:

(i) for every $f\in \Hom (Z,E^\times )^U$ the element $f_*(\nu )\in H^2 (U,E^\times )$ equals $0$;

(ii) for every dominant weight $\omega\in \Hom (T,\BG_m )^U$ one has $f_*(\nu )=0$, where $f\in \Hom (Z,E^\times )$  is the restriction of $\omega$;

(iii) every irreducible representation of $G^\circ$ whose isomorphism class is $U$-invariant can be extended to a representation of $G\times_\Gamma U$ in the same vector space.
\end{lem}

\begin{proof}
To prove that (ii)$\Rightarrow$(i), it suffices to note that every $f\in \Hom (Z,E^\times )^U$ is a restriction of some dominant $\omega\in \Hom (T,\BG_m )^U$ (this is a pro-version of Lemma~\ref{l:minuscule}). The other implications are clear.
\end{proof}

\subsection{Reconstructing $[G]$ from $\BK^+_G\,$}   \label{ss:reconstructing [G]}
As mentioned in \S\ref{sss:group schemes pro}, a pro-finite group can be considered as a group scheme over $E$. Consider $\Gamma$ as a scheme over $E$ equipped with an action of $\Gamma\times\BN$, where $\BN$ is the multiplicative monoid of positive integers ($\Gamma$ acts by conjugation and $n\in\BN$ acts as raising to the power of $n$). Let $\Aff_\Gamma$ denote the category of affine schemes $Z$ over~$\Gamma$ equipped with a lift of the above action of $\Gamma\times\BN$ to an action on $Z$.

\begin{lem}
The functor 
$$\Pross_\Gamma\to\Aff_\Gamma, \quad G\mapsto [G]$$
admits a factorization
$\Pross_\Gamma\to\Lambda_\Gamma\to\Aff_\Gamma \, $,
where the first functor is the composition of the functor \eqref{e:BK} and the natural functor $\Lambda^+_\Gamma\to\Lambda_\Gamma\,$.
\end{lem}

Before proving the lemma, let us recall that for any lambda-ring $A$ one has the Adams operations $\psi^n :A\to A$ defined as follows:
$\psi^n(a)$ is  the coefficient of $t^n$ in the formal series 
\[
-t \frac{d}{dt}\log (1+ \sum_{n> 0}\lambda^n(a) (-t)^n )\in A[[t]].
\]
Recall that each $\psi^n$ is a ring endomorphism, and one has $\psi^n\circ\psi^m=\psi^{mn}$, $\psi^1=\id_A\,$. If $H$ is a group, $\rho$ is an element of the Grothenideck ring of the category of finite-dimensional representations of $H$, and $\chi_{\rho} :H\to E$ is its character, then $\chi_{\psi^n({\rho} )} (h)=\chi_{\rho}(h^n)$. 
 
\begin{proof}[Proof of the lemma]
For every $U\in\Open_\Gamma\,$, let $\underline{U}$ denote the spectrum of the algebra of conjugation-invariant locally constant functions on $U$; equivalently, $\underline{U}$ is the quotient of $U$ by the conjugation action of $U$.

Now let $F\in\Lambda_\Gamma\,$. The corresponding object $Y_F\in\Aff_\Gamma$ is defined as follows.
For every $U\in\Open_\Gamma$ set $Z_{F,U}:=\Spec F(U)\otimes E$; this is an affine scheme over $\underline{U}$ equipped with an $\BN$-action (the latter is induced by the Adams operations on $F(U)\otimes E$). For every normal open subgroup $V\subset\Gamma$ and $\alpha\in\Gamma/V$ let $U_\alpha\subset\Gamma$ denote the preimage of the cyclic subgroup 
$\langle\alpha\rangle\subset\Gamma/V$ and let $Y_{V,\alpha}$ denote the preimage of $\alpha$ with respect to the map
$Z_{F,U_\alpha}\to\underline{U_\alpha}\to\langle\alpha\rangle$. Let $Y_{F,V}$ denote the disjoint union of $Y_{F,V,\alpha}\,$, $\alpha\in\Gamma/V$. If $V'\subset V$ then for each $\alpha'\in\Gamma/V'$ such that $\alpha'\mapsto\alpha$ we get a morphism $Y_{F,V',\alpha'}\to Y_{F,V,\alpha}\,$. So we get a morphism $Y_{F,V'}\to Y_{F,V}\,$. Finally, define $Y_F$ to be the projective limit of $Y_{F,V}\,$.

If $F$ corresponds to $G\in\Pross_\Gamma$ then $Y_{F,V}$ identifies with the quotient of $[G]$ by the conjugation action of~$V$ (indeed, 
if $f$ is a function on $G$ invariant under $V$-conjugation then the restriction of $f$ to $U_\alpha$ is invariant under $U_\alpha$-conjugation). So $Y_F$ identifies with $[G]$.
\end{proof}

\begin{cor}  \label{c:in terms of [G]}
Let $G_1,G_2\in\Pross_\Gamma\,$. Then an isomorphism $\BK^+_{G_1}\iso\BK^+_{G_2}$ is the same as a $(\Gamma\times\BN)$-equivariant isomorphism $[G_1]\iso [G_2]$ of schemes over $\Gamma$ with the following property: for any open subgroup $U\subset\Gamma$ a function on 
$[G_1]\times_\Gamma U$ is an irreducible character of $G_1\times_\Gamma U$ if and only if the corresponding function on 
$[G_2]\times_\Gamma U$ is an irreducible character of $G_2\times_\Gamma U$.
\qed
\end{cor}

\begin{proof}[Proof of Theorem~\ref{t:variant of KLV}(i)]
Let $G\in\Pross_\Gamma\,$. Suppose that a $\Pross_\Gamma$-automorphism of $G$ induces the identity on $\BK^+_G\,$. Then it 
induces the identity on $[G]$ by Corollary~\ref{c:in terms of [G]}. It remains to apply Proposition~\ref{p:rigidity}.
\end{proof}

The rest of \S\ref{s:KLV} is devoted to the proof of Theorem~\ref{t:variant of KLV}(ii).

\subsection{Reduction of Theorem~\ref{t:variant of KLV}(ii) to the finite type case}   \label{ss:easy reduction}
For any open normal subgroup $U\subset\Gamma$ we have a functor
\begin{equation}   \label{e:2BK}
\Pross_{\Gamma /U}\to \Lambda^+_{\Gamma /U} \, , \quad G\mapsto \BK^+_G
\end{equation}
similar to \eqref{e:BK}. Let $\PROSS_{\Gamma /U}^{fin}\subset\Pross_\Gamma$ be the full subcategory of those objects of 
$\PROSS_{\Gamma /U}$ which are schemes of finite type.

\begin{prop}   \label{p:easy reduction}
Suppose that the restriction of the functor \eqref{e:2BK} to $\PROSS_{\Gamma /U}^{fin}$ is fully faithful for every open normal subgroup 
$U\subset\Gamma$. Then the restriction of the functor \eqref{e:BK} to $\PROSS_\Gamma$ is fully faithful.
\end{prop}

For $G\in\Pross_\Gamma\,$, let $\Norm (G)$ denote the set of normal subgroups $H\subset G$ such that $G/H$ has finite type. Let  $\NORM (G)$ denote the set of all $H\in\Norm (G)$ such that $G/H\in\PROSS_{\Gamma /U}\,$, where $U:=\im (H\to\Gamma )$.
To prove the proposition, we need the following lemma.

\begin{lem}   \label{l:cofinality}
Let $G\in\PROSS_\Gamma$. Then for every $H\in\Norm (G)$ there exists $\tilde H\in\NORM (G)$ such that~$\tilde H\subset H$.
\end{lem}

\begin{proof}
Find a subgroup $H'\subset H\cap G^\circ$ normal in $G$ such that $G^\circ/H$ is a finite product of almost simple groups. Then take a small enough $\tilde H\in\Norm (G)$ such that $\tilde H\cap G^\circ=H'$. 
\end{proof}

\begin{proof}[Proof of Proposition~\ref{p:easy reduction}]
Let $G_1,G_2\in\PROSS_\Gamma\,$. We have to show that each isomorphism $$\varphi :\BK^+_{G_1}\iso \BK^+_{G_2}$$ comes from a unique $\Pross_\Gamma$-isomorphism $G_1\iso G_2$. One constructs it as follows. By Lemma~\ref{l:normal subgr}, $\varphi$ defines an order-preserving bijection $\Norm (G_1)\iso \Norm (G_2)$. By Theorem~\ref{t:variant of KLV}(iii), the image of $\NORM (G_1)$ under this bijection equals $\NORM (G_2)$. If $H_1\in\NORM (G_1)$ and $H_2\in\NORM (G_2)$ correspond to each other then $\varphi$ induces  an isomorphism $\BK^+_{G_1/H_1}\iso \BK^+_{G_2/H_2}\,$. By the assumption of Proposition~\ref{p:easy reduction}, this is the same as a $\PROSS_{\Gamma/U}$-isomorphism 
\begin{equation}   \label{e:iso_H}
G_1/H_1\iso G_2/H_2\, ,
\end{equation}
where $U:=\im (H_1\to\Gamma )=\im (H_2\to\Gamma )$. By Lemma~\ref{l:cofinality} and Proposition~\ref{p:coarse epi}, the isomorphisms \eqref{e:iso_H} for all possible $H_1\in\NORM (G_1)$ define a $\PROSS_\Gamma$-isomorphism $G_1\iso G_2$.
\end{proof}

\subsection{The adjoint case}   \label{ss:adjoint case}
Let $\Pross_\Gamma^{ad}\subset\Pross_\Gamma$ be the full subcategory of those $G\in\Pross_\Gamma$ for which $G^\circ$ has trivial center. It is clear that $\Pross_\Gamma^{ad}\subset\PROSS_\Gamma\,$.
In this section we will prove the following particular case of Theorem~\ref{t:variant of KLV}(ii).

\begin{prop}  \label{p:2adjoint case}
The restriction of the functor~\eqref{e:BK} to  $\Pross_\Gamma^{ad}$ is fully faithful.
\end{prop}

\subsubsection{Reduction to Proposition~\ref{p:adjoint case}}   \label{sss:reduction to}
Let $G_1\, ,G_2\in\Pross_\Gamma^{ad}$. We have to prove that any isomorphism
\begin{equation} \label{e:given_iso}
\BK^+_{G_1}\iso \BK^+_{G_2}
\end{equation} 
comes from a unique $\Pross_\Gamma$-isomorphism $G_1\iso G_2\,$. By Proposition~\ref{p:easy reduction}, we can assume that $\Gamma$ is finite and $G_1^\circ\, , G_2^\circ$ have finite type.

As a part of the data defining the isomorphism \eqref{e:given_iso}, we have a $\Gamma$-equivariant isomorphism 
$K^+(G_1^\circ )\iso K^+(G_2^\circ )$. By Proposition~\ref{p:TannakaKLV}, it comes from a unique $\Gamma$-equivariant $\Pross$-isomorphism $G_1^\circ \iso G_2^\circ$. Since $G_1^\circ$ and $G_2^\circ$ are adjoint, this is the same as a $\Gamma$-equivariant isomorphism 
$f:\Delta_1\iso\Delta_2\,$, where $\Delta_i$ is the Dynkin diagram of $G_i\,$. Again using the adjointness assumption, we see that $f$ comes from a unique $\Pross_\Gamma$-isomorphism $\varphi :G_1\iso G_2$. It remains to show that $\varphi$ induces the given isomorphism \eqref{e:given_iso}. Equivalently, we have to show that any automorphism of $\BK^+_{G_1}$ inducing the identity on $K^+(G_1^\circ )$ is trivial.
By Corollary~\ref{c:in terms of [G]}, this is equivalent to the following proposition.

\begin{prop}  \label{p:adjoint case}
Let $\Gamma$ be a finite group and 
$G\in\Pross_\Gamma\,$. Assume that $G^\circ$ is an adjoint semisimple group. Let $f:[G]\iso [G]$ be a $(\Gamma\times\BN )$-equivariant isomorphism of schemes over $\Gamma$ such that

(i) $f$ induces the identity on $[G^\circ]$;

(ii) for any subgroup $U\subset\Gamma$, $f$ preserves the set of irreducible characters of $G\times_\Gamma U$.

Then $f$ is the identity map.
\end{prop}

The rest of \S\ref{ss:adjoint case} is devoted to the proof of Proposition~\ref{p:adjoint case}.

\subsubsection{Remark}  \label{sss:normal subgr}
Lemma~\ref{l:normal subgr}(i) implies that for every normal subgroup $H\subset G$ there is a normal subgroup $H'\subset G$ such that
$f:[G]\iso [G]$ induces an isomorphism $[G/H]\iso [G/H']$. It is clear that $H'$ is unique. Since $f$ induces the identity on $[G^\circ]$ one has
\[
H'\cap G^\circ =H\cap G^\circ.
\]
Since $f$ is a map over $\Gamma$ one has
\[
\im (H'\to\Gamma ) =\im (H'\to\Gamma ).
\]

\subsubsection{Easy reductions}
Let $\Delta$ denote the Dynkin diagram of $G^\circ$.

For each $\sigma\in\im(\Gamma\to\Aut\Delta)$ we have to show that $f:[G]\iso [G]$ acts trivially on the preimage of $\sigma$ in $[G]$. 
From now on, we fix $\sigma$.
We can assume that $\im(\Gamma\to\Aut\Delta)$ is generated by $\sigma$ (otherwise replace $G$ by the preimage of the cyclic subgroup generated by $\sigma$).

We can also assume that $\Gamma$ acts transitively on the set of connected components $\pi_0(\Delta )$. Otherwise, we have nontrivial subgroups $H_1\, ,H_2\subset G^\circ$ normal in $G$ such that $H_1\cap H_2=\{ 1\}$. By \S~\ref{sss:normal subgr}, the map $f:[G]\iso [G]$ induces isomorphisms $f_i:[G/H_i]\iso [G/H_i]$, $i\in\{1,2\}$, with properties similar to those of $f$. The map $[G]\to [G/H_1]\times [G/H_2]$ is injective, so to prove that $f$ is the identity, it suffices to show this for $f_1$ and $f_2\,$.

Set $\Gamma_0:=\Ker (\Gamma\to\Aut\Delta)$. Then there is a unique normal subgroup $H\subset G$ such that $H\cap G^\circ=\{1\}$ and 
$\im (H\to\Gamma )=\Gamma_0\,$. The subgroup $H'$ corresponding to $H$ by \S\ref{sss:normal subgr} has the same properties, so $H'=H$.
Thus to prove Proposition~\ref{p:adjoint case}, it suffices to prove a similar statement for $G/H$ instead of $G$ and $\Gamma/\Gamma_0$ instead of $\Gamma$. Therefore we can assume that $\Gamma_0=\{1\}$.

\emph{From now on we assume that $\Gamma\subset\Aut\Delta$ is the cyclic subgroup generated by $\sigma$ and that the action of $\Gamma$ on $\pi_0(\Delta )$ is transitive. }

\subsubsection{Some notation}
Fix a pinning of $G^\circ$ (in particular, we fix a maximal subtorus $T\subset  G^\circ$).
Then $\Gamma$ acts on $G^\circ$ preserving the pinning, and $G$ identifies with $\Gamma\ltimes G^\circ$.

The image of $g\in G$ in $[G]$ will be denoted by $[g]$. In particular, for any $t\in T$ we have $\sigma\cdot t\in \Gamma\ltimes G^\circ=G$ and 
$[\sigma\cdot t]\in [G]$.

\subsubsection{The key lemma}
Let $[G]_\sigma$ denote the preimage of $\sigma$ in $[G]$. Our goal is to show that the map $[G]_\sigma\to [G]_\sigma$ induced by $f$ is equal to the identity. The next lemma says that this is ``almost true" even if $f$ is not assumed to be $\BN$-equivariant (the only exception essentially comes from Example~\ref{ex:why Adams}).

\begin{lem}    \label{l:keylemma}
Let $f:[G]\iso [G]$ be as in Proposition~\ref{p:adjoint case} except that $(\Gamma\times\BN )$-equivariance is replaced by $\Gamma$-equivariance. Suppose that the map $[G]_\sigma\to [G]_\sigma$ induced by $f$ is \emph{not equal} to the identity. Then

(a) $\sigma$ has even order $2m$, and there exists a vertex $i\in\Delta$ connected to $\sigma^m (i)$ by an edge;

(b) for any such $i$ one has
\begin{equation}   \label{e:f([sigma])}
f([\sigma])=[\sigma\cdot \varepsilon_i],\quad \varepsilon_i:=\check\omega_i (-1)\,;
\end{equation}
where $\check\omega_i :\BG_m\to T$ is the $i$-th fundamental coweight.
\end{lem}

Note that the situation described in Lemma~\ref{l:keylemma}(a) is essentially unique (namely, each connected component of $\Gamma$ has type $A_{2n}$ and the kernel of the action of $\Gamma$ on $\pi_0 (\Delta )$ equals $\{1,\sigma^m\}$). 

Assuming the lemma, one finishes the proof of Proposition~\ref{p:adjoint case} as follows. We have to show that if $f$ is $\BN$-equivariant then the situation of Lemma~\ref{l:keylemma} is impossible. Indeed, combining $\BN$-equivariance with \eqref{e:f([sigma])}, we would get
$[(\sigma\cdot \varepsilon_i)^{2m}]=f([\sigma^{2m}])=f([1])=[1]$, so $(\sigma\cdot \varepsilon_i)^{2m}=1$ or equivalently,
\begin{equation}   \label{e:impossible}
\prod_{j=0}^{2m-1}\sigma^j (\varepsilon_i )=1.
\end{equation}
On the other hand, since the situation described in Lemma~\ref{l:keylemma}(a) is essentially unique, it is easy to check that 
\eqref{e:impossible} does not hold. Thus it remains to prove  the lemma.

\begin{proof}[Proof of Lemma~\ref{l:keylemma}]
Let $\Hom^\sigma (T ,\BG_m)$ denote the group of $\sigma$-invariant weights of $G$.
For any dominant 
$\omega\in\Hom^\sigma (T , ,\BG_m)$, let $V_\omega$ be the unique (up to isomorphism) representation of 
$G=\Gamma\ltimes G^\circ$ such that $V_\omega$ is irreducible as a $G^\circ$-module and the action of $\Gamma$ on the highest line of $V_\omega$ is trivial (the notion of highest line makes sense because we chose a pinning of $G^\circ$). 

Let $\chi_\omega$ denote the character of $V_\omega\,$. Then $\chi_\omega\circ f$ is an irreducible character of $G$ whose restriction to $G^\circ$ is equal to that of $\chi_\omega\,$. So $\chi_\omega\circ f=\chi_\omega\cdot\beta_\omega$ for some 
$\beta_\omega\in\Hom (\Gamma ,\BG_m )$. 

For any dominant $\omega,\omega',\omega''\in\Hom^\sigma (T , ,\BG_m)$ we have the $\Gamma$-module 
$\Hom_{G^\circ} (V_{\omega''\,},V_{\omega}\otimes V_{\omega'})$. It is clear that
\begin{equation}  \label{e:nonzero trace}
\beta_{\omega''}=\beta_{\omega}\cdot\beta_{\omega'} \quad\mbox{ whenever }\Tr (\sigma,\Hom_{G^\circ} (V_{\omega''\,},V_{\omega}\otimes V_{\omega'}))\ne 0.
\end{equation}
 In particular,
\begin{equation}  \label{e:beta is multiplicative}
\beta_{\omega+\omega'}=\beta_{\omega}\cdot\beta_{\omega'}\,. 
\end{equation}

If $\eta\in\Hom^\sigma (T ,\BG_m)$ is fixed and $\omega,\omega'\in\Hom^\sigma (T ,\BG_m)$ are dominant enough then one has a
$\Gamma$-equivariant isomorphism
\begin{equation} \label{e:need reference}
\Hom_{G^\circ} (V_{\omega+\omega'+\eta\,},V_{\omega}\otimes V_{\omega'})\iso (U\fn_-)_\eta \, ,
\end{equation}
where the Lie algebra $\fn_-$ is defined using the pinning and $(U\fn_-)_\eta$ is the graded component corresponding to $\eta$. The construction of the isomorphism is given, e.g., in \cite[\S 1]{PY}, which goes back to \cite[\S 2.2]{PRV}.

Now let $O$ be an orbit of $\Gamma$ in the set of vertices of $\Delta$ and let $\alpha_O$ denote the sum of the simple roots of $G^\circ$ corresponding to the elements of $O$. Say that $O$ is \emph{exceptional\,} if there exist two vertices of $O$ connected by an edge. It is easy to check that $\Tr (\sigma,(U\fn_-)_\eta)\ne 0$ if either $\eta= -\alpha_O$ with $O$ non-exceptional or $\eta= -2\alpha_O$ with $O$ exceptional. Combining this with \eqref{e:need reference} and \eqref{e:nonzero trace}, we see that if $\omega,\omega'\in\Hom^\sigma (T ,\BG_m)$ are dominant enough then
\begin{equation}   \label{e:non-exceptional}
\beta_{\omega+\omega'-\alpha_O}=\beta_{\omega}\cdot\beta_{\omega'} \quad \mbox{ if } O \mbox{ is non-exceptional,}
\end{equation}
\begin{equation}   \label{e:exceptional}
\beta_{\omega+\omega'-2\alpha_O}=\beta_{\omega}\cdot\beta_{\omega'} \quad \mbox{ if } O \mbox{ is exceptional.}
\end{equation}

By assumption, the restriction of $f$ to $[G]_\sigma$ is not equal to the identity. So there is an irreducible character $\chi$ of $G$ such that $\chi\circ f$ and $\chi$ have different restrictions to $[G]_\sigma\,$. If $\chi$ has nonzero restriction to $[G]_\sigma$ then $\chi=\chi_\omega\cdot\mu$ for some $\mu\in\Hom (\Gamma,\BG_m)$ and some dominant $\omega\in\Hom^\sigma (T , ,\BG_m)$. 
So $\beta_\omega \ne 1$ for some $\omega$. By  \eqref{e:beta is multiplicative}, \eqref{e:non-exceptional}, \eqref{e:exceptional}, this can happen only if there exists an exceptional orbit $O_{exc}\subset\Gamma$ (it is obviously unique); moreover, in this case for any 
dominant $\omega\in\Hom^\sigma (T , ,\BG_m)$ one has
\begin{equation}   \label{e:the only beta}
\beta_\omega (\sigma)= (-1)^{(\omega,\check\omega_i)},
\end{equation}
where $\check\omega_i$ is the fundamental coweight corresponding to any $i\in O_{exc}\,$.

It remains to prove formula~\eqref{e:f([sigma])}. To do this, it suffices to check that
\begin{equation}   \label{e:chif([sigma])}
(\chi\circ f)([\sigma])=\chi ([\sigma\cdot \varepsilon_i]),\quad \varepsilon_i:=\check\omega_i (-1).
\end{equation}
It is enough to do this if $\chi=\chi_\omega\,$. Then \eqref{e:chif([sigma])} follows from \eqref{e:the only beta}.
\end{proof}

\begin{rem}
The above proof of Proposition~\ref{p:adjoint case} is straightforward but strange. The theorem of Jantzen discussed in \S\ref{ss:Satz 9} of Appendix~\ref{s:twisted conjugacy} sheds some light on what is going on. But this theorem itself is mysterious.
\end{rem}

\subsection{The general case}   \label{ss:gen case}
\subsubsection{What is to be proved}   \label{sss:What is to be proved}
Let $G_1,G_2\in\PROSS_\Gamma\,$. We have to prove that any isomorphism
\begin{equation} \label{e:2given_iso}
\BK^+_{G_1}\iso \BK^+_{G_2}
\end{equation} 
comes from a unique $\Pross_\Gamma$-isomorphism $G_1\iso G_2\,$. By Proposition~\ref{p:easy reduction}, we can assume that $\Gamma$ is finite and $G_1^\circ\,$,$G_2^\circ$ have finite type.

We think of \eqref{e:2given_iso} in terms of Corollary~\ref{c:in terms of [G]}, i.e., as 
as a $(\Gamma\times\BN)$-equivariant isomorphism $$f:[G_1]\iso [G_2]$$ of schemes over $\Gamma$ with the following property: for any subgroup $U\subset\Gamma$ a function on 
$[G_1]\times_\Gamma U$ is an irreducible character of $G_1\times_\Gamma U$ if and only if the corresponding function on 
$[G_2]\times_\Gamma U$ is an irreducible character of $G_2\times_\Gamma U$.

Note that just as in \S\ref{sss:reduction to}, the isomorphism \eqref{e:2given_iso} induces a
$\Gamma$-equivariant $\Pross$-isomorphism 
\begin{equation}   \label{e:G_i^circ is same}
G_1^\circ \iso G_2^\circ\, .
\end{equation}

\subsubsection{$G_i$ in terms of finite group extensions}

We will use the equivalence $\Pross\iso\Pross^{\prime\prime}$ form \S\ref{ss:Pross and variants}.
Recall that by \S\ref{sss:Pross''}, an object of $\Pross^{\prime\prime}$ is a collection of data a)-d).
Because of the isomorphism~\eqref{e:G_i^circ is same}, the data a)-c) corresponding to $G_1$ and $G_2$ are the same; these data are a finite Dynkin diagram $\Delta$ equipped with $\Gamma$-action and a $\Gamma$-equivariant quotient $Z$ of the group 
$Z_\Delta\,$. We also have data d), i.e., extensions
\begin{equation}   \label{e:the two extensions}
0\to Z\to\tilde\Gamma_i\to\Gamma\to 0, \quad\quad i\in\{1,2\}.
\end{equation}

By formula~\eqref{e:inverse construction}, the groups $G_i$ are expressed in terms of these data as follows:
\begin{equation}  \label{e:2inverse construction}
G_i:=(\tilde\Gamma_i\ltimes G^\circ)/\Ker (Z\times Z\overset{m}\longrightarrow Z), \quad\quad i\in\{1,2\},
\end{equation}
where $G^\circ:=G_\Delta/\Ker (Z_\Delta\epi Z)\,$.

\subsubsection{The sets $[\tilde\Gamma_i]$}   \label{sss:[tildeGamma]}
Let $[\tilde\Gamma_i]$ denote the quotient of $\tilde\Gamma_i$ by the conjugation action of $Z$. This is a set equipped with the following pieces of structure:

(a) the map $[\tilde\Gamma_i]\to\Gamma$;

(b) the action of $\Gamma\times\BN$ on $[\tilde\Gamma_i]$ (namely, $\Gamma$ acts by conjugation and $n\in\BN$ acts by raising to the $n$-th power);

(c) the action of $Z$ on $[\tilde\Gamma_i]$ by left (or equivalently, right) multiplication;

(d) for every subgroup $\Gamma'\subset\Gamma$, the group structure on $(\Gamma'\underset{\Gamma}\times [\tilde\Gamma_i])/A$, where 
$A:=\Ker (Z\epi Z_{\Gamma'})$; the group structure comes from the identification of $(\Gamma'\underset{\Gamma}\times [\tilde\Gamma_i])/A$ with the group $(\Gamma'\underset{\Gamma}\times \tilde\Gamma_i)/A\,$.

For $\gamma\in\Gamma$ let $[\tilde\Gamma_i]_\gamma$ denote the fiber of $[\tilde\Gamma_i]$ over $\gamma$; this is a torsor over 
$Z_\gamma\, $, where $Z_\gamma$ is the quotient of $Z$ by the subgroup of elements of the form $\gamma (z)/z$, $z\in Z$. Recall that $[G_i]_\gamma$ denotes the fiber of $[G_i]$ over $\gamma$; this is a scheme on which $Z_\gamma$ acts by left (or equivalently, right translations).

\begin{lem}   \label{l:[G_i]_gamma}
Let $\gamma\in\Gamma$. Then $[G_i]_\gamma$ canonically identifies with the $[\tilde\Gamma_i]_\gamma$-twist of the 
$Z_\gamma$-space $[G_0]_\gamma\,$, where $G_0:=\Gamma\ltimes G^\circ$.
\end{lem}

\begin{proof}
Use formula \eqref{e:2inverse construction} and note that 
$\tilde\Gamma_i\ltimes G^\circ=\tilde\Gamma_i\times_\Gamma G_0\,$.
\end{proof}

\subsubsection{The bijection $[\tilde\Gamma_1]\iso [\tilde\Gamma_2\,]$}
Set $G_{ad}:=G_1/Z=G_2/Z=G_0/Z$.
The isomorphism \eqref{e:2given_iso} induces an isomorphism $\BK^+_{G_1/Z}\iso \BK^+_{G_2/Z}$ or equivalently, an automorphism of $\BK^+_{G_{ad}}\,$. By Proposition~\ref{p:adjoint case}, the latter is equal to the identity. So our isomorphism $f:[G_1]\iso [G_2]$ is an isomorphism over
$[G_{ad}]$. 

For each $\gamma\in\Gamma$, Lemma~\ref{l:[G_i]_gamma} describes  $[G_i]_\gamma$ in terms of the $Z_\gamma$-torsor $[\tilde\Gamma_i]_\gamma\,$.

\begin{lem}  \label{l:phi_gamma}
For each $\gamma\in\Gamma$, the isomorphism $f_\gamma :[G_1]_\gamma\iso[G_2]_\gamma$ induced by $f$ comes from an isomorphism of $Z_\gamma$-torsors $\varphi_\gamma :[\tilde\Gamma_1]_\gamma\iso [\tilde\Gamma_2]_\gamma\,$. 
\end{lem}

\begin{proof}
$f_\gamma$ is an ismorphism over $[G_{ad}]_\gamma$. The GIT quotient of $[G_i]_\gamma$ by 
$Z_\gamma$ equals $[G_{ad}]_\gamma$. Moreover, by Lemma~\ref{l:no kernel}, the action of 
$Z_\gamma$ on $[G_i]_\gamma$ has trivial kernel, so the morphism 
$[G_i]_\gamma\to [G_{ad}]_\gamma$ is a $Z_\gamma$-torsor over a dense open subset of 
$[G_{ad}]_\gamma\,$. The lemma follows.
\end{proof}

The isomorphisms $\varphi_\gamma\,$ from Lemma~\ref{l:phi_gamma} define a bijection 
$\varphi :[\tilde\Gamma_1]\iso [\tilde\Gamma_2\,]$.

\begin{prop}   \label{p:compatibility}
$\varphi:[\tilde\Gamma_1]\iso [\tilde\Gamma_2\,]$ is compatible with the structures (a)-(d) from 
\S\ref{sss:[tildeGamma]}.
\end{prop}

\begin{proof}
Compatibility with (a) and (c) is clear. Compatibility with (b) follows from $(\Gamma\times\BN )$-equivariance of the isomorphism $f:[G_1]\iso [G_2]$.

Let us prove compatibility with (d). It is here that we use the cohomological condition from the definition of
$\PROSS_\Gamma\,$.

Without loss of generality, we can assume that the subgroup 
$\Gamma'$ from (d) equals $\Gamma$. We can also assume that the action of $\Gamma$ on $Z$ is trivial (otherwise replace $G_i$ by $G_i/A$, where $A:=\Ker (Z\epi Z_\Gamma$). Under these assumptions, 
$\varphi$ is a bijection $\tilde\Gamma_1\iso\tilde\Gamma_2$, and  we have to prove that this bijection is a group isomorphism.

$\varphi$ is compatible with (c), so for every $\gamma,\gamma'\in\Gamma_1$ there exists $z(\gamma,\gamma')\in Z$ such that
\[
\varphi (\tilde\gamma\cdot \tilde\gamma')=z(\gamma,\gamma')\cdot\varphi (\tilde\gamma)\varphi (\tilde\gamma') 
\]
for all $\tilde\gamma,\tilde\gamma'\in\tilde\Gamma_1$ such that $\tilde\gamma\mapsto\gamma$ and $\tilde\gamma'\mapsto\gamma'$. The goal is to show that $z(\gamma ,\gamma')=1$. 

We will prove that
\begin{equation}   \label{e:chars to chars}
\mbox{if } \chi\in\Hom (\tilde\Gamma_2\, ,\BG_m)\mbox{ then } \chi\circ\varphi\in\Hom (\tilde\Gamma_1\, ,\BG_m).
\end{equation}
 This will imply that $\chi (z(\gamma,\gamma'))=1$ for any 
$\chi\in\Hom (\tilde\Gamma_2\, ,\BG_m)$. By the cohomological condition from the definition of $\PROSS_\Gamma$ (see \S\ref{sss:form_thm}), any homomorphism $Z\to\BG_m$ can be extended to a homomorphism  $\tilde\Gamma_2\to\BG_m\,$. So $z(\gamma ,\gamma')$ is killed by all characters of $Z$, which means that $z(\gamma ,\gamma')=1$. 

Thus it remains to prove \eqref{e:chars to chars}. Let $\chi\in\Hom (\tilde\Gamma_2\, ,\BG_m)$. 
By Lemma~\ref{l:minuscule}, the restriction of $\chi$ to $Z$ can be extended to a $\Gamma$-invariant dominant weight $\omega$ of the maximal torus $T\subset G^\circ$. 
Let $V_\omega$ denote the irreducible $G^\circ$-module with highest weight $\omega$ and 
$l_\omega\subset V_\omega\,$ the highest line. There is a unique way to extend the $G^\circ$-action on $V_\omega$ to an action of $G_2$ in the same vector space so that the action of the subgroup 
$\tilde\Gamma_2\subset G_2$ on $l_\omega$ is given by $\chi\in\Hom (\tilde\Gamma_2\, ,\BG_m)$. This representation of $G_2$ will be denoted by $\rho_{\chi ,\omega}\,$.

Recall that $f:[G_1]\iso [G_2]$ transforms irreducible characters of $G_2$ to those of $G_1\,$. In particular,  $f$ transforms the character of $\rho_{\chi ,\omega}$ to the character of some irreducible representation 
$\rho$ of $G_1\,$. The restriction of $\rho$ to $G_1^\circ$ is isomorphic to $V_\omega\,$, so $\rho$ has to be isomorphic to $\rho_{\chi' ,\omega}$ for some $\chi'\in\Hom (\tilde\Gamma_1\, ,\BG_m)$. It is easy to see that $\chi\circ\varphi=\chi'$. So $\chi\circ\varphi\in\Hom (\tilde\Gamma_1\, ,\BG_m)$, which proves \eqref{e:chars to chars}.
\end{proof}

Proposition~\ref{p:compatibility} reduces the proof of Theorem~\ref{t:variant of KLV} to the following

\begin{prop}  \label{p:coarse extensions}
Any bijection $\varphi:[\tilde\Gamma_1]\iso [\tilde\Gamma_2\,]$ compatible with the structures (a)-(d) from \S\ref{sss:[tildeGamma]} comes from an isomorphism between the extensions \eqref{e:the two extensions}. The latter is unique up to composing with conjugations by elements of $Z$.
\end{prop}

\begin{proof}
By the definition of $\PROSS_\Gamma$ (see \S\ref{sss:form_thm}), the group 
$G^\circ=G_1^\circ=G_2^\circ$ is a product of almost-simple groups. So by Shapiro's lemma, we are reduced to the case that $G^\circ$ is almost-simple, which means that $\Delta$ is connected.

Now the uniqueness statement of the proposition follows from  Lemma~\ref{l:cyclic_Sylow} applied to the action of $\Gamma$ on $Z$.

Let us prove existence. Subtracting the two extensions of $\Gamma$ by $Z$, we can assume that the second extension is trivial and trivialized. Then the bijection 
$\varphi^{-1}:[\tilde\Gamma_2]\iso [\tilde\Gamma_1]$ yields a section
$s:\Gamma\to [\tilde\Gamma_1]$ with the following properties:

(i) $s$ is $(\Gamma\times\BN )$-equivariant;

(ii) for any subgroup $\Gamma'\subset\Gamma$, the map $\Gamma'\to (\tilde\Gamma_1\times_\Gamma\Gamma')/\Ker (Z\epi Z_{\Gamma'})$ corresponding to $s$ is a group homomorphism.

The problem is to lift $s$ to a group homomorphism ${\bf s}:\Gamma\to\tilde\Gamma_1\,$.

Let us now reduce the problem to the case where the homomorphism $\Gamma\to\Aut\Delta$ is injective.
Set $\Gamma_0:=\Ker (\Gamma\to\Aut\Delta)$. Restricting $s$ to $\Gamma_0$, we get a map $s_0:\Gamma_0\to\tilde\Gamma_1\,$. Applying property (ii) for $\Gamma'=\Gamma_0\,$, we see that $s_0$ is a homomorphism. By property (i), $s_0$ is $\Gamma$-equivariant, so the subgroup $s_0(\Gamma_0)\subset\Gamma$ is normal. Moreover,
\begin{equation}   \label{e:Gamma_0-equivariance}
s(\gamma_0\gamma )=s_0(\gamma_0)\cdot s(\gamma) \quad\mbox{ for }\gamma_0\in\Gamma_0\, ;
\end{equation}
to see this, apply property (ii) for $\Gamma'$ being the subgroup generated by $\gamma$ and 
$\Gamma_0\,$.

The extension $0\to Z\to\tilde\Gamma_1\to\Gamma\to 0$ is induced from the extension
$$0\to Z\to\tilde\Gamma_1/s(\Gamma_0)\to\Gamma/\Gamma_0\to 0.$$
Let $[\tilde\Gamma_1/s(\Gamma_0)]]$ denote the quotient set of $\tilde\Gamma_1/s(\Gamma_0)$ by the conjugation action of $Z$. By \eqref{e:Gamma_0-equivariance}, the composition 
$\Gamma\overset{s}\longrightarrow[\tilde\Gamma_1]\to [\tilde\Gamma_1/s(\Gamma_0)]$ factors as
$\Gamma\epi\Gamma/\Gamma_0\overset{s'}\longrightarrow[\tilde\Gamma_1/s(\Gamma_0)]$. The map $s'$ has properties similar to the above-properties (i)-(ii), and lifting $s$ to
${\bf s}:\Gamma\to\tilde\Gamma_1$ reduces to a similar problem for $s'$.

Thus we can assume that the homomorphism $\Gamma\to\Aut\Delta$ is injective. Let us consider two cases.

a) Assume that $\Gamma$ is cyclic. Choose a generator $\sigma\in\Gamma$ and a lift of $s(\sigma )\in [\tilde\Gamma_1]$ to an element $u\in\tilde\Gamma_1$. Since $s$ is $\BN$-equivariant, $s(\sigma^n)$ is equal to the image of $u^n$ in $\tilde\Gamma_1\,$. Applying property (ii) for 
$\Gamma'=\{ 1\}$, we see that $s(1)=1$. So if $\sigma$ has order $m$ then $u^m=s(\sigma^m )=s(1)=1$.
Now it is clear that we can define ${\bf s}$ by ${\bf s}(\sigma^n):=u^n$.

b) Assume that $\Gamma$ is not cyclic. Since $\Gamma\subset\Aut\Delta$ and $\Delta$ is connected, we see that $\Delta$ has type $D_4$ and $\Gamma=\Aut\Delta=S_3\,$. Moreover, the $S_3$-module $Z$ is either zero or isomorphic to the $S_3$-module $(\BF_2)^3/\BF_2\,$. Then $H^2(\Gamma ,Z)=H^1(\Gamma ,Z)=0$, so our extension 
$$0\to Z\to\tilde\Gamma_1\to\Gamma\to 0$$
has a splitting ${\bf s}:\Gamma\to\tilde\Gamma_1\,$, which is unique up to $Z$-conjugation. 
Let $s'$ denote the composition 
$$\Gamma\overset{\bf s}\longrightarrow\tilde\Gamma_1\to [\tilde\Gamma_1].$$
We have to show that $s=s'$. It suffices to check that for any cyclic subgroup $C\subset\Gamma$ one has $s|_C=s'|_C\,$. Both $s$ and $s'$ come from splittings $C\to\tilde\Gamma$: for $s'$ this is clear and for $s$ this follows from case a) considered above, Finally, a splitting $C\to\tilde\Gamma$ is unique up to $Z$-conjugation because $H^1(\Gamma ,Z)=0$.
\end{proof}

Let us note that Proposition~\ref{p:coarse extensions} immediately follows from the more general and natural Proposition~\ref{p:2coarse extensions} formulated in Appendix~\ref{s:Zassenhaus}.

\section{Proof of Theorem~\ref{t:main}}   \label{s:main}
Just as in Theorem~\ref{t:main}, let $X$ denote an irreducible smooth variety over $\BF_p \,$. We will use the symbols $\lambda,\lambda'$ to denote non-Archimedean places of $\BQbar$ not dividing $p$ (no relation to the word ``lambda-semiring"!).

In \S\ref{ss:LaDr} we recall the $\lambda$-independence results from \cite{La,De2,Dr}.
In \S\ref{ss:main in terms of K} we reformulate Theorem~\ref{t:main} using the language of \S\ref{s:KLV}.
In \S\ref{ss:curves}-\ref{ss:arbitrary curves} we prove Theorem~\ref{t:main} for curves by combining \S\ref{ss:LaDr} with the main result of \S\ref{s:KLV} and a well known cohomological property of smooth affine curves over $\BF_p$\,.
In \S\ref{ss:general} we prove Theorem~\ref{t:main} in general; the general case is reduced to that of curves using Hilbert irreducibility.

\subsection{Recollections on existence of companions}  \label{ss:LaDr}
In \S\ref{sss:rep hat-Pi-lambda} we defined the Tannakian category $\cT_\lambda (X)$ for each non-Archimedean place $\lambda$ of $\BQbar$ coprime to $p$. Recall that $\cT_\lambda (X)$ is the tensor category of semisimple $\lambda$-adic representations of $\Pi$ with the following property: the determinant of each irreducible component has finite order. We will use the following facts from \cite{La,De2,Dr}.

\begin{theorem}  \label{t:Langl}
Let $\lambda$ be a non-Archimedean place of $\BQbar$ coprime to $p$, and let $\E\in\cT_\lambda (X)$ . Then

(i) for each $x\in |X|$, the polynomial $\det (1-t\cdot F_x\, ,\E )$ belongs to $\BQbar [t]$;

(ii) for every non-Archimedean place $\lambda'$ of $\BQbar$  coprime to $p$, there exists $\E'\in\cT_{\lambda'} (X)$ such that
\[
\det (1-t\cdot F_x\, ,\E )=\det (1-t\cdot F_x\, ,\E' ) \quad\mbox{ for all } x\in |X|.
\]
\end{theorem}

Statement (i) is \cite[Thm.~VII.7(i)]{La}. Statement (ii) is a combination of \cite[Thm.~VII.7(ii)]{La}, \cite[Thm.~3.1]{De2} ,  and  \cite[Thm.~1.1]{Dr}. These results were deduced in \cite{La,De2,Dr} from the main theorem of~\cite{La} (namely, the Langlands conjecture for $GL(n)$ over function fields). The latter now has a shorter proof, see~\cite{VLa2}.

\begin{rem}
In the proof of \cite[Thm.~1.1]{Dr} it is essential that $X$ is assumed smooth. According to \cite[Conjecture~1.2.10(v)]{De}, the result should remain valid if $X$ is only assumed normal.
\end{rem}

In the situation of Theorem~\ref{t:Langl}(ii) we say that $\E$ and $\E'$ are \emph{compatible} or that  $\E'$ is a \emph{$\lambda'$-adic companion\,} of $\E$. By \v {C}ebotarev density, $\E'$ is unique up to isomorphism. Note that $\E'$ is irreducible if and only if $\E$ is:
the ``only if" statement follows from Theorem~\ref{t:Langl}(ii), and the ``if" statement follows by symmetry.

By \S\ref{sss:rep hat-Pi-lambda}, the lambda-semiring $K^+(\hat\Pi_{(\lambda )} )=K^+(\hat\Pi_\lambda )$ identifies with the Grothendieck semiring of $\cT_\lambda (X)$. Associating to a $\lambda$-adic representation of $\Pi$ the restriction of its character to $\Pi_{\Fr}\subset\Pi$, we realize $K^+(\hat\Pi_{(\lambda )} )$ as a lambda-subsemiring of the algebra of functions $\Pi_{\Fr}\to \BQbar_\lambda$ invariant under $\Pi$-conjugation.\footnote{Functions $\Pi_{\Fr}\to \BQbar_\lambda$ form a lambda-ring. The corresponding Adams operation $\psi^n$ is defined by $(\psi^n f)(g):=f(g^n)$.}

\begin{cor}     \label{c:Langl}
i) $K^+(\hat\Pi_{(\lambda )} )$ is a lambda-subsemiring of the algebra of functions $\Pi_{\Fr}\to \BQbar$. 

(ii) This subsemiring does not depend on $\lambda$.
\end{cor}

\begin{proof}
Follows from Theorem~\ref{t:Langl}.
\end{proof}

\begin{rem}  \label{r:from Pi to U}
Let $U\subset\Pi$ be an open subgroup. Then $U_{\Fr}=\Pi_{\Fr}\cap U$. Applying Corollary~\ref{c:Langl} to $\tilde X/U$ instead of $X$, one identifies the lambda-ring
\begin{equation}  \label{e:from Pi to U}
K^+(\hat\Pi_{(\lambda )}\times_\Pi U )
\end{equation}
with a subsemiring of the algebra of functions $\Pi_{\Fr}\cap U\to \BQbar$ invariant under $U$-conjugation; moreover, this subsemiring does not depend on $\lambda$. 
\end{rem}

\subsection{Reformulating Theorem~\ref{t:main} in terms of \S\ref{s:KLV}}   \label{ss:main in terms of K}
To any pro-finite group $\Gamma$ we associated in \S\ref{sss:ProssGamma} a groupoid $\Pross_\Gamma\,$; it depends on the choice of an algebraically closed ground field $E$ of charactersitic 0.
Now take $\Gamma =\Pi$ and $E=\BQbar$. Then $\hat\Pi_{(\lambda )}$ is an object of $\Pross_\Pi\,$.

In \S\ref{sss:BK} we defined a functor
\begin{equation}  \label{e:3BK}
\Pross_\Pi\to \Lambda^+_\Pi \, , \quad G\mapsto \BK^+_G\, .
\end{equation}
In particular, for $G=\hat\Pi_{(\lambda )}$ we get $\BK^+_{\hat\Pi_{(\lambda )}}\,$.
According to the definition from \S\ref{sss:BK}, $\BK^+_{\hat\Pi_{(\lambda )}}$ is the following collection of data: the lambda-rings \eqref{e:from Pi to U} for all open subgroups $U\subset\Pi$, the ``natural" homomorphisms between them, and the homomorphisms $ K^+(U)\to K^+(\hat\Pi_{(\lambda )}\times_\Pi U )$. By Remark~\ref{r:from Pi to U}, these data do not depend on $\lambda$. So given $\lambda$ and $\lambda'$ we have a canonical isomorphism
\begin{equation}    \label{e:iso between BKs}
\BK^+_{\hat\Pi_{(\lambda )}}\iso\BK^+_{\hat\Pi_{(\lambda' )}}
\end{equation}
in the category $\Lambda^+_\Pi$ defined in \S\ref{sss:BK}.

Now our main Theorem~\ref{t:main} can be reformulated as follows.

\begin{theorem}  \label{t:2main}
There exists a unique isomorphism 
\begin{equation}   \label{e:iso to construct}
\hat\Pi_{(\lambda )}\iso \hat\Pi_{(\lambda' )}
\end{equation}
in the category $\Pross_\Pi$ inducing the isomorphism \eqref{e:iso between BKs}.
 \end{theorem}
 
 Recall that by Proposition~\ref{p:2simply connected}, $\hat\Pi_{(\lambda )}^\circ$ is simply connected, so it is a product of almost-simple groups. In other words, $\hat\Pi_{(\lambda )}$ belongs to the full subcategory
 $\Pross_\Pi^{prod}\subset\Pross_\Pi\,$. Similarly, $\hat\Pi_{(\lambda' )}\in\Pross_\Pi^{prod}$. So uniqueness in Theorem~\ref{t:2main} follows from 
 Theorem~\ref{t:variant of KLV}(i), which says that the restriction of the functor~\eqref{e:3BK} to  $\Pross_\Pi^{prod}$ is faithful.

 Since $\hat\Pi_{(\lambda )}$ and $\hat\Pi_{(\lambda' )}$ are in $\Pross_\Pi^{prod}\,$,  Theorem~\ref{t:variant of KLV}(ii) would immediately imply existence of an isomorphism~\eqref{e:iso to construct} \emph{if} one has the equality
\begin{equation}   \label{e:dream}
\PROSS_\Pi =\Pross_\Pi^{prod} \, .
\end{equation}
In the next subsection we will show that \eqref{e:dream} does hold if $X$ is an affine curve.

\subsection{Proof of Theorem~\ref{t:2main} for affine curves}   \label{ss:curves}
It suffices to show that if $X$ is an affine curve then \eqref{e:dream} holds. According to the definition of 
$\PROSS_\Pi$ (see \S\ref{sss:form_thm}), it is enough to check that for any open subgroup $U\subset\Pi$ one has
$H^2 (U,\BQbar^\times )=0$. This is equivalent to proving that 
\begin{equation}   \label{e:cohom-vanishing}
H^2 (U,\BQ /\BZ )=0
\end{equation}
or equivalently, that $H^3 (U,\BZ )=0$. The latter fact is well known (e.g., see Theorem~8.3.17 of \cite{NSW}, which says that the fundamental group of a connected smooth affine curve over $\BF_p\,$ has strict cohomological dimension 2).

For completeness, let us give a proof of \eqref{e:cohom-vanishing}.
Without loss of generality, we can assume that $U=\Pi$ (otherwise replace $X$ by a finite etale covering).
Since $H^2 (\Pi,\BQ /\BZ )\subset H^2_{\et} (X, \BQ/\BZ )$, it suffices to prove the following

\begin{prop}   \label{p:cohom-vanishing}
Let $X$ be a smooth affine curve over $\BF_p\,$. Then 

(i) the etale cohomological dimension of $X$ is $\le 2$;

(ii) $H^2_{\et} (X, \BQ/\BZ )=0$. 
\end{prop}

\begin{proof}
The etale cohomological dimension of $X\otimes\bar\BF_p$ is $\le 1$. This implies (i).

By Artin-Schreier, $H^2_{\et} (X, \BF_p)=0$. So $H^2_{\et} (X, \BQ_p/\BZ_p )=0$.

 It remains to show that for any prime $\ell \ne p$ one has $H^2_{\et} (X, \BQ_\ell /\BZ_\ell )=0$. 
 By (i), $H^2_{\et} (X, \BQ_\ell /\BZ_\ell )$ is a quotient of $H^2_{\et} (X, \BQ_\ell )$. Finally,
$H^2_{\et} (X, \BQ_\ell )=H^1 (\Gal (\bar\BF_p/\BF_p) ,H^1_{\et} (X\otimes\bar\BF_p\, ,\BQ_\ell ))=0$
 because the weights of the geometric Frobenius on $H^1_{\et} (X\otimes\bar\BF_p\, ,\BQ_\ell )$ equal $2$ or $1$.
\end{proof}

\begin{rem}
One can ask whether $H^2_{\et} (X, \BQ_p/\BZ_p )=0$ if $X$ is a regular integral scheme quasi-finite over $\Spec\BZ [p^{-1}]$. According to \cite[Thm. 10.3.6]{NSW}, the positive answer to this question is equivalent to  Leopoldt's conjecture for the field of rational functions on $X$ and the prime $p$. More general conjectures of this type are formulated and discussed in \cite{Sch, J}.
\end{rem}

\subsection{Proof of Theorem~\ref{t:2main} for curves}   \label{ss:arbitrary curves}
We already proved the uniqueness part of Theorem~\ref{t:2main}. In the case of affine curves we also proved the existence part of Theorem~\ref{t:2main}. So Theorem~\ref{t:2main} for arbitrary curves follows from the next statement.

\begin{prop}
Let $X^1$ be an irreducible smooth variety over $\BF_p$ equipped with an open embedding $X\mono X^1$. Suppose that the existence part of Theorem~\ref{t:2main} holds for $X$. Then it holds for $X^1$.
\end{prop}

\begin{proof}
Let $\Pi^1$ denote the fundamental group of $X^1$; more precisely, $\Pi^1$ is the quotient of $\Pi$ classifying finite etale coverings of $X^1$.
Let $\hat\Pi_{(\lambda )}^1$ and $\hat\Pi_{(\lambda' )}^1$ denote the corresponding quotients of $\hat\Pi_{(\lambda )}$ and 
$\hat\Pi_{(\lambda' )}$. Fix an isomorphism $\hat\Pi_{(\lambda )}\iso\hat\Pi_{(\lambda' )}$ over $\Pi$ inducing the isomorphism  
\eqref{e:iso between BKs} (it exists by assumption). By Corollary~\ref{c:Langl}, $K^+(\hat\Pi_{(\lambda )}^1 )$ and $K^+(\hat\Pi_{(\lambda' )}^1 )$ are equal to each other as subsemirings of $K^+(\hat\Pi_{(\lambda )} )=K^+(\hat\Pi_{(\lambda' )} )$. So the isomorphism 
$\hat\Pi_{(\lambda )}\iso\hat\Pi_{(\lambda' )}$ induces an isomorphism $\hat\Pi_{(\lambda )}^1\iso\hat\Pi_{(\lambda' )}^1$. The latter has the required property. 
\end{proof}

\subsection{Proof of Theorem~\ref{t:main} in general} \label{ss:general} 
\subsubsection{Normal subgroups}   \label{sss:normals}

For any pro-semisimple group $G$, let $\Norm (G)$ denote the set of normal subgroups $H\subset G$ such that $G/H$ has finite type. Internal automorphisms of $G$ act trivially on $\Norm (G)$, so $\Norm (G)$ makes sense even if $G$ is an object of the ``coarse" groupoid $\Pross$.

Thus one has $\Norm (\hat\Pi_{(\lambda)})$ and $\Norm (\hat\Pi_{\lambda})$.
By the definition of $\hat\Pi_{(\lambda)}\,$, one has a canonical bijection
$\Norm (\hat\Pi_{(\lambda)})\iso\Norm (\hat\Pi_{\lambda}).$

On the other hand, Lemma~\ref{l:normal subgr} provides a description of $\Norm (\hat\Pi_{(\lambda)})$ in terms of $K^+(\hat\Pi_{(\lambda )} )$.
So by Corollary~\ref{c:Langl}, the set $\Norm (\hat\Pi_{(\lambda)})=\Norm (\hat\Pi_{\lambda})$ does not depend on $\lambda$.

\subsubsection{A finite-type variant of Theorem~\ref{t:main}}   \label{sss:main-finite}
Fix $\lambda$ and $\lambda'$. Let $H\in \Norm (\hat\Pi_{(\lambda)})$ and let $H'\in \Norm (\hat\Pi_{(\lambda')})$ be the image of $H$ under the canonical bijection 
\begin{equation}   \label{e:NormNorm}
\Norm (\hat\Pi_{(\lambda)})\iso\Norm (\hat\Pi_{(\lambda')})
\end{equation}
defined in \S\ref{sss:normals} above. In other words, one has the equality
\begin{equation}   \label{e:HH'}
K^+(\hat\Pi_{(\lambda)}/H)=K^+(\hat\Pi_{(\lambda')}/H'),
\end{equation}
in which both sides are considered as subsets of the set of functions $\Pi_{\Fr}\to\BQbar\,$.

Set $U:=\im (H\to\Pi )=\im (H'\to\Pi )$; this is an open normal subgroup of $\Pi$.
In \S\ref{sss:proof of p:main-finite} we will prove the following variant of Theorem~\ref{t:main}. 

\begin{prop}   \label{p:main-finite}
There exists an isomorphism $\varphi_H :\hat\Pi_{(\lambda)}/H\iso\hat\Pi_{(\lambda')}/H'$ in the category $\Pross_{\Pi/U}$ which sends the canonical map
$\Pi_{\Fr}\to [\hat\Pi_{(\lambda)}/H](\BQbar )$ to the similar map $\Pi_{\Fr}\to [\hat\Pi_{(\lambda')}/H'](\BQbar )$ . 
\end{prop}

\begin{rem}    \label{r:main-finite}
Let $\Norm^{prod} (\hat\Pi_{(\lambda)})$ denote the set of all $H\in\Norm (\hat\Pi_{(\lambda)})$ such that $(\hat\Pi_{(\lambda)}/H)^\circ$ is a product of almost-simple groups. If $H\in\Norm^{prod} (\hat\Pi_{(\lambda)})$ then the isomorphism $\varphi_H$ from Proposition~\ref{p:main-finite} is unique by Proposition~\ref{p:rigidity}. It follows that if $H,H_1\in\Norm^{prod} (\hat\Pi_{(\lambda)})$ and $H_1\subset H$ then the isomorphisms $\varphi_H$ and $\varphi_{H_1}$ are compatible. 
\end{rem}

Let us deduce Theorem~\ref{t:main} from Proposition~\ref{p:main-finite}. 
Recall that by Proposition~\ref{p:2simply connected}, $\hat\Pi_{(\lambda )}^\circ$ and $\hat\Pi_{(\lambda' )}^\circ$ are products of almost-simple groups. By Proposition~\ref{p:rigidity}, this implies uniqueness in Theorem~\ref{t:main}. This also implies that $\Norm^{prod} (\hat\Pi_{(\lambda)})$ is cofinal in $\Norm (\hat\Pi_{(\lambda)})$ and $\Norm^{prod} (\hat\Pi_{(\lambda')})$ is cofinal in $\Norm (\hat\Pi_{(\lambda')})$. By Proposition~\ref{p:main-finite} or by Theorem~\ref{t:variant of KLV}(iii), the image of $\Norm^{prod} (\hat\Pi_{(\lambda)})$ under the bijection \eqref{e:NormNorm} equals $\Norm^{prod} (\hat\Pi_{(\lambda')})$. So by Proposition~\ref{p:coarse epi}, the compatible family of isomorphisms 
$\varphi_H\,$, $H\in\Norm (\hat\Pi_{(\lambda)})$, provided by Proposition~\ref{p:main-finite} and Remark~\ref{r:main-finite} yields an isomorphism $\hat\Pi_{(\lambda)}\iso \hat\Pi_{(\lambda')}$ with the properties required in Theorem~\ref{t:main}. 

\medskip

Thus it remains to prove Proposition~\ref{p:main-finite}. We will deduce it from a similar statement for curves.

\subsubsection{Curves on $X$}
Let $C$ be a connected smooth curve over $\BF_p$ equipped with a morphism $f:C\to X$. Choose a universal covering $\tilde C\to C$ and a lift of $f:C\to X$ to a morphism $\tilde f:\tilde C\to\tilde X$. Set $\Pi^C:=\Aut (\tilde C/C)$; then there is a unique homomorphism $\Pi^C\to\Pi$ that makes the map $\tilde f:\tilde C\to\tilde X$ equivariant with respect to $\Pi^C$. It induces homomorphisms $\hat\Pi_{(\lambda)}^C\to\hat\Pi_{(\lambda)}$ and
$\hat\Pi_{(\lambda')}^C\to\hat\Pi_{(\lambda')}\,$.

\begin{lem}   \label{l:tfae}
The following properties of a pair $(C,f)$ are equivalent:

(i) the homomorphism $\hat\Pi_{(\lambda)}^C\to\hat\Pi_{(\lambda)}/H$ is surjective;

(ii) the map $K^+(\hat\Pi_{(\lambda)}/H)\to K^+(\hat\Pi_{(\lambda)}^C)$ takes irreducible elements to irreducible ones;

(i$'$) the homomorphism $\hat\Pi_{(\lambda')}^C\to\hat\Pi_{(\lambda')}/H'$ is surjective;

(ii$'$) the map $K^+(\hat\Pi_{(\lambda')}/H')\to K^+(\hat\Pi_{(\lambda')}^C)$ takes irreducible elements to irreducible ones.
\end{lem}

\begin{proof}
It is clear that (i)$\Leftrightarrow$(ii) and (i$'$)$\Leftrightarrow$(ii$'$). By Corollary~\ref{c:Langl}, 
$K^+(\hat\Pi_{(\lambda)})$,  $K^+(\hat\Pi_{(\lambda)}^C)$, and the map $K^+(\hat\Pi_{(\lambda)})\to K^+(\hat\Pi_{(\lambda)}^C)$ do not depend on $\lambda$. Combining this with \eqref{e:HH'}, we see that (ii)$\Leftrightarrow$(ii$'$).
\end{proof}

The set of all elements of $\Pi_{\Fr}$ which are $\Pi$-conjugate to elements of 
$\im (\Pi_{\Fr}^C\to \Pi_{\Fr})$ will be denoted by $\Pi_{\Fr}^{(C)}\subset\Pi_{\Fr}\,$.

\begin{lem}   \label{l:main-finite}
Suppose that the equivalent conditions of Lemma~\ref{l:tfae} hold. Then there exists an isomorphism
$\hat\Pi_{(\lambda)}/H\iso\hat\Pi_{(\lambda')}/H'$ in the category $\Pross_{\Pi/U}$ which sends the composition
$$\Pi_{\Fr}^{(C)}\mono\Pi_{\Fr}\to [\hat\Pi_{(\lambda)}/H](\BQbar )$$ 
to the similar map $\Pi_{\Fr}^{(C)}\to [\hat\Pi_{(\lambda')}/H'](\BQbar )$. 
\end{lem}

\begin{proof}
Let $H^C\subset\hat\Pi_{(\lambda)}^C$ denote the preimage of $H\subset\hat\Pi_{(\lambda)}\,$.
Let $(H')^C\subset\hat\Pi_{(\lambda')}^C$ denote the preimage of $H'\subset\hat\Pi_{(\lambda')}\,$.
By assumption, we have isomorphisms
\begin{equation}    \label{e:H and H^C}
\hat\Pi_{(\lambda)}^C/H^C\iso\hat\Pi_{(\lambda)}/H , \quad \hat\Pi_{(\lambda')}^C/(H')^C\iso\hat\Pi_{(\lambda')}/H'.
\end{equation}

We already proved Theorem~\ref{t:main} for $C$, so there exists an isomorphism $\hat\Pi_{(\lambda)}^C\iso\hat\Pi_{(\lambda')}^C$ which sends the diagram $\Pi_{\Fr}^C\to [\hat\Pi_{(\lambda)}^C] (\BQbar)\epi\Pi^C$ to the similar diagram 
$\Pi_{\Fr}^C\to [\hat\Pi_{(\lambda')}^C] (\BQbar)\epi\Pi^C$. Using \eqref{e:HH'}, we see that the above isomorphism takes $H^C$ to $(H')^C$. So by \eqref{e:H and H^C}, we get an isomorphism $\hat\Pi_{(\lambda)}/H\iso\hat\Pi_{(\lambda')}/H'$. It has the required property.
\end{proof}

\begin{lem}   \label{l:Hilbert}
For any finite $S\subset\Pi_{\Fr}$ there exists a pair $(C,f)$ such that $\Pi_{\Fr}^{(C)}\supset S$ and the equivalent conditions of Lemma~\ref{l:tfae} hold. 
\end{lem}

\begin{proof}
Let $\underline{H}\subset\hat\Pi_{\lambda}$ be the normal subgroup corresponding to $H\subset\hat\Pi_{(\lambda)}\,$. Set 
$$K:=\Ker (\Pi\to(\hat\Pi_{\lambda}/H)(\BQbar_{\lambda})).$$
Let $\ell$ be the prime such that $\lambda$ divides $\ell$, then  $\Pi/K$ has an open subgroup which is  pro-$\ell\,$. So applying a variant of Hilbert irreducibility (more precisely, \cite[Prop.~2.17]{Dr} combined with  \cite[Thm.~2.15(i)]{Dr}), we see that there exists a pair $(C,f)$ such that
$\Pi_{\Fr}^{(C)}\supset S$ and the map $\Pi^C\to\Pi/K$ is surjective. Then condition (i) of Lemma~\ref{l:tfae} clearly holds. 
\end{proof}

\subsubsection{Proof of Proposition~\ref{p:main-finite}}   \label{sss:proof of p:main-finite}
Combining Lemma~\ref{l:main-finite} and Lemma~\ref{l:Hilbert}, we see that for any finite subset $S\subset\Pi_{\Fr}$ there exists an isomorphism
$\hat\Pi_{(\lambda)}/H\iso\hat\Pi_{(\lambda')}/H'$ in the category $\Pross_{\Pi/U}$ which sends the composition
$$S\mono\Pi_{\Fr}\to [\hat\Pi_{(\lambda)}/H](\BQbar )$$ to the similar map $S\to [\hat\Pi_{(\lambda')}/H'](\BQbar )$. Since the number of isomorphisms $\hat\Pi_{(\lambda)}/H\iso\hat\Pi_{(\lambda')}/H'$ in the category $\Pross_{\Pi/U}$ is finite, we can remove the condition that $S$ is finite and then take $S=\Pi_{\Fr}\,$. \qed

\section{The group scheme $\hat\Pi^{\mot}$ and a related conjecture}  \label{s:mot}

In this section we define variants of $\hat\Pi_\lambda$ and $\hat\Pi$ denoted by $\hat\Pi_\lambda^{\mot}$ and $\hat\Pi^{\mot}$ (where ``mot" stands for ``weakly motivic"). Unlike $\hat\Pi$ itself, $\hat\Pi^{\mot}$ has a chance to depend functorially on the smooth variety $X$ (see 
\S\ref{ss:functoriality in X}). The main goal of this section is to explain why I cannot prove this (see \S\ref{ss:why diff}). Another goal is to give an unconditional definition of ``motivic Langlands parameter", see \S\ref{ss:motivic Langlands parameter}.

\subsection{Definition of $\hat\Pi_\lambda^{\mot}$ and $\hat\Pi^{\mot}$}   \label{ss:def of mot}
 Recall that $\alpha\in\BQbar$ is said to be a \emph{$p$-Weil number} if $\alpha$ is a unit outside of $p$ and there exists $n\in\BZ$ such that 
 $|\phi (\alpha)|=p^{n/2}$ for all homomorphisms $\phi :\BQbar\to\BC$. Let  $\cW_p\subset\BQbar^\times$ denote the subgroup of $p$-Weil numbers; it contains the subgroup $\mu_\infty (\BQbar)\subset\BQbar^\times$ formed by all roots of unity.
 
Define a group scheme $D$ over $\BQbar$ by $D:=\Hom (\cW_p\, ,\BG_m )$. For each non-Archimedean place $\lambda$ of $\BQbar$ not dividing $p$, set 
$D_\lambda :=D\otimes_{\BQbar}\BQbar_\lambda\,$. Note that the group $\pi_0 (D)=D/D^\circ$ identifies with the group
$\Hom (\mu_\infty (\BQbar)\, ,\BG_m)=\hat\BZ$.

On the other hand, one has a canonical homomorphism 
\begin{equation}  \label{e:2to hat Z}
\Pi\to\Gal (\bar\BF_p/\BF_p)=\hat\BZ
\end{equation}
with finite cokernel. 
Composing it with the epimorphism $\hat\Pi_\lambda\epi\Pi$, one gets a homomorphism $\hat\Pi_\lambda\to\hat\BZ$.
 Now define $\hat\Pi_\lambda^{\mot}$ to be the fiber product of  $\hat\Pi_\lambda$ and $D_\lambda$ over $\hat\BZ$.

  Similarly, define $\hat\Pi^{\mot}\in\Prored (\BQbar)$ to be the fiber product of  $\hat\Pi$ and $D$ over $\hat\BZ$. (The groupoid $\Prored (\BQbar)$ was defined in \S\ref{sss:coarsecat}.) The images of 
  $\hat\Pi^{\mot}$ and $\hat\Pi_\lambda^{\mot}$ in $\Prored (\BQbar_\lambda )$ are canonically isomorphic.
  
 One has canonical isomorphisms
\[
(\hat\Pi^{\mot})^\circ\iso \hat\Pi^\circ\times D^\circ, \quad
\pi_0 (\hat\Pi^{\mot})\iso\pi_0 (\hat\Pi)=\Pi,
\]
\[
[\hat\Pi^{\mot}]\iso [\hat\Pi ]\underset{\hat\BZ}\times D\, ,
\]
and similar isomorphisms for $\hat\Pi_\lambda^{\mot}$. The subgroup $D^\circ\subset\hat\Pi^{\mot}$ is central.

The epimorphism $D(\BQbar )\epi \hat\BZ$ has a canonical splitting over $\BZ$, namely the composition 
\begin{equation}  \label{e:split over Z}
\BZ\to\Hom (\cW_p\, ,\cW_p )\mono\Hom (\cW_p\, ,\BQbar^\times )=D(\BQbar ).
\end{equation}
Consider $\BZ$ as a group scheme\footnote{This group scheme is not affine. As a scheme, it is just a disjoint union of points labeled by integers.} over $\BQbar$; this allows us to consider the group schemes $\hat\Pi_\lambda \underset{\hat\BZ}\times\BZ$ and $\hat\Pi \underset{\hat\BZ}\times\BZ$ (which are the ``Weil versions" of $\hat\Pi_\lambda$ and $\hat\Pi$). 
The epimorphism $\hat\Pi_\lambda^{\mot}\epi\hat\Pi_\lambda$ has a canonical splitting over $\hat\Pi_\lambda \underset{\hat\BZ}\times\BZ$ 
induced by \eqref{e:split over Z}. Similarly, the epimorphism $\hat\Pi^{\mot}\epi\hat\Pi$ has a splitting 
$\hat\Pi\underset{\hat\BZ}\times\BZ\to\hat\Pi^{\mot}$, which is canonical to the extent possible (i.e., up to $\hat\Pi^{\circ}$-conjugation). 
  
\subsection{Relation between $\hat\Pi_\lambda^{\mot}$ and $\hat\Pi_\lambda^{\red}$} \label{ss:mot&red}
Recall that $\hat\Pi_\lambda^{\red}$ denotes the $\lambda$-adic pro-reductive completion of $\Pi$ (see \S\ref{ss:pro-reductive}).
Consider the composition
\[
\cW_p\mono\Hom_{\cont}(\hat\BZ ,\BQbar_\lambda^\times )\to\Hom_{\cont}(\Pi ,\BQbar_\lambda^\times )=
\Hom (\hat\Pi_\lambda^{\red},\BG_m ),
\]
in which the second arrow comes from \eqref{e:2to hat Z}. It induces a homomorphism $\hat\Pi_\lambda^{\red}\to D_\lambda\,$, where 
$D_\lambda$ was defined in \S\ref{ss:def of mot}. Combining it with the epimorphism $\hat\Pi_\lambda^{\red}\to\hat\Pi_\lambda\,$, one gets a canonical homomorphism 
\begin{equation}   \label{e:red to mot} 
\hat\Pi_\lambda^{\red}\to\hat\Pi_\lambda^{\mot}\,.
\end{equation}

\begin{prop}   \label{p:mot&red}
(i) The homomorphism \eqref{e:red to mot} is surjective.

(ii) A finite-dimensional representation of $\hat\Pi_\lambda^{\red}$ is a representation of $\hat\Pi_\lambda^{\mot}$ if and only if the corresponding $\lambda$-adic local system on $X$ is weakly motivic.
\end{prop}

According to \cite[Def.~1.8]{Dr},  ``weakly motivic" means that for every closed point $x\in X$ all eigenvalues of the geometric Frobenius of $x$ are $q_x$-Weil numbers, where $q_x$ is the order of the residue field of $x$.

\begin{proof}
Combine Proposition~\ref{p:pro-reductive} and \cite[Thm.~VII.7(i-ii)]{La}.
\end{proof}

\subsection{The maps $\Pi\to\hat\Pi_\lambda^{\mot} (\BQbar_\lambda )$ and $\Pi_{\Fr}\to [\hat\Pi^{\mot} (\BQbar )]$}
One has a canonical continuous homomorphism $$\Pi\to\hat\Pi_\lambda^{\red} (\BQbar_\lambda )$$ with Zariski-dense image. Composing it with the epimorphism \eqref{e:red to mot}, one gets a canonical continuous homomorphism 
\begin{equation}   \label{e:Pi to lambda-mot}
\Pi\to\hat\Pi_\lambda^{\mot} (\BQbar_\lambda )
\end{equation}
with Zariski-dense image.

On the other hand, the canonical maps $\Pi_{\Fr}\to [\hat\Pi (\BQbar )]$ and
\[
\Pi_{\Fr}\to\BZ\to\Hom (\cW_p\, ,\cW_p )\mono\Hom (\cW_p\, ,\BQbar^\times )=D(\BQbar )
\]
define a canonical map
\begin{equation}  \label{e:Pi-Fr to mot}
\Pi_{\Fr}\to ([\hat\Pi ]\underset{\hat\BZ}\times D)(\BQbar )=[\hat\Pi^{\mot}] (\BQbar ).
\end{equation}

The maps $\Pi_{\Fr}\to [\hat\Pi^{\mot}] (\BQbar_\lambda )$ obtained from \eqref{e:Pi to lambda-mot} and \eqref{e:Pi-Fr to mot} are the same. So the map \eqref{e:Pi-Fr to mot} has Zariski-dense image.

\begin{prop}    \label{p:2mot&red}
Let $H$ be an algebraic group  over $\BQbar_\lambda\,$. Then the map 
$$\Hom (\Pi_\lambda^{\mot},H)\to\Hom_{\cont}(\Pi , H(\BQbar_\lambda )$$ 
induced by \eqref{e:Pi to lambda-mot} is injective. A continuous homomorphism $\rho :\Pi\to H(\BQbar_\lambda )$ belongs to its image if and only if the Zariski closure of $\rho (\Pi )$ is reductive and for every $x\in |X|$ the eigenvalues of $\rho (F_x)$ in each representation of $H$ are $q_x$-Weil numbers (here $q_x$ is the order of the residue field of $x$). 
\end{prop}

\begin{proof}
Follows from Proposition~\ref{p:mot&red}.
\end{proof}

\subsection{On functoriality of $\hat\Pi_\lambda^{\mot}$ and $\hat\Pi^{\mot}$ with respect to  $X$}
\label{ss:functoriality in X}

\subsubsection{A general construction}   \label{sss:Levi-Maltsev}
Let $H_1\to H_2$ be a homomorphism of group schemes. Let $H_i^{\red}$ denote the maximal pro-reductive quotient of $H_i\,$. Note that the composition $H_1\to H_2\epi H_2^{\red}$ does not always factor through $H_1^{\red}$. However, choose a splitting $H_1^{\red}\to H_1$ (by the Levi-Maltsev theorem, it exists and is unique up to $H_1^\circ$-conjugation); composing it with the homomorphisms
$H_1\to H_2\epi H_2^{\red}$, one gets a homomorphism $H_1^{\red}\to H_2^{\red}\,$, whose $(H_2^{\red})^\circ$-conjugacy class does not depend on the choice of the splitting $H_1^{\red}\to H_1\,$.

\subsubsection{On functoriality of $\hat\Pi_\lambda^{\mot}$ with respect to  $X$}
Suppose we have pairs $(X,\tilde X)$ and $(X',\tilde X')$ as in \S\ref{sss:some} and a morphism
$(X',\tilde X')\to (X,\tilde X)$. Set $\Pi:=\Aut (\tilde X/X)$,  $\Pi':=\Aut (\tilde X'/X')$, and 
let $f:\Pi'\to\Pi$ be the unique homomorphism that makes the map $\tilde X'\to\tilde X$ equivariant with respect to $\Pi'$. For each non-Archimedean place $\lambda$ of $\BQbar$ not dividing $p$, the homomorphism $f:\Pi'\to\Pi$ induces a homomorphism between the $\lambda$-adic pro-algebraic completions of $\Pi'$ and $\Pi$. By \S\ref{sss:Levi-Maltsev}, the latter induces a $(\Pi_\lambda^{\red})^\circ$-conjugacy class of homomorphisms $\hat f_\lambda^{\red}:\widehat{\Pi'}^{\red}_\lambda\to\hat\Pi_\lambda^{\red}\,$. By Proposition~\ref{p:mot&red}(i), one has epimorphisms $\widehat{\Pi'}^{\red}_\lambda\epi\widehat{\Pi'}^{\mot}_\lambda$ and $\hat\Pi_\lambda^{\red}\epi\hat\Pi_\lambda^{\mot}$. Using
Proposition~\ref{p:mot&red}(ii) and the fact that the class of weakly motivic $\lambda$-adic local systems is obviously stable under pullbacks\footnote{Note that this statement becomes wrong if one replaces ``weakly motivic" by the property that the determinant of each irreducible component has finite order. So one does not get a homomorphism $\widehat{\Pi'}_\lambda\to\hat\Pi_\lambda\,$, in general.}, we see that $\hat f_\lambda^{\red}:\widehat{\Pi'}_\lambda^{\red}\to\hat\Pi_\lambda^{\red}$ induces a $(\Pi_\lambda^{\mot})^\circ$-conjugacy class of homomorphisms 
$\hat f_\lambda^{\mot}:\widehat{\Pi'}^{\mot}_\lambda\to\hat\Pi_\lambda^{\mot}\,$.

\subsubsection{On functoriality of  $\hat\Pi^{\mot}$ with respect to  $X$}
Now suppose that $X$ and $X'$ are smooth, then $\hat\Pi^{\mot}$ and $\widehat{\Pi'}^{\mot}$ are defined. For each $\lambda$ as above, let 
$\hat f_{(\lambda)}:\widehat{\Pi'}^{\mot}\to\hat\Pi^{\mot}$ be the homomorphism corresponding to 
$\hat f_\lambda^{\mot}:\widehat{\Pi'}^{\mot}_\lambda\to\hat\Pi_\lambda^{\mot}\,$ by Proposition~\ref{p:reductive coarse}; it is well-defined up to $(\hat\Pi^{\mot})^\circ$-conjugation.

\begin{conj}  \label{c:functoriality in X}
The $(\hat\Pi^{\mot})^\circ$-conjugacy class of $\hat f_{(\lambda )}^{\mot}$ does not depend on $\lambda$.
\end{conj}

\subsection{The difficulty}   \label{ss:why diff}
The conjecture is easy to prove if the subgroup $f (\Pi')\subset\Pi$ has finite index or if the image of the map $X'\to X$ has a single point. However, I cannot prove the conjecture in general\footnote{E.g., I cannot prove it if $X$ is a surface and $X'$ is a curve on it such that the image of $\Pi'$ in $\Pi$ does not have finite index.}: although it is easy to see that the map 
$[\hat{f}_{(\lambda )}]:[\widehat{\Pi '}^{\mathrm{mot}}]\to [\hat{\Pi }^{\mathrm{mot}}]$ does not depend on $\lambda $, this is not enough to prove the conjecture (see the examples from \cite{L1,L2,W1,W2}, which show that ``element-conjugacy'' in the sense of \cite{L1,L2} is not the same as
conjugacy).

\subsection{Unconditional definition of ``motivic Langlands parameter"} \label{ss:motivic Langlands parameter}
\subsubsection{Definition of $\Par (H,f)$}
Let $H$ be an algebraic group over $\BQbar$ equipped with a homomorphism $f: \Pi\to \pi_0(H)$. E.g., one can take $H$ to be the Langlands dual of a reductive group scheme over $X$ with connected fibers (in this case $f$ is an epimorphism).

In this situation let $\Par (H,f)$ denote the set of $H^\circ$-conjugacy classes of homomorphisms $\rho :\hat\Pi^{\mot}\to H$ such that the composition 
$\hat\Pi^{\mot}\overset{\rho}\longrightarrow H\to\pi_0(H)$ is equal to the composition $\hat\Pi^{\mot}\to\Pi\overset{f}\longrightarrow\pi_0(H)$. Elements of $\Par (H,f)$ could be called \emph{motivic Langlands parameters}.\footnote{The name is borrowed from  \cite[\S 12.2.4]{VLa2}. 
V.~Lafforgue introduces there a notion of motivic Langlands parameter using the standard conjectures, while our definition of $\Par (H,f)$ is unconditional.}

In the rest of \S\ref{ss:motivic Langlands parameter} we rephrase the definition of $\Par (H,f)$ in possibly more understandable terms.
This part of the article is not included into the journal version, so \emph{it has not been checked by the reviewers.}

\subsubsection{Definition of $\Par_{\lambda} (H,f)$}    \label{sss:Parlambda}
Let $\lambda$ be a non-Archimedean place of $\BQbar$ coprime to $p$. Then one has a canonical bijection $\Par (H,f)\iso\Par_{\lambda} (H,f)$, where $\Par_\lambda (H,f)$ is the set of $H^\circ (\BQbar_{\lambda})$-conjugacy classes of homomorphisms 
$\rho :\hat\Pi_{\lambda}^{\mot}\to H\otimes_{\BQbar}\BQbar_{\lambda}$ such that the composition 
$\hat\Pi_{\lambda}^{\mot}\overset{\rho}\longrightarrow H\otimes_{\BQbar}\BQbar_{\lambda}\to\pi_0(H)$ is equal to the composition $\hat\Pi_{\lambda}^{\mot}\to\Pi\overset{f}\longrightarrow\pi_0(H)$.  By Proposition~\ref{p:2mot&red}, such a homomorphism is the same as a continuous homomorphism $\rho :\Pi\to H (\BQbar_{\lambda})$ with the following properties:

(a) the composition 
$\Pi\overset{\rho}\longrightarrow H (\BQbar_{\lambda})\to\pi_0(H)$ is equal to $f$;

(b) the Zariski closure of $\rho (\Pi )$ is reductive;

(c) for every $x\in |X|$ the eigenvalues of $\rho (F_x)$ in each representation of $H$ are $q_x$-Weil numbers (here $q_x$ is the order of the residue field of $x$). 

\subsubsection{Strong compatibility}   \label{sss:Strong compatibility}
Let $\lambda_1$ and $\lambda_2$ be non-Archimedean places of $\BQbar$ coprime to $p$. Say that an element of $\Par_{\lambda_1}(H,f)$ and an  element of $\Par_{\lambda_2}(H,f)$ are \emph{strongly compatible} if they correspond to the same element of $\Par (H,f)$. In 
\S\ref{sss:2Par'}-\ref{sss:elem reform} we will rephrase strong compatibility in more concrete terms. 

To this end, we will define in \S\ref{sss:2Par'}-\ref{sss:2Parlambda to Par'} a huge set $\Par'(H,f)$ and an injective map 
$$\Par_\lambda (H,f)\mono \Par'(H,f).$$
By Proposition~\ref{p:Parlambda to Par'}(iii), strong compatibility is equivalent to having equal images in $ \Par'(H,f)$.
The definitions of $\Par'(H,f)$ and the map $\Par_\lambda (H,f)\mono \Par'(H,f)$ do not depend on Theorem~\ref{t:main} (in particular, they do not require smoothness of $X$). 

Theorem~\ref{t:main} is equivalent to the statement that given non-Archimedean places $\lambda_1$ and $\lambda_2\,$, every element of 
$\Par_{\lambda_1} (H,f)$ is strongly compatible with some element of $\Par_{\lambda_2} (H,f)$. Recall that the proof of Theorem~\ref{t:main} is based on the main theorem of  \cite{Dr} (which is proved using a deep theorem of L.~Lafforgue~\cite{La} and only assuming that $X$ is smooth).

\subsubsection{Definition of $\Par'(H,f)$}   \label{sss:2Par'}
Recall that $D:=\Hom (\cW_p\, ,\BG_m )$, where $\cW_p\subset\BQbar^\times$ is the subgroup of $p$-Weil numbers; the group $\pi_0 (D)=D/D^\circ$ identifies with $\Hom (\mu_\infty (\BQbar)\, ,\BG_m)=\hat\BZ$. 

Let $\Par'(H,f)$ be the set of isomorphism classes of the following data:

(i) an affine group scheme $K$ over $\BQbar$ equipped with an isomorphism $\pi_0 (K)\iso\Pi$; the group $K^\circ$ is required to be semisimple and simply connected;

(ii) an $H^\circ$-conjugacy class of homomorphisms 
\begin{equation}   \label{e:KD to H}
K\underset{\hat\BZ}\times D\to H
\end{equation}
such that the corresponding homomorphism
$\Pi\iso\pi_0 (K\underset{\hat\BZ}\times D)\to\pi_0 (H)$ equals $f$ and $\Ker (K^\circ\to H)$ is finite;

(iii) a map 
\begin{equation}   \label{e:Frob map}
\Pi_{\Fr}\to [K](\BQbar )
\end{equation}
over $\Pi$, which is equivariant with respect to $\Pi$-conjugation and has Zariski-dense image (as usual, $[K]$ denotes the GIT quotient of $K$ by the conjugation action of $K^\circ$).

\subsubsection{The map $\Par_\lambda (H,f)\to \Par'(H,f)$}   \label{sss:2Parlambda to Par'}
Let $\hat\Pi_{(\lambda )}$ be a model over $\BQbar$ for the pro-semisimple group scheme $\hat\Pi_\lambda$ (i.e., $\hat\Pi_\lambda =\hat\Pi_{(\lambda )}\otimes_{\BQbar}\BQbar_\lambda$); the choice of $\hat\Pi_{(\lambda )}$ will not matter. Set 
$\hat\Pi_{(\lambda )}^{\mot}:=\hat\Pi_{(\lambda )}\underset{\hat\BZ}\times D$.

We will think of elements of $\Par_\lambda (H,f)$ as $H^\circ (\BQbar )$-conjugacy classes of homomorphisms $\rho :\hat\Pi_{(\lambda )}^{\mot}\to H$ inducing the map $f:\Pi=\pi_0(\hat\Pi_{(\lambda )}^{\mot})\to\pi_0 (H)$. Given such $\rho$, define $K=K_\rho$ to be the quotient of 
$\hat\Pi_{(\lambda )}^{\mot}$ by the neutral connected component of $\Ker (\hat\Pi_{(\lambda )}^\circ\overset{\rho}\longrightarrow H)$. Since 
$\hat\Pi_{(\lambda )}^\circ$ is simply connected, so is $K^\circ$. Data (i)-(iii) from \S\ref{sss:2Par'} come from the isomorphism $\pi_0 (\hat\Pi_{(\lambda )}^{\mot})\iso\Pi$, the homomorphism $\rho :\hat\Pi_{(\lambda )}^{\mot}\to H$, and the canonical map  $\Pi_{\Fr}\to \hat\Pi_{(\lambda )}(\BQbar )$.

\begin{prop}   \label{p:Parlambda to Par'}
(i) The map $\Par_\lambda (H,f)\to \Par'(H,f)$ is injective.

(ii) An element of $\Par'(H,f)$ belongs to its image if and only if the corresponding map \eqref{e:Frob map} comes from a continuous homomorphism $\varphi :\Pi\to K(\BQbar_{\lambda})$ with Zariski-dense image. 
Such $\varphi$ is unique up to $K^\circ (\BQbar_\lambda )$-conjugation.

(iii) Strong compatibility in the sense of \S\ref{sss:Strong compatibility} is equivalent to having equal images in $ \Par'(H,f)$.
\end{prop}

\begin{proof}
(i) The image of $\Pi_{\Fr}$ in $[\hat\Pi_{(\lambda )}^{\mot}](\BQbar_{\lambda} )$ is Zariski-dense, so from the map \eqref{e:Frob map} one can reconstruct the morphism $[\hat\Pi_{(\lambda )}^{\mot}]\to [K]$ and therefore the subgroup $\Ker (\hat\Pi_{(\lambda )}^{\mot}\to K)$. On the other hand, by Proposition~\ref{p:rigidity} (which is applicable because $K^\circ$ is simply connected) any automorphism of $K$ inducing the identity on $[K]$ is an inner automorphism coresponding to an element of $K^\circ (\BQbar )$.

(ii) Suppose that $\varphi$ exists. Then $\varphi$ defines a homomorphism $\hat\Pi_\lambda\to K\otimes_{\BQbar}\BQbar_{\lambda}$. After conjugating $\varphi$ by an element of $K^\circ (\BQbar_\lambda )$, we get a homomorphism $\hat\Pi_{(\lambda )}\to K$ and then a homomorphism $$\hat\Pi_{(\lambda )}^{\mot}:=\hat\Pi_{(\lambda )}\underset{\hat\BZ}\times D\to K\underset{\hat\BZ}\times D.$$
Composing it with \eqref{e:KD to H}, we get the desired homomorphism $\hat\Pi_{(\lambda )}^{\mot}\to H$. Uniqueness of $\varphi$ follows from~(i).

(iii) Strong compatibility clearly implies having equal images in $ \Par'(H,f)$. The converse statement follows from (i).
\end{proof}

\subsubsection{An elementary reformulation}  \label{sss:elem reform}
Data (i)-(ii) from \S\ref{sss:2Par'} are, in fact, ``elementary". Here is a way to think of them.

Suppose we are given a triple $(K, \phi :\pi_0 (K)\iso\Pi, \rho :K\underset{\hat\BZ}\times D\to H )$ satisfying the conditions from 
\S\ref{sss:2Par'}(i-ii). Set $G:=\im (K^{\circ}\to H)$. Then $G\subset H$ is a connected semisimple subgroup, and $K^\circ$ identifies with its universal cover $\tilde G$.

One has $\im (K\underset{\hat\BZ}\times D\to H)\subset N_H(G)$, where $N_H(G)$ is the normalizer of $G$ in $H$. The homomorphism
$\rho :K\underset{\hat\BZ}\times D\to H$ induces a homomorphism\footnote{Such a homomorphism is the same as a homomorphism 
$\bar\rho :\Pi\underset{\hat\BZ}\times\BZ \to N_H(G)/G$ with open kernel such that for all (or some) $x\in X$ the element $\bar\rho (F_x)$ is semisimple and its eigenvalues in all representations of $N_H(G)/G$ are $q_x$-Weil numbers.} 
$\bar\rho :\Pi\underset{\hat\BZ}\times D\to N_H(G)/G$. Note that $\bar\rho$ has the following property: the homomorphism 
$D^\circ\to N_H(G)/G$ induced by $\bar\rho$ has a lift\footnote{Such a lift is unique. It is easy to show that it exists if the composition $\BG_m\overset{w}\longrightarrow D^\circ\to N_H(G)/G$ lifts to a homomorphism $\BG_m\to Z_H(G)$; here $w:\BG_m\to D$ corresponds to the homomorphism $\cW_p\to\BZ$ that takes a $p$-Weil number $\alpha$ to its weight, i.e., to $2\log_p |\alpha |$.} to a homomorphism $D^\circ\to Z_H(G)$, where $Z_H(G)$ is the centralizer of $G$ in $H$.

Now suppose we are given a connected semisimple subgroup $G\subset H$ and a homomorphism 
$$\bar\rho :\Pi\underset{\hat\BZ}\times\BZ \to N_H(G)/G$$ satisfying the above ``lifting property". Let us describe all triples $(K,\phi ,\rho )$ as above corresponding to $G$ and $\bar\rho$. To this end, choose a pinning $\pi$ of $G$ (in the sense of \S\ref{sss:pinnings}), and let 
$N_H(G,\pi )$ be the normalizer of $(G,\pi )$ in $H$. Let $Z(G)$ be the center of $G$. The $\bar\rho$-pullback of the exact sequence
\begin{equation} \label{e:2Par'1}
0\to Z(G)\to N_H(G,\pi)\to N_H(G)/G\to 0
\end{equation} 
is an extension of $\Pi\underset{\hat\BZ}\times D$ by $Z(G)$; because of the ``lifting property", it splits over 
$D^\circ =\Ker (\Pi\underset{\hat\BZ}\times D\to\Pi )$ and therefore descends to an extension\footnote{One can get the same extension using the homomorphism~\eqref{e:split over Z}: namely, the pullback of the exact sequence \eqref{e:2Par'1} via the composition 
$\Pi\underset{\hat\BZ}\times\BZ\to\Pi\underset{\hat\BZ}\times D\overset{\bar\rho}\longrightarrow N_H(G)/G$ 
is an extension of 
$\Pi\underset{\hat\BZ}\times\BZ$ by $Z(G)$, which is the same as an extension of $\Pi$ by $Z(G)$.}
\begin{equation}   \label{e:2Par'2}
0\to Z(G)\to \check\Pi\to \Pi\to 0.
\end{equation}
Note that by the definition of the extension~\eqref{e:2Par'2}, we have a canonical homorphism 
\begin{equation} \label{e:2Par'3}
\check\Pi\underset{\hat\BZ}\times D \to N_H(G,\pi).
\end{equation}

It is easy to check that given $G$ and $\bar\rho$ as above (in particular, $\bar\rho$ should have the ``lifting property"), 
 \emph{the groupoid of all triples $(K,\phi ,\rho )$ corresponding to $G$ and $\bar\rho$ is canonically equivalent to the groupoid of lifts of the extension \eqref{e:2Par'2} to an extension}
\begin{equation}   \label{e:2Par'4}
0\to Z(\tilde G)\to\tilde\Pi\to\Pi\to 0,
\end{equation}
where $\Pi$ acts on $Z(\tilde G)$ via the composition $\Pi\overset{\bar\rho}\longrightarrow N_H(G)/G\to\Aut Z(\tilde G)$. The triple 
$(K,\phi ,\rho )$ corresponding to such a lift is as follows: $K$ is the central term of the extension of $\Pi$ by $\tilde G$ induced by 
\eqref{e:2Par'4}, and the homomorphism $K\underset{\hat\BZ}\times D\to H$ comes from the compositions
\[
K^\circ =\tilde G\to G\mono H \quad \mbox{and} \quad \tilde\Pi\underset{\hat\BZ}\times D\to\check\Pi\underset{\hat\BZ}\times D \to N_H(G,\pi),
\]
where the last arrow is \eqref{e:2Par'3}.

\section{On $\hat\Pi_{(\lambda)}$ for $\lambda$ dividing $p$}   \label{s:crys}
Let $p$, $X$,  $\tX$, and $\Pi$ be as be as in \S\ref{sss:some}. Assume that $X$ is smooth.

\subsection{The category $\FIsocd (X)$ and the group $\pi_1^{\FIsocd} (X)$}   \label{ss:FIsocd}
\subsubsection{The Tannakian category $\FIsocd (X)$}
Let $\FIsocd (X)$ denote the category of \emph{overconvergent $F$-isocrystals} on $X$. This is a Tannakian category over $\BQ_p$ (usually a non-neutral one). One should think of $\FIsocd (X)$ as a $p$-adic analog of the category of lisse $\BQ_\ell$-sheaves on $X$.

In \cite[\S 1.3]{Ke2} and especially in \cite{Ke3} one can find a brief discussion of $\FIsocd (X)$ and the category of  \emph{convergent $F$-isocrystals}\footnote{$\FIsocd (X)$ is a full subcategory of the Tannakian category $\FIsoc (X)$ closed under tensor products and duals but usually not under subobjects. If $X$ is proper then $\FIsocd (X)=\FIsoc (X)$.}  $\FIsoc (X)\supset \FIsocd (X)$, as well as many references.

\subsubsection{The group scheme $\pi_1^{\FIsocd} (X)$}  
We will define an affine group scheme $\pi_1^{\FIsocd} (X)$, which is a crystalline analog of $\tilde\Pi_\ell\otimes_{\BQ_\ell}\BQbar_\ell$, where $\tilde\Pi_\ell$ is as in \S\ref{sss:0pro-algebraic}.

Let us recall some material from \cite[Appendix B]{DK}, which goes back to R.~Crew's article \cite{Crew-mono}. We use the notation of  \cite{DK}, which is different\footnote{In particular, the group that we denote by $\pi_1^{\FIsocd} (X)$ is \emph{different} from (but closely related to) the group denoted by $\pi_1^{\FIsocd} (X)$ in formula (2.5.4) on p.~446 of \cite{Crew-mono}.} from that of \cite{Crew-mono}.

Set 
\begin{equation}   \label{e:FIsoc on tilde X}
\FIsocd (\tX ):=\underset{U}{\underset{\longrightarrow}\lim} \FIsocd (\tX /U),
\end{equation}   
where $U$ runs through the set of open subgroups of $\Pi$. Fix an algebraic closure $\BQbar_p$ of $\BQ_p\,$. Fix a fiber functor 
\[
\tilde\xi :\FIsocd (\tX )\to\Vect_{\BQbar_p}\, , 
\]
where $\Vect_{\BQbar_p}$ is the category of finite-dimensional vector spaces over $\BQbar_p\,$. The existence of $\tilde\xi$ is guaranteed by a general theorem of Deligne~\cite{De3}; one can also construct $\tilde\xi$ by choosing a closed point $\tilde x\in\tX$ and a fiber functor $\FIsocd (\tilde x )\to\Vect_{\BQbar_p}\,$. 

Let $\xi :\FIsocd (X )\to\Vect_{\BQbar_p}$ be the composition of $\tilde\xi$ with the pullback functor from $X$ to $\tX$. We set 
$\pi_1^{\FIsocd} (X,\tilde\xi):=\Aut \xi$; usually, we write simply $\pi_1^{\FIsocd} (X)$. 
This is an affine group scheme over $\BQbar_p\,$, and one has a canonical equivalence 
$\FIsocd (X )\otimes_{\BQ_p}\BQbar_p \iso\Rep_{\BQbar_p} (\pi_1^{\FIsocd} (X))$, where $\Rep_{\BQbar_p}$ is the category of finite-dimensional representations over $\BQbar_p\,$.

One has a canonical epimorphism 
\begin{equation}   \label{e:crys to usual}
\pi_1^{\FIsocd} (X)\epi\Pi
\end{equation}
defined as follows. For every normal open subgroup $U\subset\Pi$, the Tannakian category $\FIsocd (X )$ identifies with the category of $(\Pi/U)$-equivariant objects of $\FIsocd (\tX /U )$, denoted by $\FIsocd (\tX /U )^{\Pi/U}$. So
\begin{equation}   \label{e:supset}
\FIsocd (X )=\FIsocd (\tX /U )^{\Pi/U}\supset (\Vect_{\BQ_p})^{\Pi/U}=\Rep_{\BQ_p} (\Pi/U).
\end{equation}
The restriction of $\xi$ to $\Rep_{\BQbar_p} (\Pi/U)$ is the tautological fiber functor, so the inclusion \eqref{e:supset} induces an epimorphism 
$\pi_1^{\FIsocd} (X)\epi\Pi/U$. In the limit, one gets \eqref{e:crys to usual}.

For each open subgroup $U\subset\Pi$ one has the group $\pi_1^{\FIsocd} (\tX /U)$ defined similarly to $\pi_1^{\FIsocd} (X)$; by \cite[Prop.~B.7.6(ii)]{DK}, one has a canonical isomorphism
\begin{equation}    \label{e:replacing Pi by U}
\pi_1^{\FIsocd} (\tX /U)\iso\pi_1^{\FIsocd} (X)\times_{\Pi}U.
\end{equation}

Define $\pi_1^{\FIsocd} (\tX )$ to be the projective limit of $\FIsocd (\tX /U )$ with respect to $U$. Then one has a canonical exact sequence
\begin{equation}   \label{e:pi1Fisocd of tilde X}
0\to \pi_1^{\FIsocd} (\tX )\to\pi_1^{\FIsocd} (X )\to\Pi\to 0.
\end{equation}

\begin{prop} \label{p:Crew}
The neutral connected component of $\pi_1^{\FIsocd} (X )$ equals $\pi_1^{\FIsocd} (\tX )$.
\end{prop}

This result is due to R.~Crew \cite{Crew-mono}. For a proof see \cite[Prop.~B.7.6(i)]{DK}.

\subsection{Frobeniuses}
\subsubsection{The isomorphism $\Fr_X^*\iso\Id$}
The absolute Frobenius $\Fr_X :X\to X$ induces a tensor functor $\Fr_X^*:\FIsocd (X)\to\FIsocd (X)$. One has a canonical tensor isomorphism 
$F:\Fr_X^*\iso\Id_{\FIsocd (X)}$ (this is the $F$ from the word ``$F$-isocrystal"). For each $n\in\BN$ it induces a tensor isomorphism
\begin{equation}   \label{e:F^n}
F^n: (\Fr_X^n)^*\iso\Id_{\FIsocd (X)}\, .
\end{equation}

\subsubsection{The case $X=\Spec\BF_{p^n}\,$}   \label{sss:Spec finite field}
If $X=\Spec\BF_{p^n}$ then \eqref{e:F^n} is a tensor automorphism of $\Id_{\FIsocd (X)}\,$. It defines a $\BQbar_p$-point of (the center of) 
$\pi_1^{\FIsocd} (\Spec\BF_{p^n})$, which is called the \emph{geometric Frobenius}. 

\subsubsection{Remark}
In fact, according to the easy Proposition~3.4.2.1 of \cite[Ch.~VI]{Sa}, the group scheme $\pi_1^{\FIsocd} (\Spec\BF_{p^n})$ identifies with the pro-algebraic completion of the abstract group $\BZ$ so that the geometric Frobenius defined in \S\ref{sss:Spec finite field} is equal to the image of $1\in\BZ$.

\subsubsection{General case}   \label{sss: Frobenius general}
Now let $X$ be arbitrary. Then we will associate to any  $\tilde x\in |\tX |$ a $\pi_1^{\FIsocd}(\tX)$-conjugacy class in $\pi_1^{\FIsocd}(X)$.

Recall that $\pi_1^{\FIsocd}(X)$ is a shorthand for $\pi_1^{\FIsocd}(X,\tilde\xi )$, where $\tilde\xi :\FIsocd (\tX )\to\Vect_{\BQbar_p}$ is a fiber functor. Define $\FIsocd (\{\tilde x\})$ similarly to \eqref{e:FIsoc on tilde X}. Let $R:\FIsocd (X)\to\FIsocd (\{\tilde x\})$ be the pullback functor. Given a fiber functor $\tilde\eta:\FIsocd (\{\tilde x\})\to\Vect_{\BQbar_p}$ and an isomorphism $f:\tilde\xi\iso\tilde\eta\circ R$, let
$F_{\tilde x}^{\tilde\eta ,f}\in\pi_1^{\FIsocd}(X,\tilde\xi )$ be the image of the geometric Frobenius element of $\pi_1^{\FIsocd} (\{x\},\tilde\eta )$ under the composition
\[
\pi_1^{\FIsocd} (\{ x\},\tilde\eta )\to \pi_1^{\FIsocd}(X,\tilde\eta\circ R )\iso\pi_1^{\FIsocd}(X,\tilde\xi ).
\]
Then the $\pi_1^{\FIsocd}(\tX )$-conjugacy class of $F_{\tilde x}^{\tilde\eta ,f}$ does not depend on $\tilde\eta ,f$. We call it the geometric Frobenius of $\tilde x$ and denote it by $F_{\tilde x}\,$. It is easy to check that the map $\tilde x\mapsto F_{\tilde x}$ is $\Pi$-equivariant.

\subsection{The groups $\hat\Pi_{\lambda}$ and $\hat\Pi_{(\lambda)}$ for $\lambda |p$}    \label{lambda |p}
\subsubsection{The Tannakian subcategory $\cT_p (X)\subset\FIsocd (X)$}  \label{sss:cT_p}
Let $\cT_p (X)\subset\FIsocd (X)$ be the full subcategory of semisimple objects $M\in\FIsocd (X)$ such that 
for every connected finite etale covering $\pi :X'\to X$, 
the determinant of each irreducible component of $\pi^*M\otimes_{\BQ_p}\BQbar_p\in\FIsocd (X')\otimes_{\BQ_p}\BQbar_p$ has finite order. 
In fact, it is enough to check the latter property when $\pi$ is an isomorphism: this is proved similarly to Proposition~\ref{p:Weil-2} but using \cite[Lemma 6.1]{Abe2011} instead of \cite[Prop.~1.3.4]{De}. In \S\ref{sss:cT_|p} we will see that the subcategory $\cT_p (X)\subset\FIsocd (X)$ is Tannakian.

\subsubsection{The field $\BQbar_\lambda\,$}
Fix an algebraic closure $\BQbar$ of $\BQ$. For each non-Archimedean place $\lambda$ define the field $\BQbar_\lambda$ just as in \S\ref{sss:lambda-completion}: namely, $\BQbar_\lambda$ is the union of the completions $E_\lambda$ for all subfields $E\subset\BQbar$ finite over $\BQ$.

\subsubsection{$\cT_\lambda (X)$, $\hat\Pi_{\lambda}\,$, and $\hat\Pi_{(\lambda)}$ for $\lambda |p$} \label{sss:cT_|p}
Let $\lambda$ be a $p$-adic place of $\BQbar$. Then $\BQbar_\lambda$ is an algebraic closure of~$\BQ_p\,$. Set
\begin{equation}   \label{e:cT lambda}
\cT_\lambda (X):=\cT_p (X)\otimes_{\BQ_p}\BQbar_\lambda\, ,
\end{equation}
where  $\cT_p (X)$ is as in \S\ref{sss:cT_p}. Equivalently, $\cT_\lambda (X)\subset \FIsocd (X)\otimes_{\BQ_p}\BQbar_\lambda$ is the 
full subcategory of semisimple objects $M\in\FIsocd (X)\otimes_{\BQ_p}\BQbar_\lambda$ with the following property: for every connected finite etale covering $\pi :X'\to X$, the determinant of each irreducible component of $\pi^*M$ has finite order. (It is enough to check this property if $\pi$ is an isomorphism.)

If we fix a fiber functor $\tilde\xi :\FIsocd (\tX )\to\Vect_{\BQbar_\lambda}\,$, then one has the affine group scheme $\pi_1^{\FIsocd}(X,\tilde\xi )$ over $\BQbar_\lambda\,$. Let $\hat\Pi_{\lambda}$ denote the maximal pro-semisimple quotient of $\pi_1^{\FIsocd}(X,\tilde\xi )$. Then one has a canonical fully faithful tensor functor $\Rep_{\BQbar_\lambda} (\hat\Pi_{\lambda})\mono \FIsocd (X)\otimes_{\BQ_p}\BQbar_\lambda$ whose essential image equals $\cT_\lambda (X)$. So  $\cT_\lambda (X)$ is a Tannakian subcategory of $\FIsocd (X)\otimes_{\BQ_p}\BQbar_\lambda$ canonically equivalent to $\Rep_{\BQbar_\lambda} (\hat\Pi_{\lambda})$. It follows that the subcategory $\cT_p (X)\subset\FIsocd (X)$ from 
\S\ref{sss:cT_p} is Tannakian.

The definition of $\hat\Pi_{\lambda}$ involves a choice of $\tilde\xi$. But Proposition~\ref{p:Crew} implies that $\hat\Pi_{\lambda}$ does not depend on $\tilde\xi$ if considered as an object of the ``coarse" groupoid $\Pross (\BQbar_\lambda )$. 

Let $\hat\Pi_{(\lambda)}\in\Pross (\BQbar )$ denote the image of $\hat\Pi_{\lambda}\in\Pross (\BQbar_\lambda )$ under the functor inverse to the equivalence $\Pross (\BQbar )\iso \Pross (\BQbar_\lambda )$.

By Proposition~\ref{p:Crew}, one has
\begin{equation}  \label{e:pi_0=Pi}
\pi_0(\hat\Pi_{(\lambda})=\Pi .
\end{equation}

If $\Pi'\subset\Pi$ is an open subgroup one can define $\Pi'_{(\lambda)}$ by applying the previous construction to $\tX/\Pi'$ instead of $X$; on the other hand, $\Pi'_{(\lambda)}$ canonically identifies with $\Pi_{(\lambda)}\times_{\Pi}\Pi'$, see formula \eqref{e:replacing Pi by U}.

\subsubsection{The scheme $[\hat\Pi_{(\lambda)}]$}
Recall that for any pro-reductive group $G$, the symbol $[G]$ denotes the GIT quotient of $G$ by the conjugation action of the neutral connected component $G^\circ$ (see \S\ref{sss:charFr}-\ref{sss:coarsecat}). In particular, we have the scheme $[\hat\Pi_{(\lambda)}]$ over $\BQbar$. By \eqref{e:pi_0=Pi}, we have a canonical action of $\Pi$ on $[\hat\Pi_{(\lambda)}]$ and a canonical $\Pi$-equivariant map 
$[\hat\Pi_{(\lambda)}]\epi\Pi$.

By \S\ref{sss: Frobenius general}, we have the geometric Frobenius map
\begin{equation}  \label{e:tilde X to []}
|\tX |\to [\hat\Pi_{(\lambda)}] (\BQbar_\lambda ), \quad \tilde x\mapsto F_{\tilde x\,};
\end{equation}
this map is $\Pi$-equivariant.

\subsubsection{The set $\tilde\Pi_{\Fr}$\,}

In \S\ref{sss:Pi_Fr} we defined a subset $\Pi_{\Fr}\subset\Pi$. Set $\tilde\Pi_{\Fr}:=\BN\times |\tX |$. One has the canonical $\Pi$-equivariant surjection
\begin{equation}   \label{e:tildePiFr to Pi}
\tilde\Pi_{\Fr}\epi\Pi_{\Fr} 
\end{equation}
that takes $( n,\tilde x )\in\tilde\Pi_{\Fr}$ to the $n$-th power of the geometric Frobenius of $\tilde x$ in $\Pi$. 

The map 
$\tilde\Pi_{\Fr}\to\Pi\times\tX$ induced by \eqref{e:tildePiFr to Pi} is injective, and it may be convenient to think of $\tilde\Pi_{\Fr}$ as a subset of $\Pi\times\tX$ rather than as $\BN\times |\tX |$. If $\Pi'\subset\Pi$ is an open subgroup then the subset $\tilde\Pi'_{\Fr}\subset\Pi'\times\tX$ is equal to 
$\tilde\Pi_{\Fr}\times_{\Pi}\Pi'$.

The map  \eqref{e:tildePiFr to Pi} has a canonical lift to a $\Pi$-equivariant map
\begin{equation}   \label{e:tildePiFr to []}
\tilde\Pi_{\Fr}\to [\hat\Pi_{(\lambda)}] (\BQbar_\lambda ), 
\quad( n,\tilde x )\mapsto F_{\tilde x}^n,
\end{equation}
where $F_{\tilde x}$ is as in \eqref{e:tilde X to []}.

\subsection{Recollections on existence of companions. II}
For each non-Archimedean place $\lambda$ of $\BQbar$ we have the Tannakian category $\cT_\lambda (X)$ over $\BQbar_\lambda\,$: it was defined in \S\ref{sss:rep hat-Pi-lambda} if $\lambda$ is coprime to $p$,  and by formula~\eqref{e:cT lambda} if~$\lambda |p$. Here is a generalization of Theorem~\ref{t:Langl}.  

\begin{theorem}  \label{t:AE}
Let $\lambda$ be a non-Archimedean place of $\BQbar$, and let $\E\in\cT_\lambda (X)$ . Then

(i) for each $x\in |X|$, the polynomial $\det (1-t\cdot F_x\, ,\E )$ belongs to $\BQbar [t]$;

(ii) for every non-Archimedean place $\lambda'$ of $\BQbar$  coprime to $p$, the object $\E$ has a $\lambda'$-companion, i.e.,
an object $\E'\in\cT_{\lambda'} (X)$ such that $\det (1-t\cdot F_x\, ,\E )=\det (1-t\cdot F_x\, ,\E' )$ for all $x\in |X|$; moreover, if $\dim X=1$ then $\E'$ exists even if $\lambda'|p$;

(iii) $\E'$ is unique up to isomorphism (if it exists);

(iv) if $\E$ is irreducible then so is $\E'$ (if $\E'$ exists).
\end{theorem}

If $\lambda$ and $\lambda'$ are coprime to $p$ this is Theorem~\ref{t:Langl}. The rest is due to Abe \cite{A} and Abe-Esnault \cite{AE}; in addition to \cite{La},  it is based on the  crystalline version of the Langlands correspondence for $GL(n)$ over function fields 
\cite[Thm. 4.2.2]{A}. The precise references to \cite{A,AE} are below.

If $\lambda |p$ then statement (i) of Theorem~\ref{t:AE} is shown in the proof of \cite[Thm. 4.2]{AE}.

If $\lambda |p$ and $\lambda'$ is coprime to $p$ then statement (ii) is a part of \cite[Thm. 4.2]{AE}; if $\lambda'|p$ and $\dim X=1$ this is \cite[Thm. 4.4.1]{A}.

If $\lambda'$ is coprime to $p$ then statement (iii) holds by \v {C}ebotarev density. If $\lambda'| p$ then (iii) follows from \cite[Prop.~A.3.1]{A}, which is applicable by \cite[Thm. 2.6]{AE}.

If $\lambda |p$ and $\lambda'$ is coprime to $p$ then statement (iv) is a part of \cite[Thm. 4.2]{AE}. A similar argument works for any $\lambda$ and $\lambda'$.

\begin{cor}  \label{c:algebraicity}
Let $\lambda$ be a place of $\BQbar$ dividing $p$. Then the map \eqref{e:tildePiFr to []} factors as 
\[
\tilde\Pi_{\Fr}\epi\Pi_{\Fr}\to [\hat\Pi_{(\lambda)}] (\BQbar)\mono [\hat\Pi_{(\lambda)}] (\BQbar_\lambda ),
\]
where the first arrow is the map \eqref{e:tildePiFr to Pi}.
\end{cor}

\begin{proof}
By Theorem~\ref{t:AE}(i), the image of the map \eqref{e:tildePiFr to []} is contained in $[\hat\Pi_{(\lambda)}] (\BQbar )$. Applying Theorem~\ref{t:AE}(ii) for some $\lambda'$ coprime to $p$, we see that the map \eqref{e:tildePiFr to []} factors through $\Pi_{\Fr}\,$.
\end{proof}

\subsection{The theorem}
For any non-Archimedean place $\lambda$ of $\BQbar$, we have a diagram of sets
\begin{equation}   \label{e:2diagram of sets}
\Pi_{\Fr}\to [\hat\Pi_{(\lambda)}] (\BQbar)\epi\Pi\, ;
\end{equation}
namely, the map $\Pi_{\Fr}\to [\hat\Pi_{(\lambda)}](\BQbar)$ comes from Corollary~\ref{c:algebraicity} if $\lambda |p$ and from Proposition~\ref{p:algebraicity} if 
$\lambda$ is coprime to~$p$.

\begin{theorem}  \label{t:p-main}
Assume that $\dim X=1$. Let $\lambda$ and $\lambda'$ be non-Archimedean places of $\BQbar$. Then there exists a unique isomorphism 
$\hat\Pi_{(\lambda)}\iso\hat\Pi_{(\lambda')}$ in the category $\Pross (\BQbar )$ 
which sends diagram \eqref{e:2diagram of sets} to a similar diagram $\Pi_{\Fr}\to [\hat\Pi_{(\lambda')}] (\BQbar)\epi\Pi\,$. 
\end{theorem}

\begin{proof}
In \S\ref{s:main} we already proved this if $\lambda$ and $\lambda'$ are coprime to $p$. It remains to consider the case that $\lambda' |p$ and 
$\lambda$ is coprime to $p$. It is treated similarly to \S\ref{ss:main in terms of K}-\ref{ss:arbitrary curves}.  The only difference is as follows. At the end of \S\ref{ss:main in terms of K} we used Proposition~\ref{p:2simply connected} to conclude that both $\hat\Pi_{(\lambda )}$ and 
$\hat\Pi_{(\lambda' )}$ are in $\Pross_\Pi^{prod}$. But now we are assuming that $\lambda' |p$, so Proposition~\ref{p:2simply connected} only tells us that $\hat\Pi_{(\lambda )}$ is in $\Pross_\Pi^{prod}$. However, this implies that $\hat\Pi_{(\lambda' )}$ is in $\Pross_\Pi^{prod}$ by Theorem~\ref{t:variant of KLV}(iii).
\end{proof}

If either $\lambda |p$ or $\lambda'|p$ then I cannot remove the assumption $\dim X=1$ because it appears in the second part of Theorem~\ref{t:AE}(ii).

\begin{cor}   \label{c:p-simply connected}
If $\dim X=1$ then $\hat\Pi_{(\lambda)}^\circ$ is simply connected for every non-Archimedean place $\lambda$ of $\BQbar$.
\end{cor}

\begin{proof}
By Theorem~\ref{t:p-main}, we can assume that $\lambda$ is coprime to $p$. In this case the statement follows from the easy 
Proposition~\ref{p:2simply connected}.
\end{proof}

\begin{quest}
Let $\dim X>1$ and $\lambda |p$. Is it still true that $\hat\Pi_{(\lambda)}^\circ$ is simply connected? Equivalently, is it true that the group 
$\pi_1^{\FIsocd} (\tX )$ from the exact sequence \eqref{e:pi1Fisocd of tilde X} is simply connected? 
\end{quest}

\appendix

\section{Proof of Proposition~\ref{p:coarse hom}}  \label{s:coarse hom}
\subsection{The goal of this Appendix}  \label{ss:goal}
Let $E$ be an algebraically closed field. Let $G_1$ and $G_2$ be group schemes over $E$. Recall that $\Hom_{\coarse} (G_1\, ,G_2)$ denotes the quotient of the set $\Hom(G_1\, ,G_2)$ by the conjugation action of $G_2^\circ (E)$. The goal of this Appendix is to prove the following statement, which is clearly equivalent to 
Proposition~\ref{p:coarse hom}.

\begin{prop}   \label{p:2coarse hom}
Let $N$ denote the set of all normal subgroups $H\subset G_2$ such that $G_2/H$ has finite type. Then the canonical map
\begin{equation} \label{e:bijection}
\Hom_{\coarse} (G_1\, ,G_2)\to\underset{H\in N}{\underset{\longleftarrow}{\lim}}\Hom_{\coarse} (G_1\, ,G_2/H)
\end{equation}
is bijective.
\end{prop}

\subsection{A set-theoretical lemma}
\begin{rem}
Let us recall that the projective limit of an \emph{uncountable} directed system of non-empty sets with respect to surjective maps can be empty. For instance, for any finite subset $S\subset\BR$ let $\Inj (S,\BN)$ denote the set of injections $S\mono\BN$. The projective limit of the sets 
$\Inj (S,\BN)$ is the set of injections $\BR\mono\BN$, which is empty.
\end{rem}

Now let $N$ be a poset in which any two elements $H_1\, , H_2$ have an infinum (e.g., the poset $N$ from Proposition~\ref{p:2coarse hom} has this property). We will use the symbol ``$\subset$" for the order relation and the symbol ``$\cap$" for the infinum.

\begin{lem}  \label{l:set-theor}
Let $(Y_H)$ be a projective system of sets indexed by $H\in N$ in which all transition maps\footnote{Here we follow the convention that if $H'\subset H$ then we have a transition map $Y_{H'}\to Y_H$ (rather than $Y_H\to Y_{H'}$).} are surjective. Suppose that each $Y_H$ is non-empty, but their projective limit is empty. In addition, suppose that all maps
\begin{equation}  
Y_{H\cap H'}\to Y_H\times Y_{H'}\, , \quad H,H'\in N
\end{equation}
are injective. Then for some $H'\in N$ there exists a sequence of pairs
\begin{equation}   \label{e:sequence of pairs}
(H_i\, ,\xi_i), \quad i\in\BN, \; H_i\in N,\;\xi_i\in Y_{H_i}
\end{equation}
with the following properties:

(a) $H_i\supset H_{i+1}$ and $\xi_{i+1}\mapsto\xi_i$ for all $i\in\BN$;

(b) set $F_i:=\{y\in Y_{H'}\,|\, (\xi_i\, ,y)\in Y_{H_i\cap H'}\}$, where $Y_{H_i\cap H'}$ is a shorthand for $\im (Y_{H_i\cap H'}\to Y_{H_i}\times Y_{H'})$; 
then the inclusions
\begin{equation}   \label{e:descending chain}
F_1\supset F_2\supset\ldots
\end{equation}
are strict.
\end{lem}

\begin{proof}
Consider all pairs $(S,\xi )$, where $S\subset N$ is a subset closed under finite infimums and 
$$\xi\in\underset{H\in S}{\underset{\longleftarrow}{\lim}} Y_H\,.$$ Such pairs form a poset satisfying the conditions of Zorn's lemma. So it has a maximal element $(S_{\max}\, ,\xi_{\max})$. By assumption, $S_{\max}\ne N$. 

Now let $H'\in N\setminus S_{\max}$. For each $H\in S_{\max}$ let $\xi_H\in Y_H$ be the image of $\xi$ and let
\[
F_H:=\{y\in Y_{H'}\, |\, (\xi_H\, ,y)\in Y_{H\cap H'}\}.
\]
It is clear that  if $H_1\subset H_2$ then $F_{H_1}\supset F_{H_2}\,$. Since the transition maps in our projective systems are surjective, each
$F_H$ is non-empty. On the other hand, the intersection of all sets $F_H$, $H\in S_{\max}\,$, is empty because the pair $(S_{\max}\, ,\xi_{\max})$ is maximal. So there exists a sequence of pairs \eqref{e:sequence of pairs} with properties (a) and (b).
\end{proof}

\subsection{Proof of surjectivity of the map \eqref{e:bijection}}
Fix an element of the set 
$$\underset{H\in S}{\underset{\longleftarrow}{\lim}} \Hom_{\coarse}(G_1\, ,G_2/H).$$
Then for each $H\in N$ we get an element of $\Hom_{\coarse}(G_1\, ,G_2/H)$; let $Y_H\subset\Hom (G_1\, ,G_2/H)$ denote its preimage. Clearly 
$Y_H$ is non-empty and is equipped with a transitive action of the group $$\Gamma_H:=(G_2^\circ/(H\cap G_2^\circ ))(E).$$
The pairs $(\Gamma_H , Y_H)$ form a projective system. If $H'\subset H$ then the transition map $\Gamma_{H'}\to\Gamma_H$ is clearly surjective, so the map $Y_{H'}\to Y_H$ is also surjective by transitivity of the action of $\Gamma_H$ on $Y_H\,$.

We have to show that the projective limit of the sets $Y_H$ is non-empty. Assume the contrary. Then we are in the situation of 
Lemma~\ref{l:set-theor}. Let $H_i\,$, $\xi_i\in Y_{H_i}\,$, and $F_i\subset Y_{H'}$ be as in the lemma. 
Let $K_i\subset \Gamma_{H'}$ denote the image of the stabilizer of $\xi_i$ in $\Gamma_{H_i\cap H'\,}$, then $F_i$ is an orbit of $K_i\,$ acting on 
$Y_{H'}\,$. The groups $K_i$ form a decreasing chain of algebraic subgroups of the algebraic group $\Gamma_{H'\,}$, so $K_i=K_{i+1}$ for $i$ big enough. Then $F_i=F_{i+1}\,$, so we get a contradiction.
\qed

\subsection{Proof of injectivity of the map \eqref{e:bijection}}
Let $f,\tilde f\in\Hom (G_1\, , G_2)$. For each $H\in N$ let $f_H\,,\tilde f_H\in\Hom (G_1\, ,G_2/H)$ denote the images of $f$ and $\tilde f$; 
set $\Gamma_{H}:=(G_2^\circ/(H\cap G_2^\circ ))(E)$ and 
let 
$Z_H$ denote the set of all $g\in \Gamma_H$ such that $\tilde f_H$ is $g$-conjugate to $f_H\,$. Each $Z_H$ is an algebraic variety. The  varieties $Z_H$ form a projective system. The problem is to show that if each variety $Z_H$ is non-empty then so is their projective limit.

If $H'\subset H$ set $Z_{H',H}:=\im (Z_{H'}\to Z_H)\,$. Note that the subset $Z_{H'}\subset\Gamma_{H'}$ is a left coset with respect to an algebraic subgroup  of $\Gamma_{H'}$. Since $Z_{H',H}$ is the image of $Z_{H'}$ under an algebraic group homomorphism $\Gamma_{H'}\to\Gamma_H\,$, we see that $Z_{H',H}$ is closed in $Z_H\,$. If $H''\subset H'$ then $Z_{H'',H}\subset Z_{H',H}\,$. Since $Z_H$ is Noetherian, there exists a subset 
$Y_H\subset Z_H$ such that $Z_{H',H}=Y_H$ for $H'$ small enough. Clearly $Y_H\ne\emptyset$. The sets $Y_H$ form a projective system of sets with surjective transition maps, and the problem is to show that the projective limit is non-empty.

Assume the contrary. Then we can apply Lemma~\ref{l:set-theor} and get an infinite decreasing chain \eqref{e:descending chain}. In our situation it is formed by Zariski-closed subsets of an algebraic variety, so we get a contradiction.
\qed

\section{On some results of S.~Mohrdieck, T.~A.~Springer, and J.~C.~Jantzen}   \label{s:twisted conjugacy}
We fix a ground field $E$, which is algebraically closed field and of characteristic 0. As before, the GIT quotient of a reductive group $G$ by the conjugation action of $G^\circ$ is denoted by $[G]$.

In \S\ref{sss:[G]} we recall an explicit description of $[G]$, which goes back to \cite{Mo,Sp}. In \S\ref{ss:Satz 9} we recall a mysterious theorem of J.~C.~Jantzen; then we use it to construct the automorphism promised in Example~\ref{ex:why Adams}.

\subsection{An explicit description of $[G]$} \label{sss:[G]}
We will recall an explicit description of $[G]$ for any reductive group $G$. If $G$ is connected it amounts to the usual identification of $[G]$ with~$T/W$.

Fix a Borel $B\subset G^\circ$ and a maximal subtorus $T\subset B$. Let $N_G(T)$ denote the normalizer of $T$ in $G$. Similarly, we have $N_G(B)$ and $N_{G^\circ}(T)$. Set $N_G(T,B):=N_G(T)\cap N_G(B)$. We have exact sequences
\begin{equation}  \label{e:NTB}
0\to T\to N_G(T,B)\to\pi_0 (G)\to 0,
\end{equation}
\[
0\to T\to N_{G^\circ}(T)\to W\to 0,
\]
where $W$ is the Weyl group of $G^\circ$. They are both contained in the exact sequence
\begin{equation}  \label{e:big_exact}
0\to T\to N_G(T)\to\pi_0 (G) \ltimes W\to 0.
\end{equation}

Let $[N_G(T,B)]$ denote the GIT quotient of $N_G(T,B)$ by the conjugation action of $T=N_G(T,B)^\circ$.

For any $\sigma\in\pi_0 (G)$ let  $[G]_\sigma$ denote the preimage of $\sigma$ in $[G]$. Define $N_G(T,B)_\sigma$ and 
$[N_G(T,B)]_\sigma$ similarly. 

The group $\pi_0 (G)$ acts on $W$. Let $W^\sigma\subset W$ denote the subgroup of elements fixed by $\sigma\in\pi_0 (G)$. Looking at \eqref{e:big_exact}, we see that  if $w\in W^\sigma$ and $g$ belongs to the preimage of $w$ in $N_{G^\circ}(T)$ then conjugation by $g$ preserves $N_G(T,B)_\sigma\,$. Thus one gets an action of  $W^\sigma$ on $[N_G(T,B)]_\sigma\,$. The GIT quotient $[N_G(T,B)]_\sigma/W^\sigma$ canonically maps to $[G]_\sigma\,$.

\begin{prop}   \label{p:Mohrdieck}
The canonical map $[N_G(T,B)]_\sigma/W^\sigma\to [G]_\sigma$ is an isomorphism for any $\sigma\in\pi_0 (G)\,$.
\end{prop}

\begin{proof}
The statement easily reduces to the case where $\pi_0 (G)$ is generated by $\sigma$ and the epimorphism $G\epi\pi_0 (G)$ admits a splitting. Then one reduces to the case where $G^\circ$ is semisimple and simply connected. In this crucial case the statement is just a reformulation of
 \cite[Thm.~1.1(i)]{Mo} or  \cite[Thm.~1]{Sp}. 
\end{proof}

\subsubsection{Remark}
In \cite{Mo, Sp} one can find more information about $[G]$ (e.g., if $G^\circ$ is semisimple and simply connected then for each $\sigma$, the scheme  $[G]_\sigma$ is isomorphic to some affine space).

\subsubsection{Remark}
The scheme $[N_G(T,B)]$ does not depend on the choice of $(T,B)$ (up to \emph{canonical} isomorphism). The same is true for the map $[N_G(T,B)]\to [G]$ and the action of 
 $W^\sigma$ on $[N_G(T,B)]_\sigma\,$.

\subsubsection{A reformulation of Proposition~\ref{p:Mohrdieck}}   \label{sss:fancy}
Consider the disjoint union $\bigsqcup\limits_\sigma W^\sigma$ as a group scheme over
the (finite)  scheme $\pi_0 (G)$. This group scheme acts on $[N_G(T,B)]$ over $\pi_0 (G)$, and Proposition~\ref{p:Mohrdieck} says that the GIT quotient by its action identifies  with~$[G]$.

\subsubsection{Generalization to pro-reductive groups}
The description of $[G]$ from \S\ref{sss:fancy} immediately implies a similar description of $[G]$ if $G$ is pro-reductive rather than reductive (of course, in this case $\pi_0 (G)$ is pro-finite rather than finite).

\subsection{Satz 9 of C.~J.~Jantzen's thesis} \label{ss:Satz 9}
In \S\ref{sss:Satz 9} we formulate a remarkable theorem of C.~J.~Jantzen.
In \S\ref{sss:explaining Example} we use it to construct the automorphism promised in Example~\ref{ex:why Adams}.

\subsubsection{Jantzen's theorem}   \label{sss:Satz 9}
Let $H$ be a connected simply connected almost-simple group over an algebraically closed field of characteristic 0
equipped with a pinning (in the sense of \S\ref{sss:pinnings}); in particular, the pinning involves a choice of a maximal subtorus $T\subset H$.
Let $\sigma$ be a nontrivial automorphism of the Dynkin diagram\footnote{Note that in this situation $H$ is simply laced, while $H^\sigma$ is not.} of $H$,  then $\sigma$ acts on $H$ and $T$.  Let $H^\sigma$ denote the subgroup of fixed points. Let $H_\sigma$ denote the simply connected group whose root system is dual to that of $H^\sigma$. As explained in \cite{JJ,KLP}, the maximal torus of $H_\sigma$ canonically identifies with
$T_\sigma :=\Coker (\sigma-1: T\to T)$, and  dominant weights of $H_\sigma$ identify with $\sigma$-invariant  dominant weights of $H$. Let $\omega$ be such a weight, and let
$V_\omega$ (resp.~$V'_\omega$) be the corresponding finite-dimensional irreducible $H$-module (resp. $H_\sigma$-module). 

Let $H\langle\sigma\rangle$ denote the semidirect product of $H$ and the cyclic subgroup of $\Aut H$ generated by $\sigma$. There is a unique way to extend the action of $H$ on $V_\omega$ to an action of $H\langle\sigma\rangle$ on the same vector space so that 
the action of $\sigma\in H\langle\sigma\rangle$ on the highest weight subspace of $V_\omega$ is trivial. Restricting the character of this representation of $H\langle\sigma\rangle$ to $T\cdot \sigma\subset H\langle\sigma\rangle$, one gets a regular function on $T_\sigma\,$.
A mysterious theorem from C.~J.~Jantzen's  thesis \cite[p.30, Satz~9]{JJ} says that this function is equal to the character of $V'_\omega\,$. The short and easily available article \cite{KLP} contains a discussion of this theorem; in particular, it explains how to deduce it from the properties of Lusztig's canonical basis assuming that $H$ is not of type $A_{2n}\,$.

\subsubsection{The automorphism promised in Example \ref{ex:why Adams}}   \label{sss:explaining Example}
Let us now assume that $H$ has type $A_{2n}\,$, i.e., $H=SL(2n+1)$ (this is the only case where $\sigma$ changes the ``color" of the vertices of the Dynkin diagram; in other words, the distance from any vertex of the Dynkin diagram to its $\sigma$-image is {\it odd}). 

A $\sigma$-invariant weight $\omega$ of $H$ defines a weight of $H_\sigma=Sp(2n)$; restricting the latter to the center of $H_\sigma\,$, one gets a character $\chi_\omega$ of $\BZ/2\BZ$. Explicitly, if $\omega =\sum\limits_{i=1}^n k_i (\omega_i+\omega_{2n+1-i})$ (where $\omega_1,\ldots ,\omega_{2n}$ are the fundamental weights of $H$ numbered in the usual way) then
\[
\chi_\omega (m)=\prod_{i=1}^n (-1)^{mik_i}, \quad m\in\BZ/2\BZ\, .
\]

Now set $G:=H\langle\sigma\rangle$ and define an additive automorphism $\phi$ of  the Grothendieck semigroup $K^+(G)$ as follows:
if $V$ is an irreducible representation of $G$ whose restriction to $H$ is reducible then set $\phi ([V]):=[V]$; if the restriction is irreducible and has highest weight $\omega$ then $\phi ([V]):=[V\otimes\chi_\omega ]$, where $\chi_\omega$ is considered as a character of $G/H$. 
It is clear that $\phi$ induces a nontrivial automorphism of $K^+(G/Z)$, where $Z$ is the center of $H$. Let us prove the other claims from Example~\ref{ex:why Adams}.

\begin{lem}    \label{l:why Adams}
(i) $\phi :K^+(G)\to K^+(G)$ preserves the semiring structure on $K^+(G)$.

(ii) The automorphism of $K(G)$ induced by $\phi$ does not commute with the Adams operation $\psi^2:K(G)\to K(G)$.
\end{lem}

\begin{proof}
(i) Let $\underline{G}$ denote the GIT quotient of $G$ by the conjugation action of $G$.
The algebra $K(G)\otimes E$ identifies with the algebra of regular functions on $\underline{G}$. 
Statement (i) essentially says that the automorphism of the vector space $K(G)\otimes E$ comes from an automorphism $f\in\Aut\underline{G}$.
Jantzen's theorem formulated in \S\ref{sss:Satz 9} implies that this is true for the automorphism $f\in\Aut\underline{G}$ described below.

$\underline{G}$ has two connected components; the first one is the quotient of $[H]=T/W$ by the action of 
$\langle\sigma\rangle$ and the second one identifies by Proposition~\ref{p:Mohrdieck} with $T_\sigma/W_\sigma\,$. The automorphism $f$ acts as the identity on the first component, and it acts on $T_\sigma/W_\sigma$ as multiplication by the $W_\sigma$-invariant element
$\varepsilon\in T_\sigma$ of order $2$, which can be described in three equivalent ways:

(a) $\varepsilon\in T_\sigma$ is the image of the central element $-1\in Sp(2n)=H$; 

(b)  $\langle\omega_I+\omega_{2n+1-i\,},\varepsilon\rangle=(-1)^i$ for $i\le n$;

(c) $\varepsilon =\pi (\check\omega_n (-1))$, where $\check\omega_n :\BG_m\to T/Z$ is the $n$-th fundamental coweight and $\pi :T/Z\to T_\sigma$ is induced by the projection $T\to T_\sigma\,$.

\medskip

(ii) Commutation with the Adams operation $\psi^2:K(G)\to K(G)$ would mean that the above automorphism $f\in\Aut\underline{G}$ satisfies $f(g^2)=f(g)^2$ for all $g\in\underline{G}$. This is false if $g$ is the image of $\sigma\in H\langle\sigma\rangle\subset G$. Indeed, in this case 
$f(g)^2$ is equal to the image of $N(\varepsilon)\in T$ in $T/W$, where $N:T_\sigma\to T$ is the norm map; note that $N(\varepsilon)\ne 1$.
\end{proof}

\section{On finite groups all of whose Sylow subgroups are cyclic}   \label{s:Zassenhaus}
This Appendix is related to Lemma~\ref{l:cyclic_Sylow} and Proposition~\ref{p:coarse extensions}.

\subsection{A theorem of Zassenhaus}
The next theorem is due to Zassenhaus.

\begin{theorem}  \label{t:Zassenhaus}
Let $H$ be a finite group. Then the following conditions are equivalent:

(i) all Sylow subgroups of $H$ are cyclic;

(ii) $H$ is a semidirect product of two cyclic subgroups of coprime orders.
\end{theorem}
 
For the proof of the nontrivial implication (i)$\Rightarrow$(ii) see \cite[Satz~5]{Z} or  \cite[Thm.~9.4.3]{H}.

\subsection{Coarse extensions}
\subsubsection{Subject of this subsection}
Let $\Gamma$ be a group and $Z$ an abelian group equipped with $\Gamma$-action.
Let $\Ex (\Gamma ,Z)$ denote the groupoid whose objects are extensions of $\Gamma$ by $Z$ and whose morphisms are isomorphisms of extensions modulo inner automorphisms corresponding to elements of $Z$. Paraphrasing \S\ref{sss:[tildeGamma]}, we will define another groupoid $\Ex' (\Gamma ,Z)$ and a functor $\Ex (\Gamma ,Z)\to\Ex'(\Gamma ,Z)$. This functor turns out to be fully faithful if $\im (\Gamma\to\Aut Z)$ is a finite group satisfying the conditions of Theorem~\ref{t:Zassenhaus}; moreover, it is an equivalence if the group $\im (\Gamma\to\Aut Z)$ satisfies slightly stronger conditions, see Proposition~\ref{p:2coarse extensions} below.

Each of the two groupoids $\Ex (\Gamma ,Z)$ and $\Ex' (\Gamma ,Z)$ could be called the groupoid of coarse extensions of $\Gamma$ by $Z$. The good news is that under certain conditions they are the same.

\subsubsection{The groupoid $\Ex' (\Gamma ,Z)$ and the functor $\Ex (\Gamma ,Z)\to\Ex'(\Gamma ,Z)$}   \label{sss:coarse}
Let $\widetilde{\Ex} (\Gamma ,Z)$ denote the groupoid of sets $S$ equipped with the following pieces of structure:

(a) the map $S\to\Gamma$;

(b) the action of $\Gamma\times\BN$ on $S$ (here $\BN$ is considered as a monoid with respect to multiplication);

(c) the action of $Z$ on $S$;

(d) for every subgroup $\Gamma'\subset\Gamma$, a group structure on 
$(\Gamma'\underset{\Gamma}\times S)/A$, where  $A:=\Ker (Z\epi Z_{\Gamma'})$.

Given an extension $0\to Z\to\tilde\Gamma\to\Gamma\to 0$, let $|\tilde\Gamma |$ denote the quotient set of $\tilde\Gamma$ by the conjugation action of $Z$. Just as in \S\ref{sss:[tildeGamma]}, we equip $|\tilde\Gamma |$ with structures (a)-(d). Thus we get a functor
\begin{equation}   \label{e:quot by Gamma}
\Ex (\Gamma ,Z)\to\widetilde{\Ex} (\Gamma ,Z), \quad \tilde\Gamma\mapsto|\tilde\Gamma |\, .
\end{equation}

Now we will define a full subcategory $\Ex' (\Gamma ,Z)\subset\widetilde{\Ex} (\Gamma ,Z)$ containing the essential image of the functor \eqref{e:quot by Gamma}. Namely, an object 
$S\in\widetilde{\Ex} (\Gamma ,Z)$ belongs to $\Ex' (\Gamma ,Z)$ if the following conditions hold:

(i) the map $S\to\Gamma$ is $(\Gamma\times\BN)$-equivariant and $Z$-equivariant (we assume that $Z$ acts on $\Gamma$ trivially, $\Gamma$ acts on itself by conjugation and $n\in\BN$ acts on $\Gamma$ by raising to the $n$-th power);

(ii) the actions of $\Gamma$ and $Z$ on $S$ combine into an action of $\Gamma\ltimes Z$;

(iii) for each $\gamma\in\Gamma$ the action of $Z$ on the fiber $S_\gamma$ factors through 
$Z_\gamma :=\Coker (\gamma-1 :Z\to Z)$; moreover, the $Z_\gamma$-action makes $S_\gamma$ into a $Z_\gamma$-torsor;

(iv) for each $n\in\BN$, the map $S_\gamma\to S_{\gamma^n}$ that comes from the action of $\BN$ on $S$ is $Z_\gamma$-equivariant if the action of $Z_\gamma$ on $S_{\gamma^n}$ is defined via the norm map $N:Z_\gamma\to Z_{\gamma^n}\,$, $N:=1+\gamma+\ldots +\gamma^{n-1}$;

(v) in the situation of (d) the map $(\Gamma'\underset{\Gamma}\times S)/A\to\Gamma'$ and the action of $Z_{\Gamma'}$ on $(\Gamma'\underset{\Gamma}\times S)/A$ make 
$(\Gamma'\underset{\Gamma}\times S)/A$ into a central extension of $\Gamma'$ by $Z_{\Gamma'}\,$; moreover, the action of $\Gamma'\times\BN$ on $(\Gamma'\underset{\Gamma}\times S)/A$ that comes from (b) is the same as the one that comes from the central extension structure (i.e.,  $\Gamma'$ acts by conjugation and $n\in\BN$ acts by raising to the $n$-th power).

(vi) let $\gamma\in\Gamma$ and let $\Gamma',\Gamma''\subset\Gamma$ be subgroups such that $\gamma\Gamma'\gamma^{-1}\subset\Gamma''$; then the map 
$$(\Gamma'\underset{\Gamma}\times S)/\Ker (Z\epi Z_{\Gamma'})\to (\Gamma''\underset{\Gamma}\times S)/\Ker (Z\epi Z_{\Gamma''})$$
induced by the action of $\gamma$ on $S$ and the conjugation action of $\gamma$ on $\Gamma$ is a homomorphism with respect to the group structure (d).

Let us note that by (iii) and (v), the action of each $\gamma\in\Gamma$ on the fiber $S_\gamma$ is trivial. 

It is clear that the essential image of the functor \eqref{e:quot by Gamma} is contained in 
$\Ex' (\Gamma ,Z)$. Thus we have constructed a functor
\begin{equation}   \label{e:2quot by Gamma}
\Ex (\Gamma ,Z)\to\Ex' (\Gamma ,Z).
\end{equation}

\begin{prop}  \label{p:2coarse extensions}
Let $\Gamma_0:=\Ker (\Gamma\to\Aut Z)$.
Suppose that $\Gamma/\Gamma_0$ is a finite group satisfying the conditions of Theorem~\ref{t:Zassenhaus}. Then 

(a) the functor \eqref{e:2quot by Gamma} is fully faithful;

(b) it is an equivalence if in addition to the conditions of Theorem~\ref{t:Zassenhaus}, $\Gamma/\Gamma_0$ has the following property: for any prime $p$ and any $p$-subgroups $C_1\subset C_2\subset \Gamma/\Gamma_0$ such that $C_1\ne\{ 1\}$, the normalizers of $C_1$ and $C_2$ in $\Gamma/\Gamma_0$ are equal to each other.
\end{prop}

The proof of Proposition~\ref{p:2coarse extensions} is given in \S\ref{ss:2coarse extensions-proof}-\ref{ss:end 2coarse extensions-proof} below. 

\begin{rem}  \label{r:why generalization}
It is easy to see that faithfulness of the functor \eqref{e:2quot by Gamma} is equivalent to Lemma~\ref{l:cyclic_Sylow}.
\end{rem}

\begin{rem}  \label{r:cyclic case}
If $\Gamma$ is a finite cyclic group then it is straightforward to check that  the functor \eqref{e:2quot by Gamma} is an equivalence (one uses condition (iv) from \S\ref{sss:coarse}).
\end{rem}

\subsubsection{Reducing Proposition~\ref{p:2coarse extensions} to the case $\Gamma_0=\{ 1\}$} \label{ss:2coarse extensions-proof}
We have functors 
$$\Ex (\Gamma ,Z)\to\Ex' (\Gamma ,Z)\to\Ex(\Gamma_0 ,Z)^{\Gamma/\Gamma_0},$$
where $\Ex(\Gamma_0 ,Z)^{\Gamma/\Gamma_0}$ is the groupoid of central extensions 
$0\to Z\to\tilde\Gamma_0\to\Gamma_0\to 0$ equipped with an action of $\Gamma$ on $\tilde\Gamma_0$ compatible with the action of $\Gamma$ on $Z$ and the conjugation actions of $\Gamma$ on $\Gamma_0$ and $\Gamma_0$ on $\tilde\Gamma_0\,$. Fix 
$\xi\in \Ex(\Gamma_0 ,Z)^{\Gamma/\Gamma_0}\,$. Let $\Ex (\Gamma ,Z)_\xi$ and 
$\Ex' (\Gamma ,Z)_\xi$ denote the fibers over $\xi$. Let us prove that the functor
$\Ex (\Gamma ,Z)_\xi\to\Ex' (\Gamma ,Z)_\xi$ is fully faithful and under the assumption of Proposition~\ref{p:2coarse extensions}(b) it is an equivalence.

Associated to $\xi$ is an element $v\in H^3(\Gamma/\Gamma_0\, , Z)$ such that 
$v=0\Leftrightarrow\Ex (\Gamma ,Z)_\xi\ne\emptyset$ (namely, $v$ is the class of the crossed module $\tilde\Gamma_0\to\Gamma$). We can assume that $\Ex' (\Gamma ,Z)_\xi\ne\emptyset$ (otherwise there is nothing to prove). Then $v$ has zero restriction to any cyclic subgroup of 
$\Gamma/\Gamma_0\,$. Since all Sylow subgroups of $\Gamma/\Gamma_0$ are cyclic we see that $v=0$, so $\Ex (\Gamma ,Z)_\xi\ne\emptyset$.

Since $\Ex (\Gamma ,Z)_\xi$ and $\Ex' (\Gamma ,Z)_\xi$ are non-empty, they are torsors over the Picard groupoids $\Ex (\Gamma/\Gamma_0 ,Z)$ and $\Ex' (\Gamma/\Gamma_0 ,Z)$, respectively. It remains to show that the Picard functor $\Ex (\Gamma/\Gamma_0 ,Z)\to\Ex' (\Gamma/\Gamma_0 ,Z)$ is fully faithful and under the assumption of Proposition~\ref{p:2coarse extensions}(b) it is an equivalence.

We can assume that $\Gamma_0$ is trivial (so $\Gamma$ is a finite group satisfying the conditions of Theorem~\ref{t:Zassenhaus}).
Localizing at a prime $p$, we can also assume that $Z$ is a module over $\BZ_{(p)}\,$, where $\BZ_{(p)}$ is the localization of $\BZ$ at $p$.
We need the following lemma.

\begin{lem}   \label{l:coh lemma}
Let $\Gamma$ be a finite group satisfying the conditions of Theorem~\ref{t:Zassenhaus}. Let $Z$ be a $\BZ_{(p)}$-module equipped with a 
$\Gamma$-action. Fix a Sylow $p$-subgroup $\Syl_p\subset\Gamma$. Then for every $i>0$ the restriction map $H^i(\Gamma ,Z)\to H^i(\Syl_p\,  ,Z)$ identifies $H^i(\Gamma ,Z)$ with the group of all elements $a\in H^i (\Syl_p\, ,Z)$ that have the following property: for each subgroup $C\subset\Syl_p\,$, the image of $a$ in $H^i (C,Z)$ is invariant with respect to the normalizer of $C$ in $\Gamma$.
\end{lem}

\begin{proof}
By \cite[Ch.~XII, Thm.~10.1]{CE}, for \emph{any} finite group $\Gamma$, the group $H^i(\Gamma ,Z)$ identifies with the subgroup of  all elements $a\in H^i(\Syl_p\, ,Z)$ with the following property (which is called ``stability" in \cite[\S XII.9]{CE}): for any subgroups $C_1\, ,C_2\subset\Syl_p$ and 
any $\gamma$ such that 
\begin{equation}  \label{e:gamma-conjugacy}
\gamma C_1\gamma^{-1}=C_2\, ,
\end{equation}
the isomorphism $H^i(C_1\, ,Z)\iso H^i(C_2\, ,Z)$ corresponding to $\gamma$ takes 
$r_1 (a)$ to $r_2 (a)$, where $$r_m: H^i(\Syl_p\, ,Z)\to H^i(C_m\, ,Z)$$ is the restriction map. But \eqref{e:gamma-conjugacy} implies that 
$|C_1|=|C_2|$, and
if $\Syl_p$ is cyclic this implies that $C_1=C_2$ and $\gamma$ belongs to the normalizer of $C_1\,$.
\end{proof}

\subsubsection{End of the proof of Proposition~\ref{p:2coarse extensions}}   \label{ss:end 2coarse extensions-proof}
To prove Proposition~\ref{p:2coarse extensions}(a), it suffices to combine Remark~\ref{r:cyclic case} with Lemma~\ref{l:coh lemma} for $i=1$. 
To prove Proposition~\ref{p:2coarse extensions}(b), consider the functors
\[
\Ex (\Gamma ,Z)\overset{F}\longrightarrow\Ex' (\Gamma ,Z)\overset{G}\longrightarrow\Ex' (\Syl_p\, ,Z)^{N_p/\Syl_p}=\Ex (\Syl_p\, ,Z)^{N_p/\Syl_p}\, ,
\]
where $N_p$ is the normalizer of $\Syl_p$ in $\Gamma$. By assumption, $\Gamma$ has the following property: the normalizer of any nontrivial subgroup $C\subset\Syl_p$ equals $N_p\,$. So applying Lemma~\ref{l:coh lemma} for $i=2$, we see that the essential image of $G$ is equal to the essential image of $G\circ F$. Then one checks using the same property of $\Gamma$ that the categorical fiber of $G$ over 
$0\in\Ex (\Syl_p ,Z)^{N_p/\Syl_p}$ is trivial.    \qed

\section{The group $\hat\Pi$: inventory and summary of results}  \label{s:what we know}
Let $p$, $X$,  $\tX$, and $\Pi$ be as be as in \S\ref{sss:some}. In addition, assume that $X$ is smooth.

In \S\ref{sss:lambda-completion} we fixed an algebraic closure $\BQbar$ of $\BQ$ and defined an object $\hat\Pi_{(\lambda)}$ of the ``coarse" category $\Pross (\BQbar)$, which \emph{a priori} depends on the additional choice of a non-Archimedean place $\lambda$ of $\BQbar$ not dividing $p$. However, by Theorem~\ref{t:main}, $\hat\Pi_{(\lambda)}$  does not depend on $\lambda$, so we get a well-defined object $\hat\Pi\in \Pross (\BQbar)$, which we call \emph{the pro-semisimple completion of $\Pi$.} 

Recall that the neutral connected component $\hat\Pi^\circ$ is simply connected (see Proposition~\ref{p:2simply connected}) and the group $\hat\Pi/\hat\Pi^\circ$ canonically identifies with $\Pi$.

\subsection{Inventory}
\subsubsection{The Dynkin diagram of $\hat\Pi$}   \label{ss:Dyn}
Let $\Delta_{\hat\Pi}$ denote the Dynkin diagram of $\hat\Pi$ in the sense of \S\ref{sss:Dynkin diagrams}. It is a disjoint union of finite Dynkin diagrams; in particular, it is a \emph{forest}, i.e., a disjoint union of trees. The number of trees in this forest can be infinite.
By \S\ref{sss:Dynkin diagrams}, $\Delta_{\hat\Pi}$ is equipped with a continuous action of $\Pi$.

Let $T$ (resp. $Z$) denote the maximal torus (resp.~center) of $\hat\Pi^\circ$. Since $\hat\Pi^\circ$ is simply connected $T$ and $Z$ are uniquely determined by $\Delta_{\hat\Pi}\,$; namely,
\begin{equation}   \label{e:groups via Dynkin}
T=\Hom (P,\BG_m ), \quad Z=\Hom (P/Q,\BG_m ), \quad  T/Z=\Hom (Q,\BG_m ),
\end{equation}
where $P$ is the weight group of $\Delta_{\hat\Pi}$ and $Q\subset P$ is the subgroup generated by the roots. Let $W$ denote the Weyl group of $\Delta_{\hat\Pi}$. The action of $\Pi$ on $\Delta_{\hat\Pi}$ induces its action on $T,Z,W$.

\subsubsection{The extension of $\Pi$ by $Z(\BQbar )$}  \label{sss:tildePi}
By \S\ref{sss:Pross''}, we have a group extension
\begin{equation}   \label{e:2ext of pi}
0\to Z(\BQbar )\to \tilde\Pi\to\Pi\to 0
\end{equation}
defined as an object of the ``coarse" groupoid $\Ex (\Pi ,Z(\BQbar ))$, whose objects are extensions of $\Pi$ by $Z(\BQbar )$ and whose morphisms are isomorphisms of extensions modulo conjugations by elements of $Z(\BQbar )\,$.
Let $[\tilde\Pi ]$ denote the quotient of $\tilde\Pi$ by the conjugation action of $Z(\BQbar )$. We often consider totally disconnected compact topological spaces (e.g., $\Pi$ and $[\tilde\Pi ]$) as affine schemes over $\BQbar$.

\begin{rem}   \label{r:rigidity}
The extension \eqref{e:2ext of pi} may have nontrivial automorphisms. However, as an object of the above-mentioned ``coarse" groupoid 
$\Ex (\Pi ,Z(\BQbar ))$, it has no nontrivial automorphisms preserving the Frobenius data discussed in \S\ref{sss:Frobenius data} below. 
This follows from Proposition~\ref{p:rigidity}. This can also be proved by combining Lemma~\ref{l:cyclic_Sylow} with the injectivity claim in Remark \ref{r:injectivity} below.
\end{rem}

\subsubsection{Frobenius data}  \label{sss:Frobenius data}
Recall that $\hat\Pi$ is equipped with a canonical $\Pi$-equivariant map 
\begin{equation}   \label{e:0Frobenius data}
\Pi_{\Fr}\to [\hat\Pi ] (\BQbar )
\end{equation}  
with Zariski-dense image, whose composition with the projection
$[\hat\Pi ](\BQbar )\to\Pi/\Pi^\circ=\Pi$ is equal to the inclusion $\Pi_{\Fr}\mono \Pi$. Let us reformulate these ``Frobenius data" in more concrete terms.

For any $\gamma\in\Pi$ set $Z_\gamma:=\Coker (Z\overset{\gamma-1}\longrightarrow Z)$, $T_\gamma:=\Coker (T\overset{\gamma-1}\longrightarrow T)$ and $\cT_\gamma :=T_\gamma/W^\gamma$, where $W^\gamma\subset W$ is the subgroup of $\gamma$-invariants. The group $Z_\gamma$ acts on $\cT_\gamma$ by multiplication.

Each $\cT_\gamma$ is an affine scheme over $\BQbar$. As $\gamma\in\Pi$ varies, the schemes $\cT_\gamma$ form an affine scheme $\cT$ over $\Pi$. (More precisely, $\cT$ is the spectrum of the algebra of 
regular functions $f$ on $T\times\Pi$ with the following property: for each $\gamma\in\Pi$ the restriction of $f$ to $T\times\{\gamma\}$ factors through $\cT_\gamma\,$.) Similarly, the groups $Z_\gamma$ form a group scheme $\cZ$ over $\Pi$; the group scheme $\cZ$ acts on $\cT$ (viewed as a scheme over $\Pi$).

Let $\tilde\Pi$ be as in \eqref{e:2ext of pi}, then $[\tilde\Pi ]$ is a $\cZ$-torsor over $\Pi$. By Proposition~\ref{p:Mohrdieck}, \emph{$[\hat\Pi ]$ canonically identifies with the $[\tilde\Pi ]$-twist of $\cT$} (the notion of twist makes sense because $\cZ$ acts on $\cT$).
Using this identification, we can rewrite the map \eqref{e:0Frobenius data} as a $Z$-anti-equivariant map
\begin{equation}   \label{e:Frobenius data}
[\tilde\Pi ]\times_\Pi \Pi_{\Fr}\to \cT (\BQbar ).
\end{equation}

Equip $\Pi$ with the usual action of $\Pi\times\BN$, where $\BN$ is the multiplicative monoid of positive integers ($\Pi$ acts by conjugation and $n\in\BN$ acts by raising to the $n$-th power). This action canonically lifts to an action of $\Pi\times\BN$ on $\cT$; namely, the map 
$\cT_\gamma$ to $\cT_{\gamma^n}$ corresponding to the action of
$n\in\BN$ is induced by the homomorphism
$T_\gamma\to T_{\gamma^n}$ that takes $t$ to $t\cdot\gamma (t)\cdot\ldots\cdot\gamma^{n-1}(t)$. 
Then the map \eqref{e:Frobenius data} is $(\Pi\times\BN )$-equivariant.

\begin{rem}  \label{r:injectivity}
Fix $\gamma\in\Pi$. Let $[\tilde\Pi ]_\gamma$ denote the fiber of $[\tilde\Pi ]$ over $\gamma$. 
Let $B_\gamma$ denote the direct limit of the 
$\BQbar$-algebras of all functions $V\cap \Pi_{\Fr} \to\BQbar$, where $V$ runs through the set of all open subsets of $\Pi$ containing~$\gamma$. Given $\tilde\gamma\in [\tilde\Pi ]_\gamma\,$, let
$\tilde B_{\tilde \gamma}$ denote the direct limit of the 
$\BQbar$-algebras of all functions 
$\tilde V \times_\Pi \Pi_{\Fr}\to\BQbar$, where $\tilde V$ runs through the set of all open subsets of $ [\tilde\Pi ]$ containing~$\tilde\gamma$. Let $A$ denote the $\BQbar$-algebra of regular functions on $\cT$.  For each $\tilde\gamma\in [\tilde\Pi ]_\gamma\,$, the map \eqref{e:Frobenius data} gives rise to a homomorphism $f_{\tilde\gamma}:A\to\tilde B_{\tilde\gamma}\,$. It is easy to prove that $f_{\tilde\gamma} (A)\subset B_\gamma$ and the map
\[
[\tilde\Pi ]_\gamma\to\Hom (A, B_\gamma ), \quad \tilde\gamma\mapsto f_{\tilde\gamma}
\]
is \emph{injective}.
\end{rem}

\subsubsection{A slightly more economic description of Frobenius data}
It is straightforward to split the map \eqref{e:Frobenius data} into pieces corresponding to the connected components of $\Delta_{\hat\Pi}\,$. Let us explain the details for completeness.

Fix a connected component $C\subset\Delta_{\hat\Pi}\,$. Let $T_C$ (resp.~$Z_C$, $W_C$) denote the maximal torus (resp.~center, Weyl group) of the connected simply connected semisimple group over $\BQbar$ with Dynkin diagram $C$. Set 
\[
\Pi_C:=\{ \gamma\in\Pi\,|\, \gamma (C)=C\}, \quad \Pi_{C,\Fr}:=\Pi_C\cap\Pi_{\Fr}\, .
\]
For each $\tau\in\Aut C$ we set $\cT_{C,\tau}:=T_{C,\tau}/W_C^\tau\,$, where 
$T_{C,\tau}=\Coker (T_C\overset{\tau-1}\longrightarrow T_C)$ and
$W^\tau\subset W$ is the subgroup of $\tau$-invariants. Let $\Pi_{C,\Fr ,\tau}$ denote the preimage of $\tau$ in $\Pi_{C,\Fr}\,$.

The extension \eqref{e:2ext of pi} induces an extension
\[
0\to Z_C(\BQbar )\to\tilde\Pi_C\to\Pi_C\to 0.
\]
Let $[\tilde\Pi_C]$ denote the quotient of $\tilde\Pi_C$ by the conjugation action of $Z_C(\BQbar )$.

For each connected component $C\subset\Delta_{\hat\Pi}\,$ and each $\tau\in\Aut C$, the map \eqref{e:Frobenius data} induces a 
$Z_C$-anti-equivariant map 
\begin{equation}   \label{e:data componentwise}
[\tilde\Pi_C ]\times_{\Pi_C} \Pi_{C,\Fr ,\tau}\to \cT_{C,\tau} (\BQbar ).
\end{equation}
$\Pi\times\BN$ acts on the collection of the left-hand-sides of all maps \eqref{e:data componentwise} and on the collection of the right-hand-sides. 
Moreover, the collection of all maps \eqref{e:data componentwise} is $(\Pi\times\BN )$-equivariant.

Finally, the map $\eqref{e:Frobenius data}$ can be reconstructed from the maps \eqref{e:data componentwise} using $\BN$-equivariance and the following observation. Suppose that $\gamma\in\Pi$, $\gamma^m\in\Pi_C\,$, $\gamma^i\not\in\Pi_C$ for $0<i<m$; set 
$A:=\prod\limits_{i\in\BZ/m\BZ}T_{\gamma^i(C)}\,$. Then the homomorphism $A\to T_C$ whose restriction to $T_{\gamma^i (C)}$ is given by $\gamma^{-i}$ induces an \emph{isomorphism}
\[
\Coker (A\overset{\gamma-1}\longrightarrow A)\iso T_{C,\gamma^m}\,.
\]

\subsubsection{The action of $\Gal (\BQbar /\BQ)$} \label{sss:Gal(Q)-action}
The group $\Gal (\BQbar /\BQ)$ (viewed as an abstract group\footnote{The author prefers not to discuss the notion of action of a pro-finite group on a category.}) acts on the category $\Pross (\BQbar)$, and the object $\hat\Pi\in \Pross (\BQbar)$ is $\Gal (\BQbar /\BQ)$-equivariant with respect to this action.  So one has an action of $\Gal (\BQbar /\BQ)$ on $\Delta_{\hat\Pi}$ commuting with the $\Pi$-action. This action is \emph{continuous} by Theorem~3.1 of Deligne's article \cite{De2} combined with \cite[Proposition~VII.7(i)]{La}. 

Since $\hat\Pi\in \Pross (\BQbar)$ is $\Gal (\BQbar /\BQ)$-equivariant, the group $\Gal (\BQbar /\BQ)$ acts\footnote{In terms of the isomorphism $Z(\BQbar)\iso\Hom (P/Q, \BQbar^\times )$ from \eqref{e:groups via Dynkin}, this action corresponds to the action of $\Gal (\BQbar /\BQ)$ on both $\BQbar^\times$ and $P/Q$ (the latter via the action on $\Delta_{\hat\Pi}$).} on $Z(\BQbar )$, so it acts on the groupoid $\Ex (\Pi ,Z(\BQbar ))$.
As an object of the ``coarse" groupoid $\Ex (\Pi ,Z(\BQbar ))$, the extension $\tilde\Pi$ is canonically defined, so this object is equivariant with respect to $\Gal (\BQbar /\BQ)$. Therefore $\Gal (\BQbar /\BQ)$ acts on the set $[\tilde\Pi ]$. In fact, $[\tilde\Pi ]$ is a topological space, and the action of  $\Gal (\BQbar /\BQ)$ on $[\tilde\Pi ]$ is continuous.

Recall that $T=\Hom (P,\BG_m)$, where $P$ is the weight group of $\Delta_{\hat\Pi}\,$. So 
the action of $\Gal (\BQbar /\BQ)$ on $\Delta_{\hat\Pi}$ induces its action on $P$. We will always equip
the group $T(\BQbar )=\Hom (P,\BQbar^\times )$ with the action of $\Gal (\BQbar /\BQ)$ induced by its action on both $P$ and $\BQbar^\times$.
(This action corresponds to the $\BQ$-structure on the group scheme $T$ defined by the action of $\Gal (\BQbar /\BQ)$ on $P=\Hom (T,\BG_m)$.)

The action of $\Gal (\BQbar /\BQ)$ on $T(\BQbar )$ induces its action on $\cT (\BQbar )$ 
(which is a quotient of $T(\BQbar )\times\Pi$).

The map \eqref{e:Frobenius data} is $\Gal (\BQbar /\BQ)$-equivariant; in particular, each point of the image of the map $\Pi_{\Fr}\to( \cT /Z) (\BQbar )$ induced by \eqref{e:Frobenius data} is invariant under $\Gal (\BQbar /\BQ)$.

\subsubsection{Hope for a cleaner picture}  \label{sss:cleaner conjectural}
The philosophy of motives suggests a conjectural picture (see Appendix~\ref{s:Tannakian}), which is much cleaner than that of \S\ref{sss:Gal(Q)-action}. In particular, the $\Gal (\BQbar /\BQ)$-equivariant object 
$\tilde\Pi\in\Ex (\Pi ,Z(\BQbar ))$ discussed in \S\ref{sss:Gal(Q)-action} should conjecturally come from a canonical object of the groupoid $\Ex (\Pi\times\Gal (\BQbar /\BQ) ,Z(\BQbar ))$ (and moreover, from an object of a certain 2-groupoid described in \S\ref{sss:Tann pol SCPS Pi}).

\subsection{A result of L.~Lafforgue}   \label{ss:Laurent Lafforgue}
Let $T(\BQbar )_0\subset T(\BQbar )$ denote the subgroup of all elements $t\in T(\BQbar )$ such that for any $\chi\in\Hom (T,\BG_m)$ the number $\chi (t)\in\BQbar^\times$ is a unit outside of $p$ and all Archimedean absolute values of $\chi (t)$ equal 1. For $\gamma\in\Pi$,  the image of $T(\BQbar )_0$ in $\cT_\gamma (\BQbar )$ will be denoted by $\cT_\gamma (\BQbar )_0\,$. Let $\cT (\BQbar )_0\subset\cT (\BQbar )$ denote the union of $\cT_\gamma (\BQbar )_0$ for all $\gamma\in\Pi_{\Fr}\,$.

\begin{prop}   \label{p:Laurent Lafforgue}
(i) The image of the map \eqref{e:Frobenius data} is contained in $\cT (\BQbar )_0\, $.

(ii) The action of $\Gal (\BQbar/\BQ )$ on $\Delta_{\hat\Pi}$ factors through $\Gal (\sfC/\BQ )$, where $\sfC\subset\BQbar$ is the maximal CM-subfield. Moreover, the complex conjugation $\sigma\in\Gal (\sfC/\BQ )$ acts as the canonical involution of $\Delta_{\hat\Pi}$ (i.e., $\sigma$ takes any simple root $\alpha$ to $-w_0(\alpha)$); in particular, $\sigma$ preserves each connected component of $\Delta_{\hat\Pi}\,$.
\end{prop}

\begin{proof}
Statement (i) is just a paraphrase of \cite[Theorem VII.6 (ii)-(iii)]{La}.

Let us prove (ii). Let $U\subset\Pi$ be any open subgroup, $\rho$ a finite-dimensional representation of $\hat\Pi\times_{\Pi}U$, and $\gamma\in \im (U_{\Fr}\to [\hat\Pi](\BQbar )\times_{\Pi}U)$. By \cite[Theorem VII.6 (ii)]{La}, $\Tr\rho (\gamma )\in\sfC$ and 
$\sigma (\Tr\rho (\gamma ))=\Tr\rho^* (\gamma )$, where $\rho^*$ is the dual representation. Statement (ii) follows.
\end{proof}

\subsection{$p$-adic behavior of Frobenius data}   \label{ss:p-adic behavior}

\subsubsection{The Newton map}
Recall that we fixed a universal cover $\tilde X$ of $X$ and $\Pi:=\Aut (\tilde X/X)$. Recall that $|\tilde X|$ denotes the set of closed points of 
$\tilde X$. Let $\Val_p(\BQbar)$ denote the set of valuations $v:\BQbar^\times\to\BQ$ such that $v(p)=1$. Using the map \eqref{e:Frobenius data}, we will define the \emph{Newton map\footnote{It is similar to the ``Newton polygon" in the sense of \cite{K}.}}
\begin{equation}   \label{e:Newton map}
|\tilde X|\times \Val_p(\BQbar)\to\Hom (P,\BQ )/W,
\end{equation}
where $P$ is the weight group of $\Delta_{\hat\Pi}\,$. 

First, note that $P$ is the direct sum of the weight groups $P_C$ corresponding to all connected components $C\subset\Delta_{\hat\Pi}\,$. Moreover, $W$ is the product of the groups $W_C$ acting  on $P_C\,$. So defining the map \eqref{e:Newton map} is equivalent to defining a map
\begin{equation}   \label{e:CNewton map}
|\tilde X|\times \Val_p(\BQbar)\to\Hom (P_C,\BQ )/W_C\, ,
\end{equation}
for each connected component $C\subset\Delta_{\hat\Pi}\,$. Let $\tilde x\in\tilde X$ and $F_{\tilde x}\in\Pi$ the corresponding geometric Frobenius. Choose $n\in\BN$ so that $F_{\tilde x}^n$ acts on $C$ as the identity. Then the map \eqref{e:data componentwise} associates to $F_{\tilde x}^n$ a well-defined element $u_n\in (T_C (\BQbar )\otimes\BQ )/W_C\,$. One has $u_{mn}=(u_n)^m$.
$T_C (\BQbar) =\Hom (P_C, \BQbar^\times)$, so an element
$v\in\Val_p(\BQbar)$ induces a homomorphism $\Hom (P_C, \BQbar^\times)\to \Hom (P_C, \BQ)$ and therefore a map
 $\Hom (P_C, \BQbar^\times)/W_C\to \Hom (P_C, \BQ)/W_C\,$. Take the image of $u_n$ under this map and divide it by $n\cdot\deg\tilde x$, where $\deg\tilde x$ is the degree of the residue field of $\tilde x$ over $\BF_p\,$. The result does not depend on $n$. Thus we have defined the map \eqref{e:CNewton map} and therefore the map \eqref{e:Newton map}. It is clear that the map \eqref{e:Newton map} is equivariant with respect to $\Pi\times\Gal (\BQbar /\BQ)$.
  
 \subsubsection{Remark}
 As explained by T.~Koshikawa \cite{Kos}, for any connected component $C\subset\Delta_{\hat\Pi}$ the map \eqref{e:CNewton map} is nonzero (i.e., its image contains a nonzero element of $\Hom (P_C,\BQ )/W_C\,$). Indeed, Proposition~\ref{p:Laurent Lafforgue}(i) implies that if the map \eqref{e:CNewton map} were zero then the map \eqref{e:data componentwise} corresponding to any $\tau\in\Aut C$ would have finite image; but this is impossible because the map \eqref{e:0Frobenius data} has Zariski-dense image by  \v{C}ebotarev density.

\subsubsection{Some results of \cite{VLa,DK}}
As usual, we identify the set $\Hom (P_C,\BQ )/W_C$ with the dominant cone $\Hom^+ (P_C,\BQ )\subset\Hom (P_C,\BQ )$, and we equip it with the following partial order: given dominant rational coweights $\check\mu_1$ and $\check\mu_2$ we say that $\check\mu_1\le\check\mu_2$ if $\check\mu_2-\check\mu_1$ is a linear combination of simple coroots with non-negative rational coefficients.

Fix a connected component $C\subset\Delta_{\hat\Pi}$ and a valuation $v\in\Val_p(\BQbar)$. Then the map \eqref{e:CNewton map} corresponding to $C$ yields a map
\begin{equation} \label{e:2vCNewton map}
 |\tilde X|\to\Hom (P_C\, ,\BQ )/W_C=\Hom^+ (P_C\, ,\BQ ), \quad \tilde x\mapsto \check\mu_{\tilde x}
 \end{equation}
 (this map depends on the choice of $v\in\Val_p(\BQbar)$). 
 
 \begin{prop}   \label{p:DK}
 (i) There exists $\check\mu_{\max}\in\Hom^+ (P_C\, ,\BQ )$ such that
 \begin{equation}   \label{e:theinequality}
 \check\mu_{\tilde x}\le\check\mu_{\max} \quad \mbox{ for all } \tilde x\in |\tX |
 \end{equation}
 and  $\check\mu_{\tilde x}=\check\mu_{\max}$ for some $\tilde x\in |\tX |$.
 
 (ii) Let $U$ be the set of all $\tilde x\in |\tX |$ such that $\check\mu_{\tilde x}=\check\mu_{\max}$. Then
 $U\times_{|X|}|C|$ is open in $|\tX |\times_{|X|}|C|$  for every curve $C$ in $X$.
 
 (iii) $\check\rho_C-\check\mu_{\max}\in\Hom^+ (P_C\, ,\BQ )$, where $\check\rho_C\in\Hom^+ (P_C\, ,\BQ )$ is the sum of the fundamental coweights. 
 \end{prop}
 
 The proposition is proved in \cite[\S 8.5-8.6]{DK} using $F$-isocrystals and the main theorem of T.~Abe's work \cite{A} (existence of ``crystalline companions" in the case $\dim X=1$).
 
 \begin{rem}
According to \cite[Thm.~1.4]{Ke7}, the set $U$ from statement (ii) is, in fact, open. 
\end{rem}

 Since the dominant cone $\Hom^+ (P_C,\BQ )\subset\Hom (P_C,\BQ )$ is contained in the positive cone,
 Proposition~\ref{p:DK}(iii) and the inequality \eqref{e:theinequality} imply the following
 
 \begin{cor}
 $\check\mu_{\tilde x}\le \check\rho_C$ for all $\tilde x\in |\tX |$.  \qed
 \end{cor}
 
 If $C$ has type $A_n$ the corollary was proved by V.~Lafforgue without using $F$-isocrystals (see \cite[Prop.~2.1]{VLa}). His proof is elementary modulo the Langlands conjecture for $GL(n)$ over global function fields (of dimension 1) in the  direction ``from Galois to automorphic". This conjecture was proved in \cite{La} using Piatetski-Shapiro's ``converse theorem".

\section{Tannakian categories and the motivic hope}   \label{s:Tannakian}
In \S\ref{ss:Tannakian}-\ref{ss:polarizations} we recall basic facts about Tannakian categories and symmetric polarizations. The main goal here is to formulate Corollaries~\ref{c:polarized extensions} and \ref{c:polarized-Q}.

Then we formulate the ``motivic hope" briefly mentioned in Remarks~\ref{r:mot-hope1}-\ref{r:mot-hope2}.
In \S\ref{ss:mot Tann}-\ref{ss:passing to tilde X} we formulate it in the Tannakian language. Then we use 
Corollary~\ref{c:polarized-Q} to reformulate the conjectural picture in a more down-to-earth language and to compare it with the unconditional results, see \S\ref{sss:translating to Dynkins}-\ref{ss:dreams vs reality}. In \S\ref{ss:reciprocity} we show that the ``motivic hope" would imply certain ``reciprocity laws" involving a sum over all $\ell$-adic cohomology theories (including the crystalline theory for $\ell=p$), see Conjectures \ref{conj:reciprocity} and \ref{conj:2reciprocity}.

\subsection{Recollections on Tannakian categories} \label{ss:Tannakian}
The main references on Tannakian categories are \cite{Sa,DM, De90, De3}. Let us note that at least in characteristic $0$ instead of the language of fpqc-gerbes used in  \cite{Sa,DM, De90} one can use Galois gerbes (see \cite[\S 2]{LR} and \cite[Appendix B]{Ko}).

\subsubsection{Goal of this subsection}   \label{sss:Tannakian goal}
Let $E$ be a field and $\bar E$ an algebraic closure of $E$. For simplicity, we assume that $E$ has characteristic~0.

Let $\Tann_E$ denote the 2-groupoid of Tannakian categories over $E$.
We will use ``SCPS" as a shorthand for ``connected, simply connected, pro-semisimple".
Let $\cT\in\Tann_E\,$; we say that $\cT$ is SCPS if $\cT\otimes_{E}\bar E$ is $\otimes$-equivalent to the category of finite-dimensional representations of an SCPS group scheme over $\bar E$. Let $\Tann_E^{\rm SCPS}\subset\Tann_E$ denote the full 2-subgroupoid of SCPS Tannakian categories. Our goal is to recall an explicit description of $\Tann_E^{\rm SCPS}$ (see Proposition~\ref{p:Tann-SCPS} below).

\subsubsection{Pinned SPSC groups}
The words ``Dynkin diagram" are understood in the sense of \S\ref{sss:Dynkin diagrams}.
Let $\Dyn_E$ denote the groupoid of Dynkin diagrams equipped 
with a continuous action of $\Gal (\bar E /E )$. 

By a \emph{pinning} of a pro-semisimple group scheme $G$ over $E$ we mean a pinning of $G\otimes_E\bar E$ which is invariant under $\Gal (\bar E/E)$; it exists if and only if $G$ is quasi-split.
The groupoid of pinned SPSC group schemes over $E$ canonically identifies with $\Dyn_E\,$. The pinned SPSC group scheme over $E$ corresponding to $\Delta\in\Dyn_E$ will be denoted by $G_\Delta\,$. The center of $G_\Delta$ will be denoted by $Z_\Delta\,$.

\subsubsection{Constructing SPSC Tannakian categories}   \label{sss:Constructing SPSC}
For each $\Delta\in\Dyn_E\,$, finite-dimensional representations of $G_\Delta$ form a Tannakian category over $E$, denoted by $\Rep (G_\Delta )$.

Now suppose we are given an extension of pro-finite groups
\begin{equation}   \label{e:obj of Extrue}
0\to Z_\Delta (\bar E )\to H\to\Gal(\bar E/E )\to 0,
\end{equation}
where $\Gal(\bar E/E )$ acts on $Z_\Delta (\bar E )$ in the usual way. Then one defines a ``twisted version" of $\Rep (G_\Delta )$, denoted by $\Rep_H (G_\Delta )$. Namely, an object of $\Rep_H (G_\Delta )$ is a finite-dimensional $\bar E$-vector space $V$ equipped with an action of $G_\Delta\otimes_E\bar E$ and a continuous $E$-linear action of the pro-finite group $H$ so that

(i) the action of $Z_\Delta (\bar E )\subset H$ is the same as the action of $Z_\Delta (\bar E )\subset G_\Delta(\bar E )$;

(ii) for any $\tau\in\Gal(\bar E/E )$ and any $h\in H$ such that $h\mapsto\tau$, one has
\[
h(\lambda\cdot v)=\tau (\lambda )\cdot (hv), \quad h(g v)=\tau (g)\ (hv)
\]
for all $v\in V$, $\lambda\in \bar E$, $g\in G_\Delta(\bar E )$. 

Equipped with the obvious tensor product, $\Rep_H (G_\Delta )$ is a Tannakian category over $E$. If the extension \eqref{e:obj of Extrue} is trivial and trivialized then $\Rep_H (G_\Delta )$ identifies with 
$\Rep (G_\Delta )$ by Galois descent.

Let $\Ex_{\rm true} (\Gal (\bar E /E) ,Z_\Delta (\bar E ))$ denote the following 2-groupoid: its objects are extensions \eqref{e:obj of Extrue},
and the groupoid of isomorphisms between two such objects $H_1$ and $H_2$ is defined to be the quotient groupoid of the set of isomorphisms of extensions $H_1\iso H_2$ by the action  of the group $Z_\Delta (\bar E )$ (it acts by composing an isomorphism 
$H_1\iso H_2$ with conjugation by $z\in Z_\Delta (\bar E )$).
It is easy to check\footnote{One has to check that if $z\in Z_\Delta (\bar E)$ and $\varphi_z: H\iso H$ is the automorphism of conjugation by $z$ 
then the auto-equivalence of $\Rep_H(G_\Delta )$ corresponding to $\varphi_z$ is canonically isomorphic to the identity functor.} that the assignment $H\mapsto\Rep_H(G_\Delta )$ defines a functor $\Ex_{\rm true}(\Gal (\bar E /E) ,Z_\Delta(\bar E ))\to\Tann_{E}\,$ for each
$\Delta\in\Dyn_E\,$. Combining these functors for all $\Delta\in\Dyn_E\,$, one gets a functor
\[
\E_E\to\Tann_E\, ,
\]
where $\E_E$ is the groupoid of pairs consisting of $\Delta\in\Dyn_E$ and $H\in\Ex_{\rm true} (\Gal (\bar E /E) ,Z_\Delta (\bar E ))$.

\begin{prop}  \label{p:Tann-SCPS}
This functor induces an equivalence $\E_E\iso\Tann_E^{\rm SCPS}$. 
\end{prop}

\subsubsection{Why Proposition~\ref{p:Tann-SCPS} is well known}  \label{sss:well known}
We will use the terminology from the Appendix of \cite{De90} (in particular, the name ``affine gerbe" defined in the first paragraph on p.225 of \cite{De90}). We will say  simply ``gerbe over $E$"  instead of ``fpqc-gerbe on the category of $E$-schemes"; we also use a similar convention for bands.

According to one of the main results\footnote{The result follows from the statements of \cite[\S1.10-1.13]{De90} 
using the dictionary of \cite[\S3.1-3.6]{De90}; the dictionary relates affine gerbes (``gerbe \`a lien affine")  and transitve groupoids.} of \cite{De90},
$\Tann_{E}$ is canonically equivalent to the 2-groupoid of affine gerbes: namely, to an affine gerbe one associates the tensor category of its finite-dimensional representations, to a Tannakian category over $E$ one associates the gerbe of its fiber functors.

Thus describing $\Tann_E^{\rm SCPS}$ is equivalent to describing the 2-groupoid $\sG$ of affine gerbes over $E$ whose band is SCPS
(by this we mean that the band is locally defined by an SCPS group scheme). Let $\Gamma$ denote the groupoid of such bands. Then $\Gamma$ canonically identifies with $\Dyn_E\,$.
If $\Delta\in\Dyn_E$ let $\sG_\Delta$ denote the fiber over $\Delta$ of the functor 
$\sG\to\Gamma=\Dyn_E\,$. In $\sG_\Delta$ there is a distinguished object: namely, the classifying gerbe of $G_\Delta\,$. Using this distinguished object, one identifies $\sG_\Delta$ with the 2-groupoid of $Z_\Delta$-gerbes over $E$; the latter identifies with
$\Ex_{\rm true} (\Gal (\bar E /E) ,Z_\Delta(\bar E ))$. Finally, if $J\in\sG_\Delta$ corresponds to 
$H\in\Ex_{\rm true}(\Gal (\bar E /E) ,Z_\Delta(\bar E ))$ then the category of finite-dimensional representations of the gerbe $J$ identifies with 
the category $\Rep_H(G_\Delta)$ defined above.

\subsection{Recollections on symmetrically polarized Tannakian categories}   \label{ss:polarizations}
\subsubsection{Symmetric polarizations}   \label{sss:sym polarizations}   
Recall that a \emph{symmetric polarization} on a Tannakian category $\cT$ over $\BR$ is given by specifying for each $V\in\cT$ a class of 
non-degenerate symmetric bilinear forms on $V$ which are declared to be ``positive"; this class should satisfy certain conditions formulated in 
\cite[\S V.2.4.1]{Sa} or \cite[p.169,183]{DM}. 

Symmetric polarizations on a given Tannakian category $\cT$ over $\BR$ form a set, which is either empty or a torsor over the group of
$(\BZ/2\BZ)$-gradings of $\cT$ (see  \cite[\S V.3.2.2.1]{Sa} or \cite[Cor.~5.15]{DM}); this group can also be described as $\Ker (Z(\BR )\overset{2}\longrightarrow Z(\BR ))$, where $Z$ is the center of the band of $\cT$.

If $\cT$ is a Tannakian category over $\BQ$ there is a similar notion of symmetric polarization on $\cT$ (see \cite[\S V.2.4.1]{Sa}); on the other hand, by \cite[\S V.2.4.2]{Sa} this is the same\footnote{This is true even if not all $(\BZ/2\BZ)$-gradings of $\cT\otimes\BR$ come from $(\BZ/2\BZ)$-gradings of $\cT$.} as a symmetric polarization on $\cT\otimes\BR$. 

A \emph{symmetrically polarized Tannakian category} is a Tannakian category equipped with a symmetric polarization.

\subsubsection{Relation to compact groups}   \label{sss:compact groups}

Let $\Tann_{\BR,\pol}$ denote the 2-groupoid of symmetrically polarized
Tannakian categories over $\BR$. Let us recall its description in terms of compact groups.\footnote{As usual, we identify the category of compact topological groups with a a full subcategory of the category of group schemes over $\BR$ (to a compact group $K$ one associates the spectrum of the algebra of spherical functions $K\to\BR$).}
For any compact group $K$, finite-dimensional representations of $K$ over $\BR$ from a Tannakian category over $\BR$, denoted by $\Rep_\BR (K)$. The usual notion of positive-definite bilinear form on a real vector space defines a canonical symmetric polarization on $\Rep_\BR (K)$. By \cite[Thm 4.27]{DM}, this identifies $\Tann_{\BR ,\pol}$ with the 2-groupoid whose objects are compact groups, whose 1-morphisms are isomorphisms of compact groups, and whose 2-morphisms are defined as follows: a 2-isomorphism between 1-isomorphisms $f_1 ,f_2:K\iso K'$ is an element $g\in K$ such that $f_2^{-1}(f_1(x))=gxg^{-1}$ for all $x\in K$. Equivalently, the 1-groupoid of isomorphisms between $K_1$ and $K_2$ is the groupoid of $(K_1,K_2)$-bitorsors $Y$ such that $Y(\BR)$ is not empty.\footnote{$(K_1,K_2)$-bitorsors $Y$ with $Y(\BR)=\emptyset$ correspond to tensor equivalences $\Rep_\BR (K_1)\iso \Rep_\BR (K_2)$ not compatible with the polarizations. Note that if $K$ is a compact group with center $Z$ then any nontrivial $Z$-torsor defines a $(K,K)$-bitorsor $Y$ with $Y(\BR)=\emptyset$ because the map $K(\BR )\to (K/Z) (\BR )$ is surjective. Also note that isomorphism classes of $Z$-torsors are parametrized by 
$H^1(\BR ,Z)=\Ker (Z\overset{2}\longrightarrow Z)$.}

\subsubsection{The SCPS case}   \label{sss:polarized SCPS}
Let $\Tann_{\BR ,\pol}^{\rm SCPS}\subset\Tann_{\BR,\pol}$ denote the full subgroupoid corresponding to those compact groups that are connected and simply connected; equivalently, $\Tann_{\BR ,\pol}^{\rm SCPS}$ is the pre-image of $\Tann_\BR^{\rm SCPS}$ with respect to the forgetful functor
$\Tann_{\BR ,\pol}\to\Tann_\BR\,$. Let us describe $\Tann_{\BR ,\pol}^{\rm SCPS}$ in terms of Dynkin diagrams.

We say that $\Delta\in\Dyn_\BR$ is \emph{polarizable} if the nontrivial element $\sigma\in\Gal (\BC/\BR )$ acts on the set of vertices of $\Delta$ as the canonical involution (i.e., $\sigma$ takes any simple root $\alpha$ to $-w_0(\alpha)$). Let $\Dyn_{\BR,\pol}\subset\Delta_\BR$ denote the full subgroupoid of polarizable objects of $\Delta_\BR\,$. Of course, the forgetful functor $\Dyn_{\BR,\pol}\to\Dyn_\BC$ is an equivalence.

To every $\Delta\in\Dyn_{\BR,\pol}$ one associates (see \cite[\S 8, Thm. 16]{St}) a canonical connected simply connected compact group 
$K_\Delta$ so that $K_\Delta\otimes_\BR\BC=G_\Delta\otimes_\BR\BC$ and all homomorphisms $f_\alpha:SL(2,\BC)\to K_\Delta (\BC )$, 
$\alpha\in\Delta$,  corresponding to the pinning of the group $K_\Delta\otimes_\BR\BC=G_\Delta\otimes_\BR\BC$  map the subgroup $SU(2)\subset SL(2,\BC)$ to $K_\Delta (\BR )$.  This identifies $\Tann_{\BR ,\pol}^{\rm SCPS}$ with the product of $\Dyn_{\BR,\pol}$ and the 2-groupoid of $Z_\Delta (\BR )$-gerbes over a point.

Our next goal is to formulate Proposition~\ref{p:polarizations} describing $\Rep_\BR (K_\Delta )$ in terms of Proposition~\ref{p:Tann-SCPS}. To do this, we need a certain canonical element of $Z_\Delta (\BR )$.

\subsubsection{A canonical element of the center}  \label{sss:checkrho}
Let $E$ be a field and $\Delta\in\Dyn_E\,$. Let $\check\rho=\check\rho_\Delta$ denote the sum of the fundamental coweights of $\Delta$. Then $2\check\rho$ is an element of the coroot lattice invariant under $\Gal (\bar E/E)$, so it defines a homomorphism $\BG_m\to G_\Delta\,$, denoted by $t\mapsto t^{2\check\rho}$. It is well known that $(-1)^{2\check\rho}\in Z_\Delta (E )$. Note that $((-1)^{2\check\rho})^2=1^{2\check\rho}=1$.

\subsubsection{A canonical object of $\Ex_{\rm true}(\Gal (\BC /\BR) ,Z_\Delta(\BC ))$}   \label{sss:quaternionic}
The unique extension 
\begin{equation}  \label{e:Z/4Z}
0\to\BZ/2\BZ\to\BZ/4\BZ\to\BZ/2\BZ\to 0
\end{equation}
is an object of $\Ex_{\rm true}(\BZ/2\BZ ,\BZ/2\BZ )$. For any $\Delta\in\Dyn_{\BR,\pol}\,$, let $H_\Delta\in\Ex_{\rm true}(\Gal (\BC /\BR) ,Z_\Delta(\BC ))$ denote its image under the functor 
\[
\Ex_{\rm true}(\BZ/2\BZ ,\BZ/2\BZ )\to \Ex_{\rm true}(\Gal (\BC /\BR) ,Z_\Delta(\BC ))
\]
induced by the homomorphism $\BZ/2\BZ\to Z_\Delta(\BR )$ that takes $1$ to $(-1)^{2\check\rho}\in Z_\Delta(\BR )$.

\begin{prop}    \label{p:polarizations}
For any $\Delta\in\Dyn_{\BR,\pol}\,$, the Tannakian category $\Rep_\BR (K_\Delta )$ canonically identifies with $\Rep_{H_\Delta} (G_\Delta )\,$.
\end{prop}

\begin{proof}
We need some notation. Set $\fg_\Delta:=\Lie (G_\Delta)$. The complex conjugation on $\fg_\Delta\otimes_{\BR}\BC$ and $G_\Delta (\BC )$ will be denoted by $x\mapsto\bar x$. Let $\sigma :\Delta\to\Delta$ be the canonical involution.

Since $G_\Delta$ is pinned the Lie algebra 
$\fg_\Delta\otimes_{\BR}\BC$ is equipped with generators $e_i\,, f_i\,, h_i\,$, $i\in\Delta$, such that 
$$\overline{e_i}=e_{\sigma (i)}\, ,\overline{f_i}=f_{\sigma (i)}\, ,\overline{h_i}=h_{\sigma (i)}\, .$$
These generators satisfy the usual relations; in particular, $[e_i\, ,f_i]=h_i\,$.

Let us recall the description of the subgroup $K_\Delta\subset G_\Delta (\BC )$. Define $\tau\in\Aut (\fg_\Delta\otimes_{\BR}\BC)$ by 
$$\tau (e_i)=-f_{\sigma (i)}\, ,\quad \tau (f_i)=-e_{\sigma (i)}\, , \quad \tau (h_i)=-h_{\sigma (i)}\, .$$
The corresponding automorphism of $G_\Delta (\BC )$ will also be denoted by $\tau$. Note that $\tau (\bar g)=\overline{\tau (g)}$ for 
$g\in G_\Delta (\BC )$. Then
\[
K_\Delta=\{g\in G_\Delta (\BC )\,|\, \tau (\bar g)=g\}.
\]

The key point is that the automorphism $\tau$ is inner; moreover, it is given by conjugating with an element of $G_\Delta (\BR )$ which will be now defined \emph{explicitly}. For each $i\in I$ let $\tilde s_i\in G_\Delta (\BC)$ denote the image of the matrix 
$$\left(\begin{matrix} 0&1\\-1&0\end{matrix}\right)$$
under the homomorphism $SL(2,\BC)\to G(\BC )$ corresponding to the triple $(e_i,f_i,h_i)$. Let $w_0$ denote the longest element of the Weyl group of $G_\Delta$. 
Choose a reduced decomposition $w_0=s_{i_1}\cdot\ldots\cdot s_{i_N}$ as a product of simple reflections and define $\tilde w_0\in G_\Delta (\BC)$ by 
\begin{equation}  \label{e:def tilde w0}
\tilde w_0:=\tilde s_{i_1}\cdot \ldots\cdot \tilde s_{i_N}\, . 
\end{equation}
By \cite{T}, $\tilde w_0$ does not depend on the choice of a reduced decomposition. Using the decomposition $w_0=s_{\sigma (i_1)}\cdot \ldots.\cdot s_{\sigma (i_N)}$ instead of the original one, we see that $\tilde w_0\in G_\Delta (\BR)$. By \cite[Cor.~12.3]{AV} (which corresponds to Corollary~12.4 of the e-print version of  \cite{AV}), 
\begin{equation}  \label{e:tilde w0}
\tau (x)=\tilde w_0 x \tilde w_0^{-1} \quad\mbox{ for all } x\in G_\Delta (\BC ),
\end{equation}
It is also known  that
\begin{equation} \label{e:2tilde w0}
\tilde w_0^2=(-1)^{2\check\rho}
\end{equation}
(this is a particular case of  \cite[Prop.~12.1]{AV}, which corresponds to Proposition~12.2 of the e-print version of  \cite{AV}).

Now let us construct a tensor equivalence $\Rep_{H_\Delta} (G_\Delta )\iso\Rep_\BR (K_\Delta )$. By definition, an object of $\Rep_{H_\Delta} (G_\Delta )$ is a finite-dimensional $\BC$-vector space $V$ equipped with an action of $G_\Delta$ and an anti-linear map $J:V\to V$ such that $J^2=(-1)^{2\check\rho}$ and $J(gv)=\bar g J(v)$ for $g\in G_\Delta (\BC )$ and $v\in V$. Define 
$J':V\to V$ by $J'(v):=\tilde w_0 J(v)$. By \eqref{e:2tilde w0}, $J'$ is an anti-linear involution. By \eqref{e:tilde w0}, for $g\in G_\Delta (\BC )$ and $v\in V$ one has
$$J'(gv)=\tilde w_0 J(gv)=\tilde w_0 \bar g J(v)=\tau (\bar g)\tilde w_0 J(v)=\tau (\bar g)J'(v),$$
so $J'$ commutes with the action of $K_\Delta\,$. Thus the space $\{ v\in V\,|\,J'(v)=v\}$ is an object of $\Rep_\BR (K_\Delta )$.
\end{proof}

\subsubsection{Symmetric polarizations in terms of\, $\Ex_{\rm true}\,$}   \label{sss:polarized extensions}
Let $\Delta\in\Dyn_{\BR,\pol}\,$. Let $H\in\Ex_{\rm true}(\Gal (\BC /\BR) ,Z_\Delta(\BC ))$, so we have an exact sequence 
$0\to Z_\Delta (\BC )\to H\to\Gal(\BC /\BR )\to 0$. Note that since $\Delta$ is polarizable $Z_\Delta (\BC )=Z_\Delta(\BR )$, so $Z_\Delta (\BC )$ is central in the group $H$; since $\Gal (\BC /\BR)$ is cyclic this implies that the group $H$ is abelian. In particular, the conjugation action of $Z_\Delta (\BC )$ on $H$ is trivial, so talking about elements of $H$ makes sense even though $H$ is viewed as an object of $\Ex_{\rm true}(\Gal (\BC /\BR) ,Z_\Delta(\BC ))$.

By a \emph{polarization} of  $H$ we will mean an element $j\in H$ such that 

(i) the image of $j$ in $\Gal (\BC /\BR)$ equals the complex conjugation $\sigma$;

(ii) $j^2$ is equal to the element $(-1)^{2\check\rho}\in Z_\Delta (\BR )$ defined in \S\ref{sss:checkrho}.

It is clear that $H$ admits a polarization if and only if the class of $H$ in $H^2(\Gal (\BC /\BR), Z_\Delta (\BR ))=Z_\Delta (\BR )/2 Z_\Delta (\BR )$ is equal to the image of $(-1)^{2\check\rho}\in Z_\Delta (\BR )$.

\begin{cor}     \label{c:polarized extensions}
Let $\Delta\in\Dyn_{\BR}$ and $H\in\Ex_{\rm true}(\Gal (\BC /\BR) ,Z_\Delta(\BC ))$. If $\Delta\not\in\Dyn_{\BR,\pol}$ then $\Rep_H(G_\Delta)$ has no symmetric polarizations. 

If $\Delta\in\Dyn_{\BR,\pol}$ then one has a canonical bijection between symmetric polarizations on $\Rep_H(G_\Delta)$ and polarizations of $H$; namely, the symmetric polarization on  $\Rep_H(G_\Delta)$ corresponding to a polarization $j\in H$ is defined as follows: a nondegenerate symmetric bilinear form $B$ on $V\in\Rep_H(G_\Delta)$ is declared to be ``positive" if $B(v,\tilde w_0jv)>0$ for all $v\in V\setminus\{0\}$, where $\tilde w_0\in G_\Delta (\BR )$ is defined by \eqref{e:def tilde w0}.
\end{cor}

\begin{proof}
By \S\ref{sss:polarized SCPS}, the 1-categorical truncation of the 2-groupoid $\Tann_{\BR ,\pol}^{\rm SCPS}$ identifies with 
$\Dyn_{\BR,\pol}$ via the functor $\Delta'\mapsto\Rep_\BR (K_{\Delta'})$, $\Delta'\in\Dyn_{\BR,\pol}\,$. By Proposition~\ref{p:polarizations}, 
$\Rep_\BR (K_{\Delta'})=\Rep_{H_{\Delta'}} (G_{\Delta'} )$. So a symmetric polarization on $\Rep_H(G_\Delta)$ is the same as an isomorphism class of pairs consisting of $\Delta'\in\Dyn_{\BR,\pol}$ and a tensor equivalence $F:\Rep_{H_{\Delta'}}(G_{\Delta'})\iso \Rep_H(G_\Delta)$. 
If $\Delta\not\in\Dyn_{\BR,\pol}$ there are no such pairs. By Proposition~\ref{p:Tann-SCPS}, if  $\Delta\in\Dyn_{\BR,\pol}$ then an isomorphism class of pairs $(\Delta', F)$ as above is the same as a 2-isomorphism class of 1-isomorphisms $H_\Delta\iso H$. Finally, this is the same as an element $j\in H$ satisfying the above conditions (i)-(ii); namely, $j$ is the image of $1\in\BZ/4\BZ$, where $\BZ/4\BZ$ is the central term of the extension \eqref{e:Z/4Z}. 

The concrete description of the symmetric polarization on $\Rep_H(G_\Delta)$ corresponding to a polarization $j\in H$ is obtained from the construction of the equivalence $\Rep_{H_\Delta} (G_\Delta )\iso\Rep_\BR (K_\Delta )$ in the proof of Proposition~\ref{p:polarizations}.
\end{proof}

\subsubsection{Symmetrically polarized SCPS Tannakian categories over $\BQ$}
Let $\sfC\subset\BQbar$ denote the maximal CM-subfield.
Say that $\Delta\in\Dyn_\BQ$ is \emph{polarizable} if the action of $\Gal (\BQbar/\BQ )$ on the Dynkin diagram $\Delta$ factors through $\Gal (\sfC/\BQ )$ and complex conjugation (which is a well-defined element of $\Gal (\sfC/\BQ )$) acts on $\Delta$ as the canonical involution of $\Delta$. Let $\Dyn_{\BQ,\pol}\subset\Dyn_{\BQ}$ denote the full subgroupoid of polarizable diagrams; equivalently, $\Dyn_{\BQ,\pol}$ is the preimage of $\Dyn_{\BR,\pol}$ in $\Dyn_{\BQ}\,$.

\begin{cor}   \label{c:polarized-Q}
The 2-groupoid of symmetrically polarized SCPS Tannakian categories over $\BQ$ is canonically equivalent to the 2-groupoid of triples
$(\Delta ,H,j)$, where $$\Delta\in\Dyn_{\BQ,\pol}\,, \quad H\in\Ex_{\rm true}(\Gal (\BQbar /\BQ ) ,Z_\Delta(\BQbar )),$$ and $j$ is a polarization (in the sense of \S\ref{sss:polarized extensions}) of the image of $H$ in $\Ex_{\rm true}(\Gal (\BC /\BR ) ,Z_\Delta(\BC ))$.
\end{cor}

\begin{proof}
By \cite[\S V.2.4.2]{Sa} a symmetric polarization on a Tannakian category $\cT$ over $\BQ$ is the same as a symmetric polarization on 
$\cT\otimes\BR$. It remains to use Proposition~\ref{p:Tann-SCPS} and Corollary~\ref{c:polarized extensions}.
\end{proof}

\subsection{Motivic hope in the Tannakian language}   \label{ss:mot Tann}
\subsubsection{}  
For each prime number $\ell$ we have the semisimple Tannakian category $\cT_\ell (X)$ over $\BQ_\ell\,$: it was defined in \S\ref{sss:rep hat-Pi-lambda} if $\ell\ne p$  and in \S\ref{sss:cT_p} if $\ell = p$. Recall that if $\ell\ne p$ then $\cT_\ell (X)$ is the tensor category of semisimple lisse $\BQ_\ell$-sheaves $\E$ on $X$ such that the determinant of each irreducible component of $\E\otimes_{\BQ_\ell}\BQbar_\ell$ has finite order. To define $\cT_p (X)$, one replaces lisse $\BQ_\ell$-sheaves with overconvergent $F$-isocrystals.\footnote{A brief discussion of overconvergent $F$-isocrystals and some references can be found in \cite[\S 1.3]{Ke2} and especially in \cite{Ke3}. On the other hand, I treat $\cT_p (X)$ as a black box.} By Proposition~\ref{p:Weil-2} and  \S\ref{sss:cT_p}, for every prime $\ell $ (including $\ell=p$) and every connected etale covering $\pi :X'\to X$, the functor $\pi^*$ maps $\cT_\ell (X)$ to $\cT_\ell (X')$. 

\subsubsection{}   \label{sss:mot Tann}
The philosophy of motives suggests that \emph{there should be a canonical Tannakian category $\cT (X)$ over $\BQ$ equipped with tensor equivalences $\cT (X)\otimes\BQ_\ell\iso\cT_\ell (X)$} for all primes $\ell$. Moreover, $\cT (X)$ should be a Tannakian subcategory of the category of pure motives of weight $0$ over the field of rational functions $\BF_p(X)$, and the above-mentioned equivalences should be given by the usual realization functors. Moreover, the pullback functor \{motives over $\BF_p(X)\}\to$\{motives over $\BF_p(X')$\} should map $\cT (X)$ to $\cT (X')$.

\subsubsection{}   \label{sss:standard Hodge}
Assuming the standard conjectures, the category of pure motives of weight $0$ over a field is equipped with a canonical symmetric polarization\footnote{For the notion of symmetric polarization see \S\ref{sss:sym polarizations}-\ref{sss:compact groups} and references therein.} defined in 
 \cite[\S VI.4.4]{Sa}. So $\cT (X)$ should be equipped with a canonical symmetric polarization.

\subsection{Passing to $\tX$}      \label{ss:passing to tilde X}
\subsubsection{}
Let $U\subset\Pi$ be a normal open subgroup. Then $\Pi/U$ acts on $\cT_\ell (\tX/U)$, and $\cT_\ell (X)$ identifies with the category of 
$(\Pi/U)$-equivariant objects of $\cT_\ell (\tX/U)$, which is denoted by $\cT_\ell (\tX/U)^{\Pi/U}$. Assuming \S\ref{sss:mot Tann}, we see that the functor $\cT (X)\iso\cT (\tX/U)^{\Pi/U}$ should be an equivalence.

\subsubsection{}
Let $\cT_\ell (\tX)$ (resp.~$\cT (\tX)$) denote the direct limit of the categories $\cT_\ell (\tX/U)$ (resp.~$\cT (\tX/U)$). Then $\Pi$ acts on 
$\cT_\ell (\tX)$ and $\cT (\tX)$. We have $\cT_\ell (X)=\cT_\ell (\tX)^\Pi$, $\cT (X)=\cT (\tX)^\Pi$.

\subsubsection{}    \label{sss:Tann-SCPS}
If $\ell\ne p$ then $\cT_\ell (X)=\Rep (\hat\Pi_\ell )$, where $\hat\Pi_\ell$ is the $\ell$-adic pro-semisimple completion of $\Pi$. Accordingly, 
$\cT_\ell (\tX )=\Rep (\hat\Pi_\ell^\circ )$. By Proposition~\ref{p:simply connected}, the group $\hat\Pi_\ell^\circ$ is SCPS (shorthand for ``connected, simply connected, pro-semisimple"). So for $\ell\ne p$ the Tannakian category $\cT_\ell (\tX )$ is SCPS in the sense of \S\ref{sss:Tannakian goal}. Assuming \S\ref{sss:mot Tann}, we see that $\cT (\tX )$ and $\cT_p (\tX )$ have to be SCPS. Note that if $\dim X=1$ then $\cT_p (\tX )$ is SCPS unconditionally by Corollary~\ref{c:p-simply connected}.

Let $\Tann_E^{\rm SCPS}$ denote the 2-groupoid of SCPS Tannakian categories over a field $E$ of characteristic $0$. Let
$\Tann_{\BQ ,\pol}^{\rm SCPS}$ denote the 2-groupoid of symmetrically polarized SCPS Tannakian categories over $\BQ$. Our $\cT (\tX )$ is an object of $(\Tann_{\BQ ,\pol}^{\rm SCPS})^\Pi$, i.e., the 2-groupoid of $\Pi$-equivariant objects of $\Tann_{\BQ ,\pol}^{\rm SCPS}\,$. Similarly, 
$\cT_{\ell} (\tX )\in(\Tann_{\BQ_\ell}^{\rm SCPS})^\Pi\,$.

\subsection{Description of $(\Tann_E^{\rm SCPS})^\Pi$ and $(\Tann_{\BQ ,\pol}^{\rm SCPS})^\Pi$} \label{sss:translating to Dynkins}
Proposition~\ref{p:Tann-SCPS} and Corollary~\ref{c:polarized-Q} describe the 2-groupoids $\Tann_E^{\rm SCPS}$ and $\Tann_{\BQ ,\pol}^{\rm SCPS}$ in ``concrete'' terms. This immediately yields a ``concrete'' description of $(\Tann_E^{\rm SCPS})^\Pi$ and $(\Tann_{\BQ ,\pol}^{\rm SCPS})^\Pi$. We formulate it below.

\subsubsection{Notation}
Let $\Dyn_E^\Pi$ denote the groupoid of Dynkin diagrams equipped with an action of the group $\Pi\times\Gal(\bar E/E)$. If 
$\Delta\in\Dyn_E^\Pi$ let $Z_\Delta$ denote the center of the corresponding quasi-split SCPS group over~$E$ (so $Z_\Delta$ is a pro-finite group scheme over $E$ equipped with $\Pi$-action).

\subsubsection{Description of the 2-groupoid $(\Tann_E^{\rm SCPS})^\Pi$}   \label{sss:Tann SCPS Pi}
An object of $(\Tann_E^{\rm SCPS})^\Pi$ is the same as a pair $(\Delta,H)$, where $\Delta\in\Dyn_E^\Pi$ and 
$H\in\Ex_{\rm true}(\Pi\times\Gal (\bar E/E)  ,Z_\Delta(\bar E))$. (The notation $\Ex_{\rm true}$ was introduced in \S\ref{sss:Constructing SPSC}.)

\subsubsection{Description of the 2-groupoid $(\Tann_{\BQ ,\pol}^{\rm SCPS})^\Pi$}  
\label{sss:Tann pol SCPS Pi}
Fix an embedding $\BQbar\mono\BC$; it induces an embedding $\Gal (\BC/\BR )\mono\Gal (\BQbar /\BQ)$. Let 
$\sigma\in\Gal (\BC/\BR )\subset\Gal (\BQbar /\BQ)$ denote complex conjugation. Say that $\Delta\in\Dyn_{\BQ}^\Pi$ is \emph{polarizable} if 
$\sigma$ acts on $\Delta$ as the canonical involution of $\Delta$; this implies that the action of $\Gal (\BQbar /\BQ)$ on $\Delta$ factors through
$\Gal (\sfC/\BQ )$, where $\sfC\subset\BQbar$ is the maximal CM-subfield. Let $\Dyn_{\BQ,\pol}^\Pi$ denote the full subcategory of polarizable objects of $\Dyn_{\BQ}^\Pi\,$.

An object of $(\Tann_{\BQ ,\pol}^{\rm SCPS})^\Pi$ is the same as a triple $(\Delta ,H,j)$, where 
$$\Delta\in\Dyn_{\BQ,\pol}^\Pi\, ,\quad H\in\Ex_{\rm true}(\Pi\times\Gal (\BQbar /\BQ)  ,Z_\Delta(\BQbar )),$$ 
and $j$ is an element of $H$ such that

(i) the image of $j$ in $\Pi\times\Gal (\BQbar /\BQ)$ equals $(1_\Pi\, ,\sigma)$;

(ii) $j^2$ is equal to the element $(-1)^{2\check\rho}\in Z_\Delta (\BQ )$ defined in \S\ref{sss:checkrho};

(iii) $j$ centralizes $\Ker (H\epi\Gal (\BQbar /\BQ))$.

\subsubsection{Remark}
If $\Delta\in\Dyn_{\BQ,\pol}^\Pi$ then $\sigma$ acts on $Z_\Delta(\BC )$ trivially, so 
$H^2(\Gal (\BC /\BR), Z_\Delta(\BC ))=Z_\Delta(\BC )/2Z_\Delta(\BC )$. Moreover, the existence of $j\in H$ satisfying the above properties (i)-(ii) implies that the class of $H$ in $H^2(\Gal (\BC /\BR), Z_\Delta(\BC ))=Z_\Delta(\BC )/2Z_\Delta(\BC )$ is equal to the image of 
$(-1)^{2\check\rho}\in Z_\Delta (\BC )$.

\subsection{Comparing the conjectural picture with the unconditional one}  \label{ss:dreams vs reality}
\subsubsection{The triple $(\Delta ,H,j)$ and the pairs $(\Delta_\ell ,H_\ell )$}  \label{sss:conjectural triple}
Let $(\Delta ,H,j)$ be the triple corresponding by \S\ref{sss:Tann pol SCPS Pi} to $\cT (\tX )\in(\Tann_{\BQ ,\pol}^{\rm SCPS})^\Pi$. Let 
$(\Delta_\ell ,H_\ell )$ be the pair corresponding by \S\ref{sss:Tann SCPS Pi} to $\cT_\ell (\tX )\in(\Tann_{\BQ_\ell}^{\rm SCPS})^\Pi$. Then $\Delta_\ell\in\Dyn_{\BQ_\ell}^\Pi$ is the image of $\Delta\in\Dyn_{\BQ,\pol}^\Pi\,$, and $H_\ell\in\Ex_{\rm true}(\Pi\times\Gal (\BQbar_\ell /\BQ_\ell )  ,Z_\Delta( \BQbar_\ell ))$ is the image of
$H\in\Ex_{\rm true}(\Pi\times\Gal (\BQbar /\BQ)  ,Z_\Delta(\BQbar ))$.

\subsubsection{Unconditional coarsenings of $(\Delta ,H,j)$}   \label{sss:unconditional}
Let us emphasize that the triple $(\Delta ,H,j)$ is \emph{only conjectural} (assuming \S\ref{ss:mot Tann}). However, it has some
coarsenings which are defined unconditionally.

First of all, the pair 
$(\Delta_\ell ,H_\ell )$ is defined unconditionally for $\ell\ne p$. If $\dim X=1$ this is also true for $\ell =p$ because $\cT_p (\tX )$ is SCPS unconditionally by Corollary~\ref{c:p-simply connected}.

In \S\ref{ss:Dyn} and \S\ref{sss:Gal(Q)-action} we unconditionally defined $\Delta_{\hat\Pi}\in\Dyn_{\BQ}^\Pi\,$. We also proved that 
$\Delta_{\hat\Pi}\in\Dyn_{\BQ, \pol }^\Pi\,$, see Proposition~\ref{p:Laurent Lafforgue}(ii). 
It is clear that if the triple $(\Delta ,H,j)$ from \S\ref{sss:conjectural triple} is defined (i.e., if the assumptions of \S\ref{ss:mot Tann} hold) then $\Delta$ canonically identifies with $\Delta_{\hat\Pi}\,$. Unconditionally, the image of $\Delta_{\hat\Pi}$ in 
$\Dyn_{\BQ_{\ell}}^\Pi$ canonically identifies with $\Delta_\ell$ for each prime $\ell\ne p$; by Theorem~\ref{t:p-main}, this is true even for $\ell =p$ if $\dim X=1$.

In \S\ref{sss:tildePi} and \S\ref{sss:Gal(Q)-action} we unconditionally defined  a canonical $\Gal (\BQbar /\BQ)$-equivariant object $\tilde\Pi$ of the 1-ca\-te\-go\-rical truncation\footnote{In Appendix~\ref{s:what we know} this truncation was denoted by  $\Ex (\Pi ,Z(\BQbar ))$.} of the 2-groupoid 
$\Ex (\Pi ,Z_{\Delta_{\hat\Pi}}(\BQbar ))$. It is clear that if the assumptions of \S\ref{ss:mot Tann} hold then this object comes from the object $H\in\Ex_{\rm true}(\Pi ,Z_\Delta(\BQbar ))$ defined in \S\ref{sss:conjectural triple}  (without the assumptions of \S\ref{ss:mot Tann} it is not clear how to define $H$).

\subsection{On fiber functors}
For any prime $\ell$, \emph{any\,} SCPS Tannakian category over $\BQ_{\ell}$ has a fiber functor over~$\BQ_{\ell}\,$, which is unique up to (non-unique) isomorphism: indeed, this is just a reformulation of \cite[Satz 2]{Kn}. Of course, for $\ell\ne p$ the category $\cT_{\ell} (\tX )=
\Rep (\hat\Pi_\ell^\circ )$ is equipped with a \emph{canonical} fiber functor. Note that composing a fiber functor for $\cT_{\ell} (\tX )$ with the pullback $\cT_{\ell} (X )\to\cT_{\ell} (\tX )$ one gets a fiber functor for $\cT_{\ell} (X )$.

\subsection{Reciprocity law. I}  \label{ss:reciprocity}
Assume that $\dim X=1$. Then by \S\ref{sss:unconditional} we unconditionally have
$H_\ell\in\Ex_{\rm true}(\Pi\times\Gal (\BQbar_\ell /\BQ_\ell )  ,Z_{\Delta_{\hat\Pi}}( \BQbar_\ell ))$ for each prime $\ell$ (including $\ell=p$). Let $u_\ell\in H^2(\BQ_\ell ,Z_{\Delta_{\hat\Pi}})$ denote the class of $H_\ell\,$. Recall that 
\begin{equation}  \label{e:recall Z}
Z_{\Delta_{\hat\Pi}}=\Hom (P/Q,\BG_m), 
\end{equation}
where $P$ is the weight group of $\Delta_{\hat\Pi}$ and $Q\subset P$ is the subgroup generated by the roots. Local class field theory identifies $H^2(\BQ_\ell ,Z_{\Delta_{\hat\Pi}})$ with $\Hom ((P/Q)^{\Gal (\BQbar_\ell/\BQ_\ell)},\BQ/\BZ)$, so $u_\ell\in\Hom ((P/Q)^{\Gal (\BQbar_\ell/\BQ_\ell)},\BQ/\BZ)$. Let $\bar u_\ell\in\Hom ((P/Q)^{\Gal (\BQbar/\BQ )},\BQ/\BZ )$ denote the image of $u_\ell\,$. 

Let $u_\infty\in\Hom ((P/Q)^{\Gal (\BC/\BR)},\BQ/\BZ)$ and $\bar u_\infty\in\Hom ((P/Q)^{\Gal (\BQbar/\BQ )},\BQ/\BZ)$ denote the restrictions of
the homomorphism $P/Q\to\frac{1}{2}\BZ/\BZ\subset\BQ/\BZ$ corresponding to $\check\rho\in Q^*$.

We will show in \S\ref{ss:almost all Brauer} that for any $\omega\in (P/Q)^{\Gal (\BQbar/\BQ )}$ one has
\begin{equation}   \label{e:almost all Brauer}
(\bar u_\ell ,\omega)=0 \mbox{ for almost all } \ell .
\end{equation}
So the sum $\sum\limits_\ell\bar u_\ell$ makes sense.

\begin{conj}  \label{conj:reciprocity}
One has 
\begin{equation}  \label{e:reciprocity}
\sum\limits_\ell\bar u_\ell=\bar u_\infty\,.
\end{equation}
\end{conj}

This conjecture would follow from the ``motivic hope" of \S\ref{ss:mot Tann}. Indeed, assuming \S\ref{ss:mot Tann} we would have
$H\in\Ex_{\rm true}(\Pi\times\Gal (\BQbar /\BQ)  ,Z_{\Delta_{\hat\Pi}}( \BQbar ))$ and $j\in H$ as in \S\ref{sss:Tann pol SCPS Pi}. 
Let $u\in H^2(\BQ ,Z_{\Delta_{\hat\Pi}})$ denote the class of $H$, then the image of $u$ in $H^2(\BQ_\ell ,Z_{\Delta_{\hat\Pi}})$ equals $u_\ell\,$, and the properties of $j$ (see \S\ref{sss:Tann pol SCPS Pi}) imply that the image of $u$ in $H^2(\BR ,Z_{\Delta_{\hat\Pi}})$ equals $u_\infty\,$. So \eqref{e:reciprocity} follows by global class field theory and the equality $2u_\infty=0$.

\begin{lem}   \label{l:where both sides live}
One has a canonical isomorphism
\[
\Hom ((P/Q)^{\Gal (\BQbar/\BQ )},\BQ/\BZ )\simeq (Z_{\Delta_{\hat\Pi}}/Z_{\Delta_{\hat\Pi}}^2)_{\Gal (\BQbar/\BQ )}\, ,
\]
where $Z_{\Delta_{\hat\Pi}}^2$ is the image of the homomorphism $Z_{\Delta_{\hat\Pi}}\to Z_{\Delta_{\hat\Pi}}\,$, $z\mapsto z^2$.
\end{lem}

\begin{proof}
Combine \eqref{e:recall Z} with the following consequence of Proposition~\ref{p:Laurent Lafforgue}(ii): the complex conjugation automorphism of $\BQbar$ corresponding to any embedding $\BQbar\mono\BC$ acts on $P/Q$ as $-1$. 
\end{proof}

By Lemma~\ref{l:where both sides live}, both sides of \eqref{e:reciprocity} can be considered as elements of 
$(Z_{\Delta_{\hat\Pi}}/Z_{\Delta_{\hat\Pi}}^2)_{\Gal (\BQbar/\BQ )}\,$. The right-hand side is just the image of 
$(-1)^{2\check\rho}\in Z_{\Delta_{\hat\Pi}}(\BQ )$.

\subsection{Reciprocity law. II}  \label{ss:2reciprocity}
As before, we assume that $\dim X=1$. We will formulate Conjecture~\ref{conj:2reciprocity}, which is related to Conjecture~\ref{conj:reciprocity} (see Proposition~\ref{p:two reciprocities} below).

Let $\Irr (\hat\Pi )$ denote the set of isomorphism classes of irreducible representations of $\hat\Pi$. The group $\Gal (\BQbar/\BQ)$ acts on 
$\Irr (\hat\Pi )$.

Let $\rho\in\Irr (\hat\Pi )$. Then the stabilizer of  $\rho$ in $\Gal (\BQbar/\BQ)$ equals $\Gal (\BQbar/E)$, where $E\subset\BQbar$ is the subfield generated by the coefficients of the Frobenius characteristic polynomials of all closed points of $X$. One has $[E:\BQ]<\infty$. Since the Frobenius eigenvalues are Weil numbers, $E$ is either totally real or a CM field. For the same reason, $\rho^*=\sigma (\rho )$, where $\sigma\in\Aut E$ is complex conjugation and $\rho^*\in\Irr (\hat\Pi )$ is the dual of~$\rho$. So $E$ is totally real if and only if $\rho^*=\rho$, which means that $\rho$ is either orthogonal or symplectic.

We can think of $\rho$ as a compatible system of irreducible objects
 $\rho_\lambda\in\cT_\ell (X)\otimes_{\BQ_\ell}\bar E_\lambda$, where 
$\ell$ runs through the set of all primes (including $p$) and $\lambda$ runs through the set of all places of $E$ dividing $\ell$.  The isomorphism class of $\rho_\lambda$ is $\Gal (\bar E_\lambda/E_\lambda )$-invariant. Let $\beta_\lambda\in\Br (E_\lambda )=\BQ/\BZ$ be the Brauer obstruction corresponding to $\rho_\lambda\,$. We will show in \S\ref{ss:almost all Brauer} that  
\begin{equation} \label{e:2almost all Brauer}
\beta_\lambda=0 \mbox{ for almost all } \lambda .
\end{equation}
Since $\rho^*=\sigma (\rho)$ we have $\beta_\lambda+\beta_{\sigma (\lambda )}=0$, so
\[
\sum_\lambda \beta_\lambda\in\frac{1}{2}\BZ/\BZ \, .
\]

\begin{conj}  \label{conj:2reciprocity}
In this situation\footnote{An important feature of our situation is that each $\rho_\lambda$ has weight $0$ (because $\det\rho_\lambda$ has finite order).}\,
$\sum\limits_\lambda \beta_\lambda$ equals $\frac{1}{2}$ if and only if $\rho$ is symplectic and $[E:\BQ ]$ is odd.
\end{conj}

Let us show that Conjecture~\ref{conj:2reciprocity} would follow from the ``motivic hope" of \S\ref{ss:mot Tann}. Indeed, assuming \S\ref{ss:mot Tann}, $\rho$ corresponds to an irreducible object $M\in\cT (X)\otimes\BQbar$ whose isomorphism class is $\Gal (\bar E /E )$-invariant. Associated to $M$ is the Brauer obstruction $\beta\in\Br (E)$. For every non-Archimedean place $\lambda$ of~$E$, the image of 
$\beta$ in $\Br (E_\lambda )$ equals $\beta_\lambda\,$. So $\sum\limits_\lambda \beta_\lambda\ne 0$ only if $E$ is totally real. It remains to show that in this case for every embedding $E\mono\BR$ the image of $\beta$ in $\Br (\BR )$ equals $0$ if $M$ is orthogonal and 
$\frac{1}{2}$ if $M$ is symplectic. This follows from the fact that by \S\ref{sss:standard Hodge} and \S\ref{sss:compact groups} the Tannakian category $\cT (X)\otimes\BR$ is equivalent to the category of real representations of a compact group.

\begin{prop}  \label{p:two reciprocities}
Suppose that Conjecture~\ref{conj:2reciprocity} holds for all curves $X$. Then Conjecture~\ref{conj:reciprocity} holds for all curves $X$.
\end{prop}

\begin{proof}  
Both sides of \eqref{e:reciprocity} are homomorphisms $(P/Q)^{\Gal (\BQbar/\BQ )}\to \BQ/\BZ$. By Lemma~\ref{l:minuscule}, the map 
$P^{\Gal (\BQbar/\BQ )}\to (P/Q)^{\Gal (\BQbar/\BQ )}$ is surjective. So it suffices to show that
\[
\sum\limits_\ell (\bar u_\ell ,\omega)=(\bar u_\infty,\omega) \quad\mbox{ for all }\omega\in P^{\Gal (\BQbar/\BQ )}.
\]

Replacing $X$ by a finite etale covering of $X$, we can assume that there exists $\rho\in \Irr (\hat\Pi )^{\Gal (\BQbar/\BQ)}$ whose restriction to
$\hat\Pi^\circ$ is the irreducible representation with highest weight $\omega$. Then $(\bar u_\ell ,\omega)=\beta_\ell\,$, where $\beta_\ell$ is the Brauer obstruction corresponding to $\rho$ and $\ell$. 

On the other hand, $(\bar u_\infty,\omega)\in\frac{1}{2}\BZ/\BZ$ is the image of $(\check\rho ,\omega ) \in\frac{1}{2}\BZ$, where $\check\rho\in Q^*$ is the sum of the fundamental coweights. So $(\bar u_\infty,\omega)$ equals $0$ if $\rho$ is orthogonal and $\frac{1}{2}$ if $\rho$ is symplectic.
\end{proof}

\subsection{Proof of \eqref{e:almost all Brauer} and \eqref{e:2almost all Brauer}} \label{ss:almost all Brauer}
We will deduce \eqref{e:almost all Brauer} and \eqref{e:2almost all Brauer} from Proposition~\ref{p:irred reductions} below.

Let $E\subset\BQbar$, $[E:\BQ]<\infty$. For each non-Archimedean place $\lambda$ of $E$ not dividing $p$, we fix an algebraic closure $\bar E_\lambda\supset E_\lambda\,$. Let $\bar O_\lambda\subset \bar E_\lambda$ be the ring of integers and $\bar m_\lambda\subset \bar O_\lambda$ the maximal ideal. Suppose that for each $\lambda$ as above we are given an irreducible representation $\rho_\lambda :\Pi\to GL(n,\bar E_\lambda )$ so that for each $x\in |X|$ the polynomial $\det (1-t\rho_\lambda (F_x))$ is defined over $E$ and independent of $\lambda$. Let $r_\lambda$ denote the reduction of $\rho_\lambda$ modulo $\bar m_\lambda$; it is well-defined as an element of the Grothendieck semigroup of the category of representations of $\Pi$ over $\bar O_\lambda/\bar m_\lambda\,$.

\begin{prop}  \label{p:irred reductions}
$r_\lambda$ is irreducible for almost all $\lambda$.
\end{prop}

Note that if $r_\lambda$ is irreducible then the Brauer obstruction $\beta_\lambda\in\Br (E_\lambda )$ corresponding to $\rho_\lambda$ is zero, so Proposition~\ref{p:irred reductions}  implies \eqref{e:2almost all Brauer}. On the other hand, the proof of Proposition~\ref{p:two reciprocities} shows that  \eqref{e:2almost all Brauer} implies \eqref{e:almost all Brauer}. So it remains to prove Proposition~\ref{p:irred reductions}.

\begin{lem}  \label{l:de Jong}
Suppose that $\dim X=1$. Let $r:\Pi\to GL(n', \bar O_\lambda/\bar m_\lambda )$ be an irreducible representation. Assume\footnote{If $X$ is affine then the assumption that $\lambda$ does not divide $n'$ is not necessary (in \cite{dJ} it is used only in the proof of \cite[Lemma 3.11]{dJ} and only if $X$ is projective).}
that $\lambda$ does not divide $p$, $2$, and $n'$. Then $r$ can be lifted to a representation $\Pi\to GL(n', \bar O_\lambda )$. 
\end{lem}

\begin{proof}  
Without loss of generality, we can assume that $r$ is not induced from any open subgroup of $\Pi$ different from $\Pi$. Then $r$ is \emph{geometrically} irreducible. So the statement follows from Remark 3.6(b) of de Jong's article \cite{dJ} and the fact that his conjecture mentioned in that remark was proved in~\cite{G} for $\ell\ne 2$.
\end{proof}

\begin{proof}[Proof of Proposition~\ref{p:irred reductions}]
By Hilbert irreducibility (e.g., see \cite[Prop.~2.17]{Dr} and \cite[Thm.~2.15(i)]{Dr}) we can assume that $\dim X=1$. 
Let $\bar X$ be the smooth compactification of $X$. Let $D$ be the Swan conductor of $\rho_\lambda\,$; this is an effective divisor on 
$\bar X$ supported on $\bar X\setminus X$, which is independent of $\lambda$ by \cite[Thm.~9.8]{De0}. Fix $x_0\in |X|$; for each $n'<n$ let $A_{n'}$ be the set of isomorphism classes of irreducible continuous representations $\Pi\to GL(n',\BQbar_{\mathfrak l})$ with Swan conductor $\le D$ whose determinant is trivial on $F_{x_0}\in\Pi$ (here $\mathfrak l$ is a non-Archimedean place of $\BQbar$ not dividing $p$, and the only difference between $\BQbar_{\mathfrak l}$ and $\bar E_\lambda$ is that $\BQbar_{\mathfrak l}$ is equipped with a homomorphism $\BQbar\to \BQbar_{\mathfrak l}$); applying the main theorem of L.~Lafforgue's article \cite{La}, we see that the set $A_{n'}$ is finite and independent of $\mathfrak l$. 

We write $\rho$ for the compatible system $\{\rho_\lambda \}$. For $x\in |X|$ set $P_x (\rho ,t):=\det (1-t\cdot\rho_\lambda (F_x))\in E[t]$. For each $\rho'\in A_{n'}$ one has a similar polynomial
$P_x (\rho' ,t)\in \BQbar [t]$. After enlarging $E$, we can assume that $P_x (\rho' ,t)\in E [t]$. In fact, the coefficients of the polynomials $P_x (\rho ,t)$ and $P_x (\rho' ,t)$ are in $O_E [p^{-1}]$, where $O_E\subset E$ is the ring of algebraic integers.

We have to prove that $r_\lambda$ is irreducible for almost all $\lambda$. Assume the contrary. Then there exists a partition
\[
n=n_1+\ldots +n_k, \quad k>1, \;n_i\in\BN
\]
and an infinite set $S$ of places of $E$ such that for all $\lambda\in S$ one has
\[
r=\sum_{i=1}^k r_{\lambda\, ,i}\, ,
\]
where $r_{\lambda\, ,i}$ is an irreducible representation of $\Pi$ over $\bar O_\lambda/\bar m_\lambda$ of dimension $n_i\,$. By Lemma~\ref{l:de Jong}, we can assume that for $\lambda\in S$ the representation $r_{\lambda ,i}$ can be lifted to a representation of $\Pi$ over $\bar O_\lambda\,$. Since the group $\bar O_\lambda^\times$ is divisible, we see that after twisting such a lift by some homomorphism 
\begin{equation}   \label{e:chi_lambda}
\chi_{\lambda,i}:\Gal (\BFbar_p/\BF_p )\to \bar O_\lambda^\times, 
\end{equation}
one gets an element of $A_{n_i}\,$. Since $A_{n_i}$ is finite, we can assume (after shrinking~$S$) that this element is independent of $\lambda$; denote it by $\rho_i\,$. 

Now consider the scheme $Y$ over $O_E [p^{-1}]$ whose $R$-points are $k$-tuples $(z_1,\ldots , z_k)\in (R^\times )^k$ such that the identities
\begin{equation}  \label{e:infinite system}
P_x (\rho ,t)=\prod_{i=1}^k P_x (\rho_i\, ,tz_i^{\deg x}), \quad x\in |X|,
\end{equation}
hold in $R[t]$. For each $\lambda\in S$ one has $Y(\bar O_\lambda/\bar m_\lambda)\ne\emptyset$: indeed, the identities \eqref{e:infinite system} hold if $z_i\in (\bar O_\lambda/\bar m_\lambda)^\times$ is the image of $\chi_{\lambda,i}(F)^{-1}$, where $\chi_{\lambda,i}$ is as in \eqref{e:chi_lambda} and $F\in\Gal (\BFbar_p/\BF_p )$ is the geometric Frobenius. Since $Y$ has finite type and $S$ is infinite it follows that $Y$ has a $\BQbar$-point $(z_1,\ldots ,z_k)\in (\BQbar^\times )^k$. The identities  \eqref{e:infinite system} for these numbers $z_i$ mean that the class of $\rho$ in the Grothendieck group is the sum of the classes of the $z_i$-twists of 
$\rho_i\,$, $1\le i\le k$. This is impossible because $\rho$ was assumed irreducible.
\end{proof}

\bibliographystyle{ams-alpha}

\end{document}